\documentclass[11pt,reqno,sumlimits]{amsart}

\usepackage{amssymb, amscd, amsmath, epsfig, mathtools}
\usepackage[utf8]{inputenc}
\usepackage{amsthm}
\usepackage{enumerate}
\usepackage{xcolor}
\usepackage{scalerel}
\usepackage{soul}
\usepackage{tikz-cd}
\usepackage{esint}

\usepackage[margin=1.0in]{geometry}

\newtheorem{theorem}{Theorem}[section]
\newtheorem*{theorem*}{Theorem}
\newtheorem{definition}{Definition}[section]
\newtheorem{corollary}{Corollary}[section]
\newtheorem{lemma}{Lemma}[section]
\newtheorem{proposition}{Proposition}[section]
\theoremstyle{definition}
\newtheorem{remark}{Remark}[section]

\newcommand{\R}{\mathbb R}

\newcommand{\calC}{\mathcal C}
\newcommand{\calL}{\mathcal L}
\newcommand{\dvol}{ d\text{Vol}_{g}}

\begin{document}

\title[The Yamabe-Type Equations on Closed Manifolds]{Solving the Yamabe-Type Equations on Closed Manifolds by Iteration Schemes}
\author[J. Xu]{Jie Xu}
\address{
Department of Mathematics and Statistics, Boston University, Boston, MA, USA}
\email{xujie@bu.edu}
\address{
Institute for Theoretical Sciences, Westlake University, Hangzhou, Zhejiang, CHINA}
\email{xujie67@westlake.edu.cn}

\date{}							

\maketitle

\begin{abstract} We introduce a double iterative scheme and local variational method to solve the Yamabe-type equation $ - \frac{4(n - 1)}{n - 2}\Delta_{g} u + (S_{g} + \beta ) u = \lambda u^{\frac{n + 2}{n - 2}} $ for some constant $ \beta \leqslant 0 $, locally on Riemannian domain $ (\Omega, g) $ with trivial Dirichlet condition and globally on closed manifolds $ (M, g) $; the dimensions of $ \Omega $ and $ M $ are at least 3. In contrast to the traditional global variational method, these Yamabe-type equations on closed manifolds are analyzed by local analysis and monotone iteration scheme. In particular, we do not need to use the Weyl tensor. In particular, the sign of the first eigenvalue $ \eta_{1} $ of conformal Laplacian $ \Box_{g} u : = - \frac{4(n - 1)}{n - 2}\Delta_{g} u + S_{g} u $ plays the central role. When letting $ \beta \rightarrow 0 $ from the left, we reproof the classical Yamabe problem as a natural consequence of the Yamabe-type equations. The results are classified by the sign of $ \eta_{1} $.
\end{abstract}

\section{Introduction}
In this paper, we introduce a double iterative scheme to solve the Yamabe-type equations, both in a local region as well as in a closed manifold. Precisely speaking, we are seeking for positive, smooth solutions for
\begin{equation}\label{intro:eqn0}
-a\Delta_{g} u + (S_{g} + \beta) u = \lambda u^{p-1} \; {\rm in} \; \Omega, u \equiv 0 \; {\rm on} \; \partial \Omega;
\end{equation}
for some $ \beta < 0 $, and
\begin{equation}\label{intro:eqn00}
-a\Delta_{g} u + (S_{g} + \beta) u = \lambda u^{p-1} \; {\rm in} \; M
\end{equation}
for some $ \beta \leqslant 0 $. Here $ \Omega $ is some subset of $ \R^{n} $ equipped with some Riemannian metric, $ M $ is a closed Riemannian metric, $ n $ is the dimension of the domain, $ a = \frac{4(n -1)}{n - 2}, p = \frac{2n}{n - 2} $, $ S_{g} $ is the scalar curvature with respect to $ g $, $ \lambda \in \R $ is some constant. When $ \beta = 0 $, both (\ref{intro:eqn0}) and (\ref{intro:eqn00}) are Yamabe equations. When the first eigenvalue of the conformal Laplacian $ -a\Delta_{g} + S_{g} $ is positive on closed manifold $ M $, the perturbed Yamabe equation in (\ref{intro:eqn00}) leads to the solvability of the Yamabe equation with some $ \lambda > 0 $, and hence reproof the Yamabe problem with a brand new approach such that the Weyl tensor is not used.

The double iteration method is inspired by \cite{XU2}, in which the Yamabe equation can be solved by an iterative method in a small enough bounded, open subset of $ \R^{n} $ equipped with some metric $ g $ and with nontrivial Dirichlet boundary conditions. This iterative method is developed for hyperbolic operators \cite{Hintz}, \cite{HVasy} and elliptic operators  \cite{XU2}, \cite{Xu} with a long history in PDE theory dating back to \cite{Moser1}, \cite{Moser2}. Instead of a global analysis in terms of calculus of variation, this article applies the local analysis developed especially in \cite{XU2} and \cite{WANG} to the local Yamabe-type equation. The local analysis allows us to obtain the existence of the solutions of the perturbed Yamabe equations globally, and as a natural consequence, to reproof the Yamabe problem in all dimensions by one method, in contrast to previous proofs that separately treated low dimensions and locally conformally flat manifolds. 
\medskip

This type of problems arises from the Yamabe problem. In 1960, Yamabe proposed the following generalization of the classical uniformization theorem for surfaces:
\medskip

\noindent
{\bf The Yamabe Conjecture.} {\it Given a compact Riemannian manifold $ (M, g) $ of dimension $ n \geqslant 3 $, there exists a metric conformal to $ g $ with constant scalar curvature.}       
\medskip
  
Consider the conformal change $ \tilde{g} = e^{2f} g $. Set $ e^{2f} = u^{p-2} $ with $ u>0 $. Then
\begin{equation}\label{intro:eqn1}
S_{\tilde{g}} = u^{1-p}\left(-4 \cdot \frac{n-1}{n-2} \Delta_{g} u + S_{g} u\right).
\end{equation}
\noindent 
We have that $ \tilde{g} = u^{p-2} g $ has constant scalar curvature $ \lambda $ if and only if $ u > 0 $ satisfies the Yamabe equation  
\begin{equation}\label{intro:eqn2}
-a\Delta_{g} u + S_{g} u = \lambda u^{p-1},
\end{equation}
where $\Delta_g = -d^*d$ is negative definite. It was solved in several steps by important work of Yamabe, Trudinger, Aubin and Schoen, with the final cases solved in 1984. A thorough presentation of the historical development of the solution is in \cite{PL}. There are also results for manifolds with boundary \cite{Brendle}, \cite{Escobar2}, \cite{ESC}, \cite{Marques} and open manifolds \cite{Aviles-McOwen}, \cite{Grosse}, \cite{PDEsymposium} with certain restrictions.
\medskip

The first main result for the Yamabe-type equation is on a Riemannian domain $ (\Omega, g) $.
\begin{theorem}\label{intro:thm1}
Let $ (\Omega, g) $ be Riemannian domain in $\R^n$, $ n \geqslant 3 $, with $C^{\infty} $ boundary, and with ${\rm Vol}_g(\Omega)$ and the Euclidean diameter of $\Omega$ sufficiently small. Let $ \beta < 0 $ be any constant. Assume $ S_{g} < 0 $ within the small enough closed domain $ \bar{\Omega} $. Then for any $ \lambda > 0 $, the Dirichlet problem (\ref{intro:eqn0}) has a real, positive, smooth solution $ u \in \calC^{\infty}(\Omega) \cap H_{0}^{1}(\Omega, g) $ in a very small domain $ \Omega $ that vanishes at $ \partial \Omega $.
\end{theorem}
Let $ \eta_{1} $ be the first eigenvalue of the conformal Laplacian $ \Box_{g} : =  -a\Delta_{g} + S_{g} $ on $ (M, g) $. Then globally, the next main result is as follows:
\begin{theorem}\label{intro:thm2} Let $ (M, g) $ be a closed Riemannian manifold with $ \dim M \geqslant 3 $. Then
\begin{enumerate}[(i)]
\item If $ \eta_{1} < 0 $ and $ S_{g} < 0 $ somewhere, (\ref{intro:eqn00}) has a real, positive solution $ u \in \calC^{\infty}(M) $ with some $ \lambda < 0 $ and $ \beta \leqslant 0 $;
\item If $ \eta_{1} > 0 $ and $ S_{g} > 0 $ somewhere, (\ref{intro:eqn00}) has a real, positive solution $ u \in \calC^{\infty}(M) $ with some $ \lambda > 0 $ and $ \beta < 0 $.
\end{enumerate}
\end{theorem}
The result of the Yamabe-type equation mentioned in (i) above is given in Theorem \ref{manifold:thme1}, the result in (ii) above is given in Theorem \ref{manifold:thm3}. As we introduced above, the Yamabe problem is then a natural consequence of  Theorem \ref{intro:thm2}:
\begin{theorem}\label{intro:thm3} Let $ (M, g) $ be a closed Riemannian manifold with $ \dim M \geqslant 3 $. Then
\begin{enumerate}[(i)]
\item If $ \eta_{1} = 0 $, (\ref{intro:eqn2}) has a real, positive solution $ u \in \calC^{\infty}(M) $ with $ \lambda = 0 $;
\item If $ \eta_{1} < 0 $ and $ S_{g} < 0 $ somewhere, (\ref{intro:eqn2}) has a real, positive solution $ u \in \calC^{\infty}(M) $ with some $ \lambda < 0 $;
\item If $ \eta_{1} > 0 $ and $ S_{g} > 0 $ somewhere, (\ref{intro:eqn2}) has a real, positive solution $ u \in \calC^{\infty}(M) $ with some $ \lambda > 0 $.
\end{enumerate}
\end{theorem}
These cover all cases of the Yamabe problem which also provides a classification in terms of the sign of $ \eta_{1} $. Case (i) is a trivial eigenvalue problem; Case (ii) is solved in Theorem \ref{manifold:thm1} and \ref{manifold:thm2} below, also covered by Theorem \ref{manifold:thme1}; Case (iii) is solved in Theorem \ref{manifold:thms} and \ref{manifold:thm5}. Theorem \ref{manifold:thm3}, in which a perturbed Yamabe equation is solved, plays a central role in Case (iii).
\medskip

Historically, the PDE (\ref{intro:eqn2}) is analyzed as the Euler-Lagrange equation of the functional
\begin{equation}\label{intro:eqn3}
\lambda(M) = \inf_{u \neq 0} \int_{u} \frac{\int_{M} a\lvert \nabla_{g} u \rvert^{2} + S_{g} u^{2} \dvol}{\left( \int_{M} u^{p} \dvol \right)^{\frac{2}{p}}} = \inf_{u \neq 0} Q(u)
\end{equation}
and the Yamabe Conjecture reduced to proving a minimizer exists. After results showing the existence the solution provided $ \lambda(M) < \lambda(\mathbb{S}^{n}) $, the problem is equivalent to the nontrivial task of finding a test function $ \phi $ such $ Q(\phi) < \lambda(\mathbb{S}^{n}) $ in (\ref{intro:eqn3}). 
\medskip

There are some critical technical differences between our approach and the classical one. In a small enough domain $ \Omega $ with radius $ r $, a local test function $ \phi(\epsilon, x) $ is chosen to show that the Yamabe quotient $ Q(\phi(\epsilon, x)) $ satisfies
\begin{equation*}
Q(\phi(\epsilon, x)) \leqslant (1 + O(r^{2}))(aT + O(\epsilon)).
\end{equation*}
with best Sobolev constant $ T $ for Euclidean Sobolev embedding $ \calL^{p} \hookrightarrow H^{1} $. Since we want $ \lambda(M) < aT $, the key is to control the smallness of $ O(r^{2}) $. In the classical approach, a local test function $ \phi \in \calC_{c}^{\infty}(\Omega) $ and a particular Riemannian metric within the conformal class were carefully chosen to show that $ Q(\phi) < \lambda(\mathbb{S}^{n}) = aT $ when $ n \geqslant 6 $ and the manifold is not locally conformally flat; Schoen successfully applied the positive mass theorem to the remaining cases. In this article, especially for the cases when $ \lambda > 0 $, we apply a direct analysis, which can be found in Appendix \ref{APP}, to show that a perturbed functional
\begin{equation*}
Q_{\beta}(u) = \frac{\int_{M} a \lvert \nabla_{g} u \rvert^{2} + \left(S_{g} + \beta \right) u^{2} \dvol}{\left( \int_{M} u^{p} \dvol \right)^{\frac{2}{p}}}, \beta < 0.
\end{equation*}
satisfies $ Q_{\beta}(\phi) < aT $ locally in $ \Omega $ for all metric $ g $ with $ n \geqslant 3 $, provided that $ \Omega $ is small enough. The first main result Theorem \ref{intro:thm1} then follows due to this key estimates. It is a special case for the existence of some local solution (\ref{local:eqn10}) of the form $ Lu = b(x) u^{p-1} + f(x, u) $ on $ \Omega $ for linear second order elliptic operator $ L $ with trivial Dirichlet boundary condition. 

We then apply monotone iteration scheme, e.g. \cite{Aviles-McOwen}, \cite{KW}, to show the existence of a positive solution of the perturbed Yamabe equation
\begin{equation}\label{intro:eqn4}
-a\Delta_{g} u_{\beta} + \left(S_{g} + \beta \right) u_{\beta} = \left( \lambda_{\beta} - \kappa \right) u_{\beta}^{p-1} \; {\rm on} \; M.
\end{equation}
Here $ \lambda_{\beta} = \inf_{u \neq 0} Q_{\beta} u $ and $ \kappa > 0 $ is some fixed positive constant. The local analysis and monotone iteration scheme give us the flexibility to choose the coefficient, especially the introduction of $ \kappa $, ahead of the nonlinear term; the introduction of $ \kappa $ cannot be achieved by classical method in which the constant ahead of the nonlinear term is the minimizer of some functional and is determined at the very last step. A limiting argument shows $ u = \lim_{\beta \rightarrow 0} u_{\beta} $ exists for a subsequence of $ \lbrace u_{\beta} \rbrace $ and solves the Yamabe equation (\ref{intro:eqn2}). The inequality $ Q_{\beta}(\phi) < aT $ holds for all $ n \geqslant 3 $, and is independent of the Weyl tensor precisely because of the introduction of $ \beta $. This method bypasses the subcritical exponent analysis of uniform boundedness of the sequence $ \lbrace \varphi_{s} \rbrace $ when $ s \rightarrow p $, where $ \varphi_{s} $ solves $ \Box_{g} \varphi_{s} = \lambda_{s} \varphi_{s}^{s - 1} $ for $ s \in (2, p) $.
\medskip

In particular, we constructed sub-solutions $ u_{-} $ and super-solutions $ u_{+} $ of the Yamabe equation (\ref{intro:eqn2}) or the perturbed Yamabe equation (\ref{intro:eqn4}), then applied monotone iterative method--see e.g. \cite{Aviles-McOwen}, \cite{KW}--to get a solution, this is one layer of iteration. The essential difficulty is to show that $ u_{-} \not\equiv 0 $. Historically the cases when $ \eta_{1} < 0 $ is easy to handle. We applied local analysis in \cite{XU2} to give a different way to construct the sub-solutions/super-solutions, this is the second layer of iteration. For the cases when $ \eta_{1} > 0 $, barely no one constructed the sub-solution/super-solution before, based on the best understanding of us. The full power of this iteration scheme can be seen in solving the boundary Yamabe problem on compact manifolds with boundary, see \cite{XU4}, \cite{XU5}. The size of the local domain we choose is also important: in particular, the scalar curvature at each point tells us the deviation of the volume of small enough geodesic ball centered at the same point from the volume of Euclidean ball with same radius. Weyl tensor cannot provide this type of information.
\medskip

The paper is organized as follows. In \S2, the preliminaries and required results for double iteration scheme are listed. In \S3, we constructed local solutions on a small enough single chart of $ (M, g) $ with Dirichlet boundary conditions for different cases classified by $ \text{sgn}(S_{g}) $ and $ \text{sgn}(\eta_{1}) $: when $ \eta_{1} < 0 $, the local analysis in \cite{XU2} is applied, see Proposition \ref{local:prop1} and Proposition \ref{local:prop2}; when $ \eta_{1} > 0 $, a local calculus of variation is applied to show a local Yamabe-type equation with Dirichlet boundary condition, see Proposition \ref{local:prop3}. When $ \eta_{1} > 0 $, the region in which $ S_{g} < 0 $ plays a central role in the local analysis.

In \S4, sub-solutions $ u_{-} $ and super-solutions $ u_{+} $ are constructed in various cases: (i) Theorem \ref{manifold:thm1} and Theorem \ref{manifold:thm2} are related to the case $ \eta_{1} < 0 $. As discussed before, they are historically easy case, here a new construction of sub- and super- solutions is applied. We modify the arguments in Theorem \ref{manifold:thm1} and Theorem \ref{manifold:thm2} a little bit to get a general result for Yamabe-type equation in Theorem \ref{manifold:thme1}. (ii) Theorem \ref{manifold:thms} and Theorem \ref{manifold:thm5} are related to the case $ \eta_{1} > 0 $. For the case $ \eta_{1} > 0 $, Theorem \ref{manifold:thm3} plays a central role for the existence of solution of a perturbed Yamabe equation. The constructions of sub- and super- solutions rely on the local analysis given in Proposition \ref{local:prop3}. To handle the cases when $ S_{g} \leqslant 0 $ everywhere and $ S_{g} \geqslant 0 $ everywhere, we shows in Theorem \ref{manifold:thm4} and Corollary \ref{manifold:cor4} that they are equivalent to the cases $ S_{g} > 0 $ somewhere, $ \eta_{1} < 0 $ and $ S_{g} < 0 $ somewhere, $ \eta_{1} > 0 $, respectively. Note that the last two results mentioned above relates to the question: whether a function $ f $ can be the prescribed scalar curvature of some metric under conformal change?

\section{The Preliminaries}
In this section, we list definitions and results required for our analysis. 
\medskip

Let $ \Omega $ be a connected, bounded, open subset of $ \R^{n} $ with smooth boundary $ \partial \Omega $ equipped with some Riemannian metric $ g $ that can be extended smoothly to $ \bar{\Omega} $. We call $ (\Omega, g) $ a Riemannian domain. Furthermore, let $ (\bar{\Omega}, g) $ be a compact manifold with smooth boundary extended from $ (\Omega, g) $. Let $ (M, g) $ be a closed manifold with $ \dim M \geqslant 3 $. Let $ (\bar{M}', g) $ be a general compact manifold with interior $ M' $ and smooth boundary $ \partial M' $. Throughout this article, we denote the space of smooth functions with compact support by $ \calC_{c}^{\infty} $, smooth functions by $ \calC^{\infty} $, and continuous functions by $ \calC_{0} $.
\medskip

We define two equivalent versions of the $\calL^p$ norms and two equivalent versions of the Sobolev norms on $ (\Omega, g) $. We also define global $ \calL^{p} $ norms and Sobolev norms on $ (M, g) $.
\begin{definition}\label{pre:def1} Let $ (\Omega, g) $ be a Riemannian domain. Let $ (M, g) $ be a closed  Riemannian $n$-manifold with volume density $ d\omega $ with local expression $\dvol$. Let $u$ be a real valued function. Let $ \langle v,w \rangle_g$ and $ |v|_g = \langle v,v \rangle_g^{1/2} $ denote the inner product and norm  with respect to $g$. 

(i) 
For $1 \leqslant p < \infty $,
\begin{align*}
\mathcal{L}^{p}(\Omega)\ &{\rm is\ the\ completion\ of}\  \left\{ u \in \calC_c^{\infty}(\Omega) : \Vert u\Vert_p^p :=\int_{\Omega} \lvert u \rvert^{p} dx < \infty \right\},\\
\mathcal{L}^{p}(\Omega, g)\ &{\rm is\ the\ completion\ of}\ \left\{ u \in \calC_c^{\infty}(\Omega) : \Vert u\Vert_{p,g}^p :=\int_{\Omega} \left\lvert u \right\rvert^{p} d\text{Vol}_{g} < \infty \right\}, \\
\mathcal{L}^{p}(M, g)\ &{\rm is\ the\ completion\ of}\ \left\{ u \in \calC^{\infty}(M) : \Vert u\Vert_{p,g}^p :=\int_{M} \left\lvert u \right\rvert^{p} d\text{Vol}_{g} < \infty \right\}.
\end{align*}

(ii) For $\nabla u$  the Levi-Civita connection of $g$, 
and for $ u \in \calC^{\infty}(\Omega) $ or $ u \in \calC^{\infty}(M) $,
\begin{equation}\label{pre:eqn1}
\lvert \nabla^{k} u \rvert_g^{2} := (\nabla^{\alpha_{1}} \dotso \nabla^{\alpha_{k}}u)( \nabla_{\alpha_{1}} \dotso \nabla_{\alpha_{k}} u).
\end{equation}
\noindent In particular, $ \lvert \nabla^{0} u \rvert^{2}_g = \lvert u \rvert^{2} $ and $ \lvert \nabla^{1} u \rvert^{2}_g = \lvert \nabla u \rvert_{g}^{2}.$\\

(iii) For $ s \in \mathbb{N}, 1 \leqslant p < \infty $,
\begin{align}\label{pre:eqn2}
W^{s, p}(\Omega) &= \left\{ u \in \mathcal{L}^{p}(\Omega) : \lVert u \rVert_{W^{s,p}(\Omega)}^{p} : = \int_{\Omega} \sum_{j=0}^{s} \left\lvert D^{j}u \right\rvert^{p} dx < \infty \right\}, \\
W^{s, p}(\Omega, g) &= \left\{ u \in \mathcal{L}^{p}(\Omega, g) : \lVert u \rVert_{W^{s, p}(\Omega, g)}^{p} = \sum_{j=0}^{s} \int_{\Omega} \left\lvert \nabla^{j} u \right\rvert^{p}_g d\text{Vol}_{g} < \infty \right\} \nonumber, \\
W^{s, p}(M, g) &= \left\{ u \in \mathcal{L}^{p}(M, g) : \lVert u \rVert_{W^{s, p}(M, g)}^{p} = \sum_{j=0}^{s} \int_{M} \left\lvert \nabla^{j} u \right\rvert^{p}_g d\omega < \infty \right\} \nonumber.
\end{align}
\noindent Here $ \lvert D^{j}u \rvert^{p} := \sum_{\lvert \alpha \rvert = j} \lvert \partial^{\alpha} u \rvert^{p} $ 
in the weak sense. Similarly, $ W_{0}^{s, p}(\Omega) $ is the completion of $ \calC_{c}^{\infty}(\Omega) $ with respect to the 
$ W^{s, p} $-norm.
In particular, $ H^{s}(\Omega) : = W^{s, 2}(\Omega) $ and $ H^{s}(\Omega, g) : = W^{s, 2}(\Omega, g) $, $ H^{s}(M, g) : = W^{s, 2}(M, g) $ are the usual Sobolev spaces. We similarly define $H_{0}^{s}(\Omega), H_{0}^{s}(\Omega,g)$.

(iv) We define the $ W^{s, p} $-type Sobolev space on $ (\bar{M}', g) $ the same as in (iii) when $ s \in \mathbb{N}, 1 \leqslant p < \infty $.
\end{definition}

The elliptic regularity for Riemannian domain $ (\Omega, g) $ and closed manifold $ (M, g) $, and for $ W^{s, p} $-type and $ H^{s} $-type, are listed below, respectively.
\begin{theorem}\label{pre:thm1} Let $ (\Omega, g) $ be a Riemmannian domain. Let $ (M, g) $ be a closed Riemannian manifold and $ (\bar{M}', g) $ be a compact Riemannian manifold with smooth boundary.

(i)\cite[Ch.~5, Thm.~1.3]{T} Let $ L $ be a second order elliptic operator of the form $ Lu = -\Delta_{g} u + Xu $ where $ X $ is a first order differential operator with smooth coefficients on $ \bar{M}' $. For $ f \in \calL^{2}(M', g) $, a solution $ u \in H_{0}^{1}(\Omega, g) $ to $ Lu = f $ in $ M' $ with $ u \equiv 0 $ on $ \partial M' $ belongs to $ H^{2}(M', g) $, and
\begin{equation}\label{pre:eqn3}
\lVert u \rVert_{H^{2}(M', g)} \leqslant C^{*} \left( \lVert f \rVert_{\calL^{2}(M', g)} + \lVert u \rVert_{H^{1}(M', g)} \right).
\end{equation}
$ C^{*} = C^*(L, M', g) $ depends on $ L $ and $ (\bar{M}', g) $. 
\medskip

(ii)\cite[Vol.~3, Ch.~17]{Hor}Let $ -\Delta_{g} $ be the Laplacian with respect to Riemannian metric $ g $ on $ \Omega $, and let $ u \in H_{0}^{1}(\Omega, g) $ be a weak solution of $ -\Delta_{g} u = f $. \\
(Interior Regularity) If $ f \in W^{s, p}(\Omega, g) $, then $ u \in W^{s + 2, p}(\Omega, g) $ . Also, if $ u \in \mathcal{L}^{p}(\Omega, g) $, then
\begin{equation}\label{pre:eqn4}
  \lVert u \rVert_{W^{s + 2, p}(\Omega,g)} \leqslant D_{1} (\lVert -\Delta_{g} u \rVert_{W^{s, p}(\Omega,g)} + \lVert u \rVert_{\mathcal{L}^{p}(\Omega,g)}), 
\end{equation}
for some $D_{1} > 0$ where $ D_{1} $ only depends on $ -\Delta_{g} $, $ \Omega$ and $ \partial \Omega $. \\
(Schauder Estimates) If $ f \in \calC^{s, \alpha}(\Omega) $, then $ u \in \calC^{s + 2, \alpha}(\Omega) $. Also, if $ u \in \calC^{0, \alpha}(\Omega) $, then
\begin{equation}\label{pre:eqn5}
  \lVert u \rVert_{\calC^{s + 2, \alpha}(\Omega)} \leqslant D_{2} (\lVert -\Delta_{g} u \rVert_{\calC^{s, \alpha}(\Omega)} + \lVert u \rVert_{\calC^{0, \alpha}(\Omega)}), 
\end{equation}
for some $D_{2} >0 $ where $ D_{2} $ only depends on $ -\Delta_{g} $, $ \Omega$ and $ \partial \Omega $.
\medskip

(iii) \cite[Thm.~2.4]{PL}Let $ -\Delta_{g} $ be the Laplace-Beltrami operator, and let $ u \in H^{1}(M, g) $ be a weak solution of $ -\Delta_{g} u = f $. If $ f \in W^{s, p}(M, g) $, then $ u \in W^{s + 2, p}(M, g) $. Also, if $ u \in \mathcal{L}^{p}(M, g) $, then above estimates in (\ref{pre:eqn4}) and (\ref{pre:eqn5}) also hold when replacing $ (\Omega, g) $ by $ (M, g) $.
\end{theorem}

Sobolev embeddings, both on $ (\Omega, g) $ and $ (M, g) $, play an important role in this article.
\begin{theorem}\label{pre:thm2}\cite[Ch.~4]{Adams}  (Sobolev Embeddings) 
Let $ \Omega \in \R^{n} $ be a bounded, open set with smooth boundary $ \partial \Omega $ and equipped with any Riemannian metric $ g $.

(i) For $ s \in \mathbb{N} $ and $ 1 \leqslant p \leqslant p' < \infty $ such that
\begin{equation}\label{pre:eqn6}
   \frac{1}{p} - \frac{s}{n} \leqslant \frac{1}{p'},
\end{equation}
\noindent  $ W^{s, p}(\Omega) $ continuously embeds into $ \mathcal{L}^{p'}(\Omega) $ with the following estimates: 
\begin{equation}\label{pre:eqn6a}
\lVert u \rVert_{\calL^{p'}(\Omega, g)} \leqslant K \lVert u \rVert_{W^{s, p}(\Omega, g)}.
\end{equation}

(ii) For $ s \in \mathbb{N} $, $ 1 \leqslant p < \infty $ and $ 0 < \alpha < 1 $ such that
\begin{equation}\label{pre:eqn7}
  \frac{1}{p} - \frac{s}{n} \leqslant -\frac{\alpha}{n},
\end{equation}
Then  $ W^{s, p}(\Omega) $ continuously embeds in the H\"older space $ \calC^{0, \alpha}(\Omega) $ with the following estimates:
\begin{equation}\label{pre:eqn7a}
\lVert u \rVert_{\calC^{0, \alpha}(\Omega)} \leqslant K' \lVert u \rVert_{W^{s, p}(\Omega, g)}.
\end{equation}

(iii) The above results and estimates also hold when applying to the closed manifold $ (M, g) $.
\end{theorem} 

A more precise version of Sobolev embedding is provided by Gagliardo-Nirenberg interpolation inequality.
\begin{proposition}\label{pre:prop1}\cite[Thm.~3.70]{Aubin} Let $ M $ be either $ \R^{n} $, or a compact Riemannian manifold with or without boundary. Let $ q, r,l $ be real numbers  with $ 1 \leqslant q, r,l \leqslant \infty $, and let $ j, m $ be integers with $ 0 \leqslant j < m $.
Define  $\alpha $ by solving
\begin{equation}\label{pre:eqn8}
\frac{1}{l} = \frac{j}{n} + \alpha \left(\frac{1}{r} - \frac{m}{n}\right) + \frac{1- \alpha}{q},
\end{equation}
as long as $l >0.$ If $\alpha \in \left[\frac{j}{m},1\right]$, then there exists a constant $ C_{0} $, depending only on $ n, m , j , q, r,  \alpha $ and on the manifold, such that for all $ u \in \calC_{c}^{\infty}(M) $ (with $ \int_{M} u = 0 $ in the compact case without boundary),
\begin{equation}\label{pre:eqn9}
\lVert \nabla^{j} u \rVert_{\calL^\ell(\Omega,g)} \leqslant C_{0} \lVert \nabla^{m} u \rVert_{\calL^r(\Omega,g)}^{\alpha} \lVert u \rVert_{\calL^q(\Omega,g)}^{1 - \alpha}.
\end{equation}
(If $ r = \frac{n}{m - j} \neq 1 $, then (\ref{pre:eqn8}) is not valid for $ \alpha = 1 $.)
\end{proposition}

To see how different choices of $ \lambda $ affect the behaviors of solutions of Yamabe equation $ -a \Delta_{g} u + S_{g} u = \lambda u^{p-1} $ in $ (\Omega, g) $ and $ u \equiv c > 0 $ on $ \partial \Omega $, we need both strong and weak maximum principles.
\begin{theorem}\label{pre:thm3}
(i) \cite[Cor.~3.2]{GT} (Weak Maximum Principle) Let $ L $ be a second order elliptic operator of the form
\begin{equation*}
Lu = -\sum_{\lvert \alpha \rvert = 2} -a_{\alpha}(x) \partial^{\alpha} u + \sum_{\lvert \beta \rvert = 1} -b_{\beta}(x) \partial^\beta u + c(x) u
\end{equation*}
\noindent where $ a_{\alpha}, b_{\beta}, c \in \calC^{\infty}(\Omega) $ are smooth and bounded real-valued functions on the bounded domain $ \Omega \subset \R^{n} $. Let $ u \in \calC^2(\bar{\Omega}) $. Denote $u^{-} : = \min(u, 0)$, and $ u^{+} : = \max(u, 0) $. Then we have
\begin{equation}\label{pre:eqn10}
\begin{split}
Lu \geqslant 0, c(x) \geqslant 0 & \Rightarrow \inf_{\Omega} u = \inf_{\partial \Omega} u^{-}; \\
Lu \leqslant 0, c(x) \geqslant 0 & \Rightarrow \sup_{\Omega} u = \sup_{\partial \Omega} u^{+}.
\end{split}
\end{equation}

(ii) \cite[Thm.~3.1]{GT} Let $ u $, $ L $ and its coefficients be the same as in (i) above. Then we have
\begin{equation}\label{pre:eqn10a}
\begin{split}
Lu \geqslant 0, c(x) = 0 & \Rightarrow \inf_{\Omega} u = \inf_{\partial \Omega} u; \\
Lu \leqslant 0, c(x) = 0 & \Rightarrow \sup_{\Omega} u = \sup_{\partial \Omega} u.
\end{split}
\end{equation}

(iii) \cite[Thm.~3.5]{GT} (Strong Maximum Principle) Let $ u \in \calC^{2}(\Omega) $. Assume that $\partial \Omega \in \calC^{\infty}$. Let $ L $ be a second order uniformly elliptic operator as above. If $Lu \geqslant 0$,  $ c(x) \geqslant 0 $, and if $ u $ attains a nonpositive minimum over $ \bar{\Omega} $ in an interior point, then $ u $ is constant within $ \Omega $; If $ Lu \leqslant 0 $, $ c(x) \geqslant 0 $, and if $ u $ attains a nonnegative maximum over $ \bar{\Omega} $ in an interior point, then $ u $ is constant within $ \Omega $.

(iv) \cite[Ch.~8]{GT} All weak and strong maximum principles above hold when $ u \in H^{1}(\Omega, g) $ or $ u \in H^{1}(M, g) $, respectively, provided that $ L $ is uniformly elliptic with some other restrictions.
\end{theorem}

Let's denote the conformal Laplacian by
\begin{equation}\label{pre:eqn11}
\Box_{g} u : = -a\Delta_{g} u + S_{g} u, u \in \calC_{c}^{\infty}(\Omega) \; {\rm or} \; u \in \calC^{\infty}(M).
\end{equation}
The next two results from Kazdan and Warner \cite{KW} for closed manifold, and from Escobar \cite{ESC} for compact manifold with boundary below, respectively, show that the sign of $ \lambda $ in the Yamabe equation is determined by the sign of $ \eta_{1} $. Here $ \eta_{1} $ is the first eigenvalue of conformal Laplacian with respect to appropriate boundary conditions, if necessary.
\begin{theorem}\label{pre:thm4}
(i) \cite[Lemma~2.5]{KW} Let $ (M, g) $ be a closed manifold. Let $ \eta_{1} $ be the first nonzero eigenvalue of the conformal Laplacian $ \Box_{g} $. Then the sign of $ \lambda $, the constant scalar curvature with respect to $ \tilde{g} = u^{p-2} g $, must be the same as the sign of $ \eta_{1} $.

(ii) \cite[Lemma~1.1]{ESC} Let $ (\bar{M}', g) $ be a compact manifold with boundary. Then the same conclusion above holds on $ (\bar{M}', g) $ with the boundary condition $ \frac{\partial u}{\partial \nu} + \frac{n - 2}{2} h_{g} u = 0 $, where $ h_{g} $ is the mean curvature with respect to $ \partial M' $.
\end{theorem}
\begin{remark}\label{pre:re1}
Note that $ \eta_{1} $ satisfies the PDE
\begin{equation*}
-a\Delta_{g} \varphi + S_{g} \varphi = \eta_{1} \varphi.
\end{equation*}
It is well-known that the first eigenfunction $ \varphi > 0 $ on closed manifolds $ M $, or on $ \bar{M}' $ on compact manifold with smooth boundary and the same boundary condition in Theorem \ref{pre:thm4}(ii).
\end{remark}

The key step that converts our local solutions on Riemannian domain to manifolds is the application of monotone iteration method, in terms of obtaining solutions of Yamabe equation/perturbed Yamabe equation provided the existences of subsolution and supersolution. Here we cite a result due to Kazdan and Warner, see \cite{KW2} and \cite{KW}.
\begin{theorem}\label{pre:thm5}\cite[Lemma~2.6]{KW} Let $ (M, g) $ be a closed manifold with $ \dim M \geqslant 3 $. Let $ h, H \in \calC^{\infty}(M) $ for some $ p > n = \dim M $. Let $ m > 1 $ be a constant. If there exists functions $ u_{-}, u_{+} \in \calC_{0}(M) \cap H^{1}(M, g) $ such that
\begin{equation}\label{pre:eqn12}
\begin{split}
-a\Delta_{g} u_{-} + hu_{-} & \leqslant Hu_{-}^{m} \; {\rm in} \; (M, g); \\
-a\Delta_{g} u_{+} + hu_{+} & \geqslant Hu_{+}^{m} \; {\rm in} \; (M, g),
\end{split}
\end{equation}
hold weakly, with $ 0 \leqslant u_{-} \leqslant u_{+} $ and $ u_{-} \not\equiv 0 $, then there is a $ u \in W^{2, p}(M, g) $ satisfying
\begin{equation}\label{pre:eqn13}
-a\Delta_{g} u + hu = Hu^{m} \; {\rm in} \; (M, g).
\end{equation}
In particular, $ u \in \calC^{\infty}(M) $.
\end{theorem}
\begin{remark}\label{pre:re2}
(i)In the proof of Kazdan and Warner \cite{KW2}, \cite{KW}, they constructed the monotone iteration steps as
\begin{equation*}
-a\Delta_{g} u_{j + 1} + k u_{j +1} = -f(x, u_{j}) + ku_{j}, j \in \mathbb{Z}_{\geqslant 0}.
\end{equation*}
where
\begin{equation*}
f(x, u) : = hu - Hu^{m}, k = \sup_{x \in M, u_{-}(x) \leqslant u(x) \leqslant u_{+}(x)} \frac{\partial f(x, u)}{\partial u}, k > 0, u_{0} : = u_{+}.
\end{equation*}
The requirement of $ u_{-}, u_{+} \in W^{2, p}(M, g) $ for $ p $ large enough implies that $ u_{-}, u_{+} \in \calC_{0}(M) $, due to Sobolev embedding in Theorem \ref{pre:thm2}(iii). Starting with $ u_{0} = u_{+} $, the function $ f(x, u_{0}) \in \calL^{p}(M, g) $, hence $ u_{1} \in W^{2, p}(M, g) $ by elliptic regularity. Inductively with  $ f(x, u_{j}) \in \calL^{p}(M, g) $, elliptic regularity in Theorem \ref{pre:thm1}(iii) shows that $ u_{j + 1} \in W^{2, p}(M, g) $, hence $ u_{j + 1} \in \calC_{0}(M) $ and thus $ f(x, u_{j + 1}) \in \calL^{p}(M, g) $ by the same Sobolev embedding argument. It follows that Theorem \ref{pre:thm5} still holds if we start with $ u_{-}, u_{+} \in \calC_{0}(M) $. In addition, $ u_{-}, u_{+} \in H^{1}(M, g) $ are required, since then the sub-solution and super-solution at least hold in the weak sense and a maximum principle in weak sense applies. We see this by claiming that $ (-a\Delta_{g} + k) (u - v) \geqslant 0 $ with $ u, v \in H^{1}(M) \cap \calC_{0}(M) $ implies $ u \geqslant v $. We choose $ w = \max \lbrace 0, v - u \rbrace $, then $ w \in H^{1}(M) $ and $ w \geqslant 0 $. Observe also that $ (-a\Delta_{g} + k) w \leqslant 0 $. Thus
\begin{equation*}
0 \geqslant \int_{M} w (-a\Delta_{g} w + kw) d\omega = \int_{M} \left( a \lvert \nabla_{g} w \rvert^{2} + k w^{2} \right) d\omega \geqslant 0 \Rightarrow w = 0.
\end{equation*}
Thus we must have $ u \geqslant v $.
\medskip

(ii)The relation $ u_{-} \leqslant \dotso \leqslant u_{j + 1} \leqslant u_{j} \leqslant \dotso \leqslant u_{+} $ is proved by taking subtraction of two adjacent iterations above, applying mean value theorem and maximum principle in $ H^{1} $-spaces. It then follows from \cite{KW} that $ u \in W^{2, p}(M, g) $ solves (\ref{pre:eqn13}). Therefore $ u \in \calC^{1, \alpha}(M) $ by Sobolev embedding. When $ h, H \in \calC_{0}(M) $, the requirement of $ u_{-} > 0 $ can be replaced by 
\begin{equation*}
u_{-} \geqslant 0 \; {\rm and}\; u_{-} \not\equiv 0.
\end{equation*}
If $ u_{-} $ satisfies the restrictions above, it follows that $ u $ has the same property. Choosing $ M = \max_{x \in M} \lbrace h - Hu^{m-1}, 0 \rbrace $, (\ref{pre:eqn13}) implies that
\begin{equation*}
-a\Delta_{g} u + M u \geqslant - a\Delta_{g} u + \left(h - Hu^{m - 1} \right)u = 0.
\end{equation*}
By strong maximum principle in Theorem \ref{pre:thm3}, we conclude that if $ u(x_{0}) = 0 $ for some $ x_{0} \in M $ then $ u \equiv 0 $, contradicting to $ u \geqslant u_{-} $ and $ u_{-} \not\equiv 0 $. Hence $ u > 0 $ on $ M $.
\medskip

(iii) In later sections, Theorem \ref{pre:thm5} are used by taking $ h = S_{g} $ or $ h = S_{g} + \beta $ with some $ \beta < 0 $, $ H = \lambda $ or $ H = \lambda - \kappa $ with some $ \kappa > 0 $ and $ m = p - 1 = \frac{n +2}{n - 2} $ with $ u_{-}, u_{+} \in \calC_{0}(M) \cap H^{1}(M, g) $.
\end{remark}

\begin{theorem}\label{pre:thm6}\cite[Thm. 7]{LY} Let $ (\bar{M}, g) $ be a compact manifold with smooth boundary, let $ r_{inj} $ be the injectivity radius of $ M $, and let $ h_{g} $ be the minimum of the mean curvature of $ \partial M $. Choose $ K \geq 0$  such that  $ Ric_{g} \geqslant -(n - 1)K $. If $ \lambda_{1} $ is the first nonzero eigenvalue
\begin{equation*}
-\Delta_{g} u = \lambda_{1} u \; {\rm in} \; M', u \equiv 0 \; {\rm on} \; \partial M'.
\end{equation*}
Then
\begin{equation}\label{pre:eqn14}
\lambda_{1} \geqslant \frac{1}{\gamma} \left( \frac{1}{4(n - 1) r_{inj}^{2}} \left( \log \gamma \right)^{2} - (n - 1)K \right),
\end{equation}
where
\begin{equation}\label{pre:eqn15}
\gamma = \max \left\lbrace \exp[{1 + \left( 1 - 4(n - 1)^{2} r_{inj}^{2}K \right)^{\frac{1}{2}}}], \exp[{-2(n - 1)h_{g} r_{inj}} ] \right\rbrace.
\end{equation}
\end{theorem}

\begin{remark}\label{pre:re3}
We apply Theorem \ref{pre:thm6} above for Riemannian domain $ (\Omega, g) $, for which $ (\bar{\Omega}, g) $ is considered to be a compact manifold with boundary. In Riemannian normal coordinates centered at $ P \in \Omega$, $g$ agrees with the Euclidean metric up to terms of order $O(r^2)$. Thus if $\Omega$ is a $g$-geodesic ball of radius $r$, the mean curvature of $\partial\Omega$ is close $(n-1)/r$, the mean curvature of a Euclidean $r$-ball in $\R^n$.  In (\ref{pre:eqn15}), as $r\to 0$,  $K$ can be taken to be unchanged (since $g$ is independent of $r$), $r_{inj}\to 0$, and $h\cdot r_{inj}\to n-1$. Thus $\gamma\to e^2$,  the right hand side of (\ref{pre:eqn14}) goes to infinity as $r\to 0$, and $\lambda_1\to\infty. $ If $\Omega$ is a general, small enough Riemannian domain with a small enough injectivity radius, then $\Omega$ sits inside a $g$-geodesic ball $\Omega''$ of small radius.  By the Rayleigh quotient characterization of $\lambda_1$, we have $\lambda_{1,\Omega''} \leqslant \lambda_{1,\Omega} $.  Thus for all Riemannian domains $(\Omega,g)$, $ \lambda_{1}^{-1} \rightarrow 0 $ as the radius of $\Omega$ goes to zero.
\end{remark}

\section{The Local Analysis of Yamabe-Type Equation}
In this section, we consider the solvability of the Yamabe equation
\begin{equation}\label{local:eqn1}
\begin{split}
-a\Delta_{g} u + \left(S_{g} + \beta \right) u & = \lambda u^{p-1} \; {\rm in} \; (\Omega, g); \\
u & \equiv c \; {\rm on} \; \partial \Omega
\end{split}
\end{equation}
with some constant $ c > 0 $ and $ \beta = 0 $ in Proposition \ref{local:prop1} and \ref{local:prop2}, due to an iterative scheme from \cite{XU2}. We then consider the case $ c = 0 $ and $ \beta < 0 $ in Proposition \ref{local:prop3}, due to a local analysis in calculus of variations. The small enough Riemannian domain $ (\Omega, g) $ is chosen to be a subset of the closed manifold $ (M, g) $. Note that all results are dimensional specific, i.e. different dimensions of manifolds may lead to different choices of $ \lambda $.
\medskip

In this section, we introduce local analysis of Yamabe equation/perturbed Yamabe equation for (i) $ S_{g} < 0 $ everywhere and $ \lambda < 0 $ in Proposition \ref{local:prop1}; (ii) $ S_{g} > 0 $ somewhere and $ \lambda < 0 $ in Proposition \ref{local:prop2}; (iii) $ S_{g} < 0 $ somewhere and $ \lambda > 0 $ in Proposition \ref{local:prop3}. In the proofs of Proposition \ref{local:prop1} and \ref{local:prop2}, we refer to Theorem 2.3, Theorem 2.8 and Theorem 3.1 in \cite{XU2}.
\medskip

The first result is for the case $ S_{g} < 0 $ everywhere and $ \lambda < 0 $. Note that $ S_{g} < 0 $ implies $ \int_{M} S_{g} < 0 $, hence $ \eta_{1} $ in (\ref{pre:eqn11}) is negative. Therefore by Theorem \ref{pre:thm4} we search for the solution of Yamabe equation with $ \lambda < 0 $.
\begin{proposition}\label{local:prop1}\cite{XU2}
Let $ (\Omega, g) $ be Riemannian domain in $\R^n$, $ n \geqslant 3 $, with $C^{\infty} $ boundary, and with ${\rm Vol}_g(\Omega)$ and the Euclidean diameter of $\Omega$ sufficiently small. Assume $ S_{g} < 0 $ on $ \bar{\Omega} $. Then for some $ \lambda < 0 $, (\ref{local:eqn1}) with $ \beta = 0 $ has a real, positive, smooth solution $ u \in \calC^{\infty}(\Omega) $ such that $ \inf_{x \in \bar{\Omega}} u = c > 0 $. 
\end{proposition}
\begin{proof} As in \cite{XU2}, this result is proved by applying the following iterative scheme
\begin{equation}\label{local:eqn2}
\begin{split}
au_{0} - a\Delta_{g} u_{0} & = f_{0} \; {\rm in} \; (\Omega, g), u_{0} \equiv c \; {\rm on} \; \partial \Omega; \\
au_{k} - a\Delta_{g} u_{k} & = a u_{k - 1} - S_{g} u_{k - 1} + \lambda u_{k - 1}^{p-1}  \; {\rm in} \; (\Omega, g), u_{k} \equiv c \; {\rm on} \; \partial \Omega, k \in \mathbb{N}.
\end{split}
\end{equation}
Setting $ \tilde{u}_{k} = u_{k} - c, k \in \mathbb{Z}_{\geqslant 0} $, the above iterations are equivalent to
\begin{equation}\label{local:eqn3}
\begin{split}
a\tilde{u}_{0} - a\Delta_{g} \tilde{u}_{0} & = f_{0} - ac \; {\rm in} \; (\Omega, g), \tilde{u}_{0} \equiv 0 \; {\rm on} \; \partial \Omega; \\
a\tilde{u}_{k} - a\Delta_{g} \tilde{u}_{k} & = a u_{k - 1} - S_{g} u_{k - 1} + \lambda u_{k - 1}^{p-1} - ac  \; {\rm in} \; (\Omega, g), \tilde{u}_{k} \equiv 0 \; {\rm on} \; \partial \Omega, k \in \mathbb{N}.
\end{split}
\end{equation}
The solvability of each linear elliptic PDE in (\ref{local:eqn2}) and (\ref{local:eqn3}) is achieved by Lax-Milgram \cite[Ch.~6]{Lax}. Thus $ u_{k}, \tilde{u}_{k} \in H_{0}^{1}(\Omega, g) \cap H^{2}(\Omega, g), \forall k $, due to elliptic regularity in Theorem \ref{pre:thm1}(i). With appropriate choice of $ f_{0} > 0 $, $ f_{0} \in \calC^{\infty}(M) $, we conclude that $ u_{0}, \tilde{u}_{0} \in \calC^{\infty}(\Omega) $. By maximum principle in Theorem \ref{pre:thm3}
\begin{equation}\label{local:eqn4a}
\lVert u_{0} \rVert_{H^{2}(\Omega, g)} \leqslant 1, u_{0} > 0 \; {\rm in} \; \bar{\Omega}, \lvert u_{0} \rvert \leqslant C_{M_{n}}, \forall x \in \bar{\Omega}.
\end{equation}
The upper bound for $ \lvert u_{0} \rvert $ is obtained by appropriate choice of $ f_{0} $ and repeated use of $ W^{s, p} $-type elliptic regularity in Theorem \ref{pre:thm1}(ii) and Sobolev embedding in Theorem \ref{pre:thm2}(i). We refer \cite[Thm.~3.1]{XU2} for details. Choosing $ \lambda $ here such that
\begin{equation}\label{local:eqn4}
\lvert \lambda \rvert C_{M_{n}}^{p - 2} \leqslant \inf (-S_{g}) \Rightarrow \lvert \lambda \rvert \sup u_{0}^{p-2} \leqslant \inf (-S_{g}) \Rightarrow -S_{g} u_{0} + \lambda u_{0}^{p-1} \geqslant 0.
\end{equation}
With (\ref{local:eqn4}) and the same argument in Theorem 2.3 and Theorem 3.1 of \cite{XU2}, we conclude that
\begin{equation*}
u_{1}, \tilde{u}_{1} \in \calC^{\infty}(\Omega), \lVert u_{1} \rVert_{H^{2}(\Omega, g)} \leqslant 1, u_{1} > 0 \; {\rm in} \; \bar{\Omega}, \lvert u_{1} \rvert \leqslant C_{M_{n}}, \forall x \in \bar{\Omega}.
\end{equation*}
Inductively, we assume 
\begin{equation}\label{local:eqn5a}
u_{k - 1}, \tilde{u}_{k - 1} \in \calC^{\infty}(\Omega), \lVert u_{k-1} \rVert_{H^{2}(\Omega, g)} \leqslant 1, u_{k-1} > 0 \; {\rm in} \; \bar{\Omega}, \lvert u_{k-1} \rvert \leqslant C_{M_{n}}, \forall x \in \bar{\Omega}, k \in \mathbb{N}.
\end{equation}
By choosing $ \lambda $ the same as in (\ref{local:eqn4}), 
\begin{equation}\label{local:eqn5}
\lvert \lambda \rvert C_{M_{n}}^{p - 2} \leqslant \inf (-S_{g}) \Rightarrow \lvert \lambda \rvert \sup u_{k-1}^{p-2} \leqslant \inf (-S_{g}) \Rightarrow -S_{g} u_{k-1} + \lambda u_{k-1}^{p-1} \geqslant 0.
\end{equation}
It follows that
\begin{equation}\label{local:eqn5b}
u_{k}, \tilde{u}_{k} \in \calC^{\infty}(\Omega), \lVert u_{k} \rVert_{H^{2}(\Omega, g)} \leqslant 1, u_{k} > 0 \; {\rm in} \; \bar{\Omega}, \lvert u_{k} \rvert \leqslant C_{M_{n}}, \forall x \in \bar{\Omega}.
\end{equation}
Note that $ C_{M_{n}} $ is uniform in $ k $. In addition, (\ref{local:eqn2}) implies that
\begin{equation*}
au_{k} - a\Delta_{g} u_{k} \geqslant au_{k - 1}
\end{equation*}
based on the choice of $ \lambda $ in (\ref{local:eqn4}). If the infimum of $ u_{k} $ is attained on $ \partial \Omega $, then $ \inf_{x \in \bar{\Omega}} u_{k} = c $. If the infimum of $ u_{k} $ is attained at some interior point $ x_{0} $, then we have $ -a\Delta_{g} u_{k}(x_{0}) \leqslant 0 $. It follows that 
\begin{equation*}
a \inf_{x \in \bar{\Omega}} u_{k} = a u_{k}(x_{0}) \geqslant au_{k}(x_{0})  - a\Delta_{g} u_{k}(x_{0}) \geqslant au_{k - 1}(x_{0}) \geqslant a \inf_{x \in \bar{\Omega}} u_{k-1}.
\end{equation*}
If $ \inf_{x \in \bar{\Omega}} u_{0} = c $, which can be done by setting $ \inf_{x \in \bar{\Omega}} f_{0} = c $, then $ \inf_{x \in \bar{\Omega}} u_{k} = c, \forall k $. Hence
\begin{equation*}
C_{M_{n}} \geqslant c.
\end{equation*}
Therefore the sequence $ \lbrace u_{k} \rbrace $ is uniformly bounded in $ H^{2} $-norm and is convergent to some $ u \in H^{2}(\Omega, g) $ which solves (\ref{local:eqn1}). By Sobolev embedding and elliptic regularity again, it follows that $ u \in \calC^{\infty}(\Omega) $ and $ u \geqslant 0 $. Due to strong maximum principle in Theorem \ref{pre:thm3}(iii) we conclude that $ u > 0 $ since $ u \equiv c > 0 $ on $ \partial \Omega $. Observe that
\begin{equation*}
\lvert \lambda \rvert u_{k}^{p-2} \leqslant -S_{g}, \forall x \in \Omega \Rightarrow -S_{g}u + \lambda u^{p-1} \geqslant 0 \Rightarrow -a\Delta_{g} u \geqslant 0.
\end{equation*}
By Theorem \ref{pre:thm3}(ii) we conclude that $ u $ achieves its minimum at $ \partial \Omega $. Note that in this proof, we need to choose $ \lambda $ small enough twice so that (\ref{local:eqn5b}) holds, uniformly for all $ k $.
\end{proof}
\medskip

Next result corresponds to the case $ S_{g} > 0 $ somewhere in $ M $ and $ \lambda < 0 $, the proof is essentially the same as in \cite{XU2}. In the next proposition, we assume, without loss of generality, that $ S_{g} \leqslant \frac{a}{2} $. This can be done by scaling the metric $ g $. We start with this metric.
\begin{proposition}\label{local:prop2}\cite{XU2}
Let $ (\Omega, g) $ be Riemannian domain in $\R^n$, $ n \geqslant 3 $, with $C^{\infty} $ boundary, and with ${\rm Vol}_g(\Omega)$ and the Euclidean diameter of $\Omega$ sufficiently small. Assume $ S_{g}(x) \leqslant \frac{a}{2}, \forall x \in M $, and $ S_{g} > 0 $ within the small enough closed domain $ \bar{\Omega} $. Then for some $ \lambda < 0 $, (\ref{local:eqn1}) with $ \beta = 0 $ has a real, positive, smooth solution $ u \in \calC^{\infty}(\Omega) $ such that $ \sup_{x \in \bar{\Omega}} u = c > 0 $. 
\end{proposition}
\begin{proof} We apply the same iterative scheme as in (\ref{local:eqn2}) and (\ref{local:eqn3}) as above. In this case we choose $ f_{0} = 0 $ in (\ref{local:eqn2}), thus we have
\begin{equation*}
au_{0} - a\Delta_{g} u_{0} = 0 \; {\rm in} \; (\Omega, g), u \equiv c > 0 \; {\rm on} \; \partial \Omega.
\end{equation*}
It follows immediately from the PDE above that $ -\Delta_{g} u \leqslant 0 $, therefore $ u_{0} \geqslant 0 $ by weak maximum principle in Theorem \ref{pre:thm3}(i). Then by strong maximum principle in Theorem \ref{pre:thm3}(ii), we conclude that $ \sup_{ x \in \bar{\Omega}} u_{0} = c $. For the first iteration $ k = 1 $ in (\ref{local:eqn2}), we choose $ \lambda $ so that
\begin{equation}\label{local:eqn6}
\lvert \lambda \rvert c^{p-2} \leqslant \frac{a}{2} \leqslant a + \inf_{\Omega} (-S_{g})  \Rightarrow \lvert \lambda \rvert \sup u_{0}^{p-2} \leqslant a - S_{g} \Rightarrow au_{0} - S_{g} u_{0} + \lambda u_{0}^{p-1} \geqslant 0.
\end{equation}
By the same argument as in Theorem 2.3, Theorem 2.8, Theorem 3.1 of \cite{XU2} and Proposition \ref{local:prop1} above, we conclude that (\ref{local:eqn4a}) holds with $ 0 \leqslant u_{0} \leqslant c $. Inductively we assume that (\ref{local:eqn5a}) holds with $ 0 \leqslant u_{k - 1} \leqslant c $. observe that since $ S_{g} > 0 $, $ \lambda < 0 $ and $ u_{k - 1} \geqslant 0 $, it follows from the $ kth $ iteration of (\ref{local:eqn2}) that
\begin{equation*}
au_{k} - a\Delta_{g} u_{k} \leqslant a u_{k - 1} - S_{g} u_{k - 1} + \lambda u_{k - 1}^{p-1} \leqslant a u_{k - 1}.
\end{equation*}
If the supremum is attained at the boundary, then we have $ \sup_{ x \in \bar{\Omega}} u_{k} = c $. If the supremum is attained in some interior point $ x_{0} $, at this point we have $ -a\Delta_{g} u_{k}(x_{0}) \geqslant 0 $ hence
\begin{equation*}
a \sup_{x \in \bar{\Omega}} u_{k} = au_{k}(x_{0}) \leqslant au_{k}(x_{0}) -a\Delta_{g} u_{k}(x_{0}) \leqslant au_{k - 1}(x_{0}) \leqslant a \sup_{x \in \bar{\Omega}} u_{k - 1} = ac.
\end{equation*}
Choosing the same $ \lambda $ as in (\ref{local:eqn6}), it follows that
\begin{equation*}
au_{k -1} - S_{g} u_{k-1} + \lambda u_{k-1}^{p-1} \geqslant 0.
\end{equation*}
Therefore (\ref{local:eqn5b}) holds with $ 0 \leqslant u_{k} \leqslant c $.
Due to the same argument as above, we conclude that the smooth sequence $ \lbrace u_{k} \rbrace \subset \calC^{\infty}(\Omega) $ converges to some $ u \in H^{2}(\Omega, g) $. Since $ u \equiv c $ on $ \partial \Omega $, it is not a trivial solution. By bootstrapping strategy with respect to Sobolev inequality (\ref{pre:eqn6a}), (\ref{pre:eqn7a}) and elliptic regularity (\ref{pre:eqn4}), we conclude that $ u \in \calC^{0, \alpha}(\Omega) $ and hence Schauder estimates (\ref{pre:eqn5}) in Theorem \ref{pre:thm1}(iii) says that $ u \in \calC^{\infty}(\Omega) $.  In particular, since each $ u_{k} \in [0, c] $, it follows that $ u \in [0, c] $. 

With the choice of $ \lambda $ we have
\begin{equation*}
S_{g} > 0, -\lambda > 0, u \geqslant 0 \Rightarrow S_{g} u - \lambda u^{p-1} \geqslant 0 \Rightarrow -a\Delta_{g} u \leqslant 0,
\end{equation*}
on $ (\Omega, g) $ and it follows that $ u $ achieves its supremum at $ \partial \Omega $ by Theorem \ref{pre:thm3}(ii).
\end{proof}
\begin{remark}\label{local:re0}
Consider the PDE
\begin{equation*}
-a\Delta_{g} u + S_{g} u - \lambda u^{p-1} = 0 \; {\rm in} \; (\Omega, g).
\end{equation*}
For $ \lambda < 0 $ satisfies Theorem 2.3 of \cite{XU2} and (\ref{local:eqn6}) above, we have
\begin{equation}\label{local:eqn6a}
\lvert \lambda \rvert c^{p-2} \leqslant \frac{a}{2} \Rightarrow -\lambda c^{p-1} \leqslant \frac{a}{2} c \Rightarrow \lambda c^{p-1} \geqslant - \frac{a}{2} c.
\end{equation}
(\ref{local:eqn6a}) we be used later in Section 4.
\end{remark}
\medskip

For the case when $ S_{g} < 0 $ somewhere and $ \lambda > 0 $, we need the following theorem, due to Wang \cite{WANG}, which gives the existence of the positive solution of the following Dirichlet problem
\begin{equation}\label{local:eqn7}
\begin{split}
Lu & : = -\sum_{i, j} \partial_{i} \left (a_{ij}(x) \partial_{j} u \right) = b(x) u^{p- 1} + f(x, u) \; {\rm in} \; \Omega; \\
u & > 0 \; {\rm in} \; \Omega, u = 0 \; {\rm on} \; \partial \Omega.
\end{split}
\end{equation}
Here $ p - 1 = \frac{n+2}{n -2} $ is the critical exponent with respect to the $ H_{0}^{1} $-solutions of (\ref{local:eqn7}). The solution of (\ref{local:eqn7}) is the minimizer of the functional
\begin{equation}\label{local:eqn8}
J(u) = \int_{\Omega} \left( \frac{1}{2} \sum_{i, j} a_{ij}(x) \partial_{i}u \partial_{j} u - \frac{b(x)}{p} u_{+}^{p} - F(x, u) \right) dx,
\end{equation}
where $ u_{+} = \max \lbrace u, 0 \rbrace $ and $ F(x, u) = \int_{0}^{u} f(x, t)dt $. Set
\begin{equation}\label{local:eqn9}
\begin{split}
A(D) & = \text{essinf}_{x \in D} \frac{\det(a_{ij}(x))}{\lvert b(x) \rvert^{n-2}}, \forall D \subset \Omega; \\
T & = \inf_{u \in H_{0}^{1}(\Omega)}  \frac{\int_{\Omega} \lvert Du \rvert^{2} dx}{\left( \int_{\Omega} \lvert u \rvert^{p} dx \right)^{\frac{2}{p}}}; \\
K & = \inf_{u \neq 0} \sup_{t > 0} J(tu), K_{0} = \frac{1}{n} T^{\frac{n}{2}} \left( A(\Omega) \right)^{\frac{1}{2}}; \\
T_{L} & = \inf_{u \neq 0} \frac{\int_{\Omega} a_{ij} \partial_{i} u \partial_{j} u dx}{\left( \int_{\Omega} b(x) \lvert u \rvert^{p+1} dx \right)^{\frac{2}{p}}}.
\end{split}
\end{equation}
\begin{theorem}\label{local:thm1}\cite[Thm.~1.1, Thm.~1.4]{WANG} Let $ \Omega $ be a bounded smooth domain in $ \R^{n}, n \geqslant 3 $. Let $ Lu = -\sum_{i, j} \partial_{i} \left (a_{ij}(x) \partial_{j} u \right) $ be a second order elliptic operator in divergence form. Let ${\rm Vol}_g(\Omega)$ and the diameter of $\Omega$ sufficiently small. Let $ b(x) \neq 0 $ be a nonnegative bounded measurable function. Let $ f(x, u) $ be measurable in $ x $ and continuous in $ u $. Assume
\begin{enumerate}
\item There exist $ c_{1}, c_{2} > 0 $ such that $ c_{1} \lvert \xi \rvert^{2} \leqslant \sum_{i, j} a_{ij}(x) \xi_{i} \xi_{j} \leqslant c_{2} \lvert \xi \rvert^{2}, \forall x \in \Omega, \xi \in \R^{n} $;
\item $ \lim_{u \rightarrow + \infty} \frac{f(x, u)}{u^{p-1}} = 0 $ uniformly for $ x \in \Omega $;
\item $ \lim_{u \rightarrow 0} \frac{f(x, u)}{u} < \lambda_{1} $ uniformly for $ x \in \Omega $, where $ \lambda_{1} $ is the first eigenvalue of $ L $;
\item There exists $ \theta \in (0, \frac{1}{2}), M \geqslant 0, \sigma > 0 $, such that $ F(x, u) = \int_{0}^{u} f(x, t)dt \leqslant \theta u f(x, u) $ for any $ u \geqslant M $, $ x \in \Omega(\sigma) = \lbrace x \in \Omega, 0 \leqslant b(x) \leqslant \sigma \rbrace $.
\end{enumerate}
Furthermore, we assume that $ f(x, u) \geqslant 0 $, $ f(x, u) = 0 $ for $ u \leqslant 0 $. We also assume that $ a_{ij}(x) \in \calC_{0}(\bar{\Omega}) $. If
\begin{equation}\label{local:eqn10}
K < K_{0} \; {\rm or} \; K < \frac{1}{n} T_{L}^{\frac{n}{2}},
\end{equation}
then the Dirichlet problem (\ref{local:eqn7}) possesses a solution $ u $ which satisfies $ J(u) \leqslant K $.
\end{theorem}
We now apply Theorem \ref{local:thm1} to show the following result. Unfortunately we cannot show
\begin{equation*}
-a\Delta_{g} u + S_{g} u = \lambda u^{p-1} \; {\rm in} \; \Omega, u = 0 \; {\rm on} \; \partial \Omega
\end{equation*}
directly. Instead we perturb $ S_{g} $ by $ S_{g} + \beta $ for some $ \beta < 0 $, and show the existence of the positive solution of
\begin{equation*}
-a\Delta_{g} u + (S_{g} + \beta) u = \lambda u^{p-1} \; {\rm in} \; \Omega, u = 0 \; {\rm on} \; \partial \Omega.
\end{equation*}
\begin{proposition}\label{local:prop3}
Let $ (\Omega, g) $ be Riemannian domain in $\R^n$, $ n \geqslant 3 $, with $C^{\infty} $ boundary, and with ${\rm Vol}_g(\Omega)$ and the Euclidean diameter of $\Omega$ sufficiently small. Let $ \beta < 0 $ be any constant. Assume $ S_{g} < 0 $ within the small enough closed domain $ \bar{\Omega} $. Then for any $ \lambda > 0 $, the following Dirichlet problem
\begin{equation}\label{local:eqn11}
-a\Delta_{g} u + \left(S_{g} + \beta \right) u = \lambda u^{p-1} \; {\rm in} \; \Omega, u = 0 \; {\rm on} \; \partial \Omega.
\end{equation}
has a real, positive, smooth solution $ u \in \calC^{\infty}(\Omega) \cap H_{0}^{1}(\Omega, g) $ in a very small domain $ \Omega $ that vanishes at $ \partial \Omega $.
\end{proposition}
\begin{proof}
We show this by showing that (\ref{local:eqn11}) can be converted to a special case in Theorem \ref{local:thm1}, and thus a positive solution is admitted. Without loss of generality, we may assume $ 0 \in \Omega $. Under $ W^{s, p}(\Omega, g) $-norms, (\ref{local:eqn11}) is the Euler-Lagrange equation of the functional
\begin{equation}\label{local:eqn12}
J^{*}(u) = \int_{\Omega} \left( \frac{1}{2} a\sqrt{\det(g)} g^{ij} \partial_{i}u \partial_{j} u - \frac{\sqrt{\det(g)}}{p} \lambda u_{+}^{p} - \int_{0}^{u} \sqrt{\det(g)(x)}(-S_{g}(x) - \beta) t dt \right) dx.
\end{equation}
It is straightforward to check that
\begin{align*}
\frac{d}{dt} \bigg|_{t = 0} J^{*}(u + t v) & = \int_{\Omega} \left( \left(-a\partial_{i}(\sqrt{\det(g)} g^{ij} \partial_{j} u) \right) v - \sqrt{\det(g)} \lambda u_{+}^{p-1} v - \sqrt{\det(g)} (-S_{g} - \beta) u  v \right) dx \\
& = \int_{\Omega} \left( -a \frac{1}{\sqrt{\det(g)}} \left(\partial_{i}(\sqrt{\det(g)}g^{ij} \partial_{j} u ) \right) - \lambda u_{+}^{p-1} + \left(S_{g} + \beta \right) u \right) v \sqrt{\det(g)} dx \\
& = \langle -a\Delta_{g} u - \lambda u_{+}^{p-1} + \left( S_{g} + \beta \right) u, v \rangle_{g}.
\end{align*}
Hence $ u > 0 $ in $ \Omega $ minimizes $ J^{*}(u) $ if and only if (\ref{local:eqn11}) admits a positive solution. Note that $ J^{*} u $ is just a special expression of (\ref{local:eqn8}) with
\begin{equation}\label{local:eqn13}
 a_{ij}(x) = a\sqrt{\det(g)} g^{ij}(x), b(x) = \lambda \sqrt{\det(g)(x)}, f(x, u) = \sqrt{\det(g)(x)} \left(-S_{g}(x) - \beta \right)u.
 \end{equation}
 We check that hypotheses in Theorem \ref{local:thm1} are satisfied.
\medskip

First of all, $ b(x) \geqslant 0 $ on $ \Omega $ clearly. As mentioned in Remark \ref{pre:re3}, we can choose $ \Omega $ small enough so that it is a geodesic ball centered at some point $ p \in M $. At the point $ p $, $ g^{ij} = \delta^{ij} $ and it follows that $ g^{ij} $ is a small purterbation of $ \delta^{ij} $ within small enough $ \Omega $. Thus the first hypothesis in Theorem \ref{local:thm1} holds. With the expressions if $ b $ and $ f $ in (\ref{local:eqn13}), we have
\begin{equation*}
\lim_{u \rightarrow + \infty} \frac{f(x, u)}{u^{p-1}} = \lim_{u \rightarrow + \infty} \sqrt{\det(g)(x)} (-S_{g}(x) + \beta) u^{2 - p} = 0.
\end{equation*}
Hence the second hypothesis is satisfied.
\medskip

Hypothesis (3) is for the positivity of $ J(u) $. In the general case (\ref{local:eqn7}),  we choose $ \varphi > 0 $ be the first eigenfunction of the Dirichlet problem $ L \varphi = \lambda_{1} \varphi $ in $ \Omega $, $ \varphi = 0 $ on $ \partial \Omega $, we observe that if $ u > 0 $ solves (\ref{local:eqn7}), 
\begin{align*}
\lambda_{1} \int_{\Omega} u \varphi dx  & = \int_{\Omega} u \cdot L \varphi dx = \int_{\Omega} Lu \cdot \varphi dx = \int_{\Omega} \left( b(x) u^{p-1} + f(x, u) \right) \varphi \dvol \\
& = \int_{\Omega} \left( b(x) u^{p-1} + \frac{f(x, u)}{u} \cdot u \right) \varphi dx \geqslant \int_{\Omega} \frac{f(x, u)}{u} \cdot u \varphi dx.
\end{align*}
It follows that if $ \lim_{u \rightarrow 0} \frac{f(x, u)}{u} > \lambda_{1} $ then it is impossible to get a positive solution $ u $. In our case for $ J^{*} $, we choose $ \varphi > 0 $ and $ \lambda_{1} > 0 $ be the first eigenfunction and first nonzero eigenvalue of $ -\Delta_{g} $, respectively. It follows from (\ref{local:eqn11}) that
\begin{align*}
a\lambda_{1} \int_{\Omega} u \varphi \dvol  & = \int_{\Omega} u \cdot (-a\Delta_{g} \varphi) \dvol = \int_{\Omega} -a\Delta_{g} u \cdot \varphi dx = \int_{\Omega} \left( \lambda u^{p-1} + (-S_{g} - \beta ) u \right) \varphi \dvol \\
& \geqslant \int_{\Omega} (-S_{g} - \beta) \cdot u \varphi \dvol.
\end{align*}
It follows that Hypothesis (3) holds if $ \sup_{x \in \bar{\Omega}} (-S_{g}) - \beta \leqslant a\lambda_{1} $. For later use, we further require
\begin{equation}\label{local:eqns1}
\frac{a}{n} - \left(\frac{n-2}{2n} + \frac{1}{2} \right) \left( \sup_{x \in \bar{\Omega}} (-S_{g}) - \beta \right) \cdot \lambda_{1}^{-1} > 0.
\end{equation}
by making $ \lambda_{1} $ large enough. Due to the argument of first eigenvalue in Remark \ref{pre:re3}, this can be achieved when $ \Omega $ is small enough. Fix this $ \Omega $ from now on. It follows that for any $ \beta' < 0 $ with $ \beta < \beta' $, the same $ \Omega $ applies.
\medskip

For the last hypothesis, Wang \cite{WANG} mentioned that if $ b(x) > 0 $ a.e. in $ \Omega $ then (4) in Theorem \ref{local:thm1} is not needed. Here our $ b(x) = \sqrt{\det(g)} \lambda > 0 $ as always since in a small geodesic ball, $ \sqrt{\det(g)} \approx 1 + O(\text{diam}(\Omega)^{2}) $ and $ \lambda $ is chosen to be positive.
\medskip

Lastly, we check (\ref{local:eqn10}) by checking $ K < K_{0} $. This can be done by showing that there exists some $ u > 0 $ in $ \Omega $, $ u \equiv 0 $ on $ \partial \Omega $ such that $ sup_{t > 0} J^{*}(tu) < K_{0} $. Assume such an $ u $ does exist, then we have
\begin{equation*}
J^{*}(tu) = \int_{\Omega} \left( t^{2} \frac{1}{2} a\sqrt{\det(g)} g^{ij} \partial_{i} u \partial_{j}u  - t^{p} \frac{\sqrt{\det(g)} \lambda }{p} u^{p} - t^{2} \frac{1}{2} \sqrt{\det(g)} \left(-S_{g}(x) - \beta \right) u^{2} \right)  dx.
\end{equation*}
It follows that $ J^{*}(tu) < 0 $ when $ t $ large enough. Hence $ \sup_{t > 0} J^{*}(tu) \geqslant 0 $ is achieved for some finite $ t_{0} \geqslant 0 $. If $ t_{0} = 0 $, then $ \sup_{t > 0} J^{*}(tu)  = 0 $ and thus $ K \leqslant 0 $, again $ K < K_{0} $ holds trivially. Assume now $ t_{0} > 0 $. In this case, we have
\begin{equation}\label{local:eqn13a}
\begin{split}
& \frac{d}{dt} \bigg|_{t = t_{0}} J^{*}(tu) = 0 \\
& \qquad \Rightarrow  t_{0} \int_{\Omega} a\sqrt{\det(g)} g^{ij} \partial_{i} u \partial_{j}u dx - t_{0}^{p- 1} \int_{\Omega} \sqrt{\det(g)} \lambda u^{p} dx - t_{0} \int_{\Omega} \sqrt{\det(g)} (-S_{g}(x) - \beta ) u^{2} dx = 0.
\end{split}
\end{equation}
Denote
\begin{equation*}
V_{1} =  \int_{\Omega} a\sqrt{\det(g)} g^{ij} \partial_{i} u \partial_{j}u dx, V_{2} = \int_{\Omega} \sqrt{\det(g)} (-S_{g}(x) - \beta) u^{2} dx, W = \left(\int_{\Omega} \sqrt{\det(g)} \lambda u^{p} dx \right)^{\frac{1}{p}}.
\end{equation*}
Since $ t_{0} > 0 $, (\ref{local:eqn13a}) implies
\begin{equation}\label{local:eqn14}
t_{0}^{p-2} = \frac{V_{1} - V_{2}}{W^{p}} \Rightarrow t_{0} = \frac{\left( V_{1} - V_{2} \right)^{\frac{1}{p-2}}}{W^{\frac{p}{p-2}}}.
\end{equation}
By (\ref{local:eqn14}), $ J^{*}(t_{0} u) $ is of the form
\begin{equation*}
J^{*}(t_{0} u) = \frac{\left( V_{1} - V_{2} \right)^{\frac{2}{p-2}}}{2W^{\frac{2p}{p-2}}} V_{1} - \frac{\left( V_{1} - V_{2} \right)^{\frac{p}{p-2}}}{W^{\frac{p^{2}}{p-2}}} \cdot \frac{1}{p} W^{p} - \frac{\left( V_{1} - V_{2} \right)^{\frac{2}{p-2}}}{2W^{\frac{2p}{p-2}}} V_{2}.
\end{equation*}
Note that $ p = \frac{2n}{n - 2} $ hence we have
\begin{equation*}
\frac{2}{p-2} = \frac{n-2}{2}, \frac{p}{p -2} = \frac{n}{2}, p - \frac{p^{2}}{p - 2} = p  \left( 1- \frac{p}{p - 2} \right) = -n \Rightarrow \frac{2p}{p - 2} = n = \frac{p^{2}}{p - 2} - p.
\end{equation*}
In terms of $ n $, $ J^{*}(t_{0}u) $ becomes
\begin{equation}\label{local:eqn15}
\begin{split}
J^{*}(t_{0}u) & = \left( \frac{1}{2} - \frac{1}{p} \right) \frac{\left( V_{1} - V_{2} \right)^{\frac{n}{2}}}{W^{n}} = \frac{1}{n}  \frac{\left( V_{1} - V_{2} \right)^{\frac{n}{2}}}{W^{n}} \\
& =  \frac{1}{n} \left( \frac{\int_{\Omega} a\sqrt{\det(g)} g^{ij} \partial_{i} u \partial_{j} u dx- \int_{\Omega} \sqrt{\det(g)} \left(-S_{g} - \beta \right) u^{2} dx}{\left( \int_{\Omega} \sqrt{\det(g)} \lambda u^{p} dx \right)^{\frac{2}{p}}} \right)^{\frac{n}{2}}.
\end{split}
\end{equation}
The inequality $ K < K_{0} $ reduces to find some $ u \in H_{0}^{1}(\Omega) $, $ u > 0 $ such that (\ref{local:eqn15}) is less than $ K_{0} $. We check the expression of $ K_{0} $ first. by $ a_{ij}, b $ in (\ref{local:eqn13}), we have
\begin{align*}
A(\Omega) & = \text{essinf}_{\Omega} \frac{\det(a_{ij}(x))}{\lvert b(x) \rvert^{n-2}} = \text{essinf}_{\Omega} \frac{\det\left(a \sqrt{\det(g)} g^{ij} \right)}{ \left( \sqrt{\det(g)} \lambda \right)^{n-2}} = \text{essinf}_{\Omega} \frac{\left( \sqrt{\det(g)} \right)^{n} a^{n} \det( g^{ij})}{ \left( \sqrt{\det(g)} \lambda \right)^{n-2}} \\
& = \text{essinf}_{\Omega} a^{n} \lambda^{2 - n} \sqrt{\det(g)}^{2} \det(g^{ij}) = \lambda^{2 - n}a^{n}.
\end{align*}
It follows that
\begin{equation}\label{local:eqn16}
K_{0} = \frac{1}{n} T^{\frac{n}{2}} A(\Omega)^{\frac{1}{2}} = \frac{1}{n} \lambda^{\frac{2- n}{2}}a^{\frac{n}{2}} T^{\frac{n}{2}}.
\end{equation}
Compare expressions in (\ref{local:eqn15}) and (\ref{local:eqn16}), showing $ K < K_{0} $ is equivalent to find some $ u > 0 $ in $ \Omega $ such that
\begin{equation}\label{local:eqn17}
\begin{split}
& J_{0} : = \frac{\int_{\Omega} a\sqrt{\det(g)} g^{ij} \partial_{i} u \partial_{j} u dx- \int_{\Omega} \sqrt{\det(g)} \left(-S_{g} - \beta \right) u^{2} dx}{\left( \int_{\Omega} \sqrt{\det(g)} \lambda u^{p} dx \right)^{\frac{2}{p}}} < \lambda^{\frac{2-n}{n}} aT \\
\Leftrightarrow &  \frac{\int_{\Omega} \sqrt{\det(g)} g^{ij} \partial_{i} u \partial_{j} u dx- \frac{1}{a} \int_{\Omega}  \sqrt{\det(g)} \left(-S_{g} - \beta \right) u^{2} dx}{\left( \int_{\Omega} \sqrt{\det(g)} u^{p} dx \right)^{\frac{2}{p}}} < T.
\end{split}
\end{equation}
\medskip

{\bf Claim:} When both $ \Omega $ and $ \epsilon $ are small enough, the function $  u_{\epsilon, \Omega}(x) = \frac{\varphi_{\Omega}(x)}{( \epsilon + \lvert x \rvert^{2} )^{\frac{n-2}{2}}} $ satisfies
\begin{equation}\label{local:eqn18}
\frac{\lVert \nabla_{g} u_{\epsilon, \Omega} \rVert_{\calL^{2}(\Omega, g)}^{2} + \frac{1}{a} \int_{\Omega} \left(S_{g} + \beta \right) u_{\epsilon, \Omega}^{2} \sqrt{\det(g)} dx}{\lVert u_{\epsilon, \Omega} \rVert_{\calL^{p}(\Omega, g)}^{2}} < T.
\end{equation}
This claim is proved in Appendix \ref{APP}. 
\medskip

Finally we conclude that
\begin{equation*}
K < K_{0}.
\end{equation*}
Note that this result is peculiar for our special PDE in (\ref{local:eqn12}) with small enough $ \Omega $. The regularity of $ u $ follows exactly the same as in [Sec, 1]\cite{Niren3}. We would like to denote here that when $ u \in \calL^{\infty}(\Omega) $ then $ u^{p-1} \in \calL^{p}(\Omega, g) $ for all $ p $, and by bootstrapping method, it follows that $ u \in \calC^{\infty}(\Omega) $, provided that $ \partial \Omega $ is smooth. In addition, $ u \in \calC^{1, \alpha}(\bar{\Omega}) $, the constant $ \alpha $ is determined by the nonlinear term $ u^{p-1} $.
\end{proof}

\begin{remark}\label{local:re1}
The proof of Theorem \ref{APP:thm1} in Appendix \ref{APP} plays a central role for the validity of Proposition \ref{local:prop3}. This proof is essentially due to Brezis and Nirenberg \cite{Niren3}, but they are not the same. The proof in Theorem \ref{APP:thm1} is a sharp estimate in the sense that, for example, if any negative term with order $ \epsilon^{\frac{4 - n}{2}} $ on numerator and any positive term of order $ \epsilon^{\frac{4-n}{2}} $ on denominator is dropped when $ n \geqslant 5 $, then the estimate fails. A detailed analysis with respect to Beta and Gamma functions are required to connect the best Sobolev constant, the area of unit sphere, and $ K_{2} $ defined in Appendix \ref{APP}.
\end{remark}
\begin{remark}\label{local:re2}
The quantity $ Q_{\epsilon, \Omega} $ we estimate in Appendix \ref{APP} is a perturbation of Yamabe quotient, and hence is not conformally invariant. We show that $ Q_{\epsilon, \Omega} < T $ for all $ n \geqslant 3 $, no matter the ambient manifold is locally conformally flat or not.
\end{remark}
\medskip

\section{Yamabe Equation on Closed Manifolds}
In this section, we apply the sub-solution and super-solution technique in Theorem \ref{pre:thm5} and Remark \ref{pre:re2} to solve Yamabe-type equations
\begin{equation}\label{manifold:eqn0}
-a\Delta_{g} u + (S_{g} + \beta)u = \lambda u^{p-1} \; {\rm in} \; M
\end{equation}
for five cases:
\begin{enumerate}[(A).]
\item  $ S_{g} \leqslant 0 $ everywhere with $ \beta = 0 $;
\item  $ S_{g} > 0 $ somewhere and $ \beta = 0 $ when the conformal Laplacian $ \Box_{g} $ admits a negative first eigenvalue $ \eta_{1} $;
\item  $ S_{g} < 0 $ somewhere and $ \beta < 0 $ when the conformal Laplacian $ \Box_{g} $ admits a positive first eigenvalue $ \eta_{1} $;
\item  $ S_{g} \geqslant 0 $ everywhere and $ \beta < 0 $;
\item $ \eta_{1} = 0 $ and $ \beta = 0 $.
\end{enumerate}
When $ \eta_{1} = 0 $ in case (E), we can solve Yamabe problem trivially with $ \lambda = 0 $, this is just an eigenvalue problem. Note that generically zero is not an eigenvalue of conformal Laplacian $ \Box_{g} $, see \cite{GHJL}. 
\medskip

For the rest four cases above, the sub-solutions and super-solutions will be constructed due to Propositions \ref{local:prop1} through \ref{local:prop3}, respectively. By Theorem \ref{pre:thm4}(i), we know that (A) and (B) are equivalent to find a solution of (\ref{intro:eqn2}) with $ \lambda < 0 $. (C) and (D) are equivalent to find a solution of (\ref{manifold:eqn0}) with $ \lambda > 0 $. The classical Yamabe problem is then reproved by letting $ \beta \rightarrow 0^{-} $ as a corollary of the solvability of the perturbed Yamabe equation when the first eigenvalue of conformal Laplacian is positive. In contrast to the classical calculus of variation approach--in which they find the minimizer of (\ref{intro:eqn3}) first, and conclude that the desired $ \lambda $ is automatically determined by some functional with the minimizer of (\ref{intro:eqn3})--this article determines an appropriate choice of $ \lambda $ first, and then conclude that the Yamabe equation admits a real, smooth, positive solution associated with the choice of $ \lambda $.
\medskip

Recall the eigenvalue problem of conformal Laplacian on $ (M, g) $.
\begin{equation}\label{manifold:eqn1}
-a\Delta_{g} \varphi + S_{g} \varphi = \eta_{1} \varphi \; {\rm on} \; (M, g).
\end{equation}
The first nonzero eigenvalue is characterized by
\begin{equation}\label{manifold:eqn2}
\eta_{1} = \inf_{\varphi \neq 0} \frac{\int_{M} \left( a\lvert \nabla_{g} \varphi \rvert^{2} + S_{g} \varphi^{2} \right) d\omega}{\int_{M} \varphi^{2} d\omega} : = \inf_{\varphi \neq 0} L(\varphi).
\end{equation}
Observe that $ L(\varphi) = L(\lvert \varphi \rvert) $, we may characterize $ \eta_{1} $ by nonnegative functions, for details, see e.g. \cite{KW}. By bootstrapping and maximum principle, we conclude that the first eigenfunction $ \varphi > 0 $ is smooth.
\medskip

The first result concerns the case globally when $ S_{g} < 0 $ everywhere. Note that taking $ u = 1 $ in either (\ref{intro:eqn3}) or (\ref{manifold:eqn2}), we have $ \lambda(M) < 0 $, or $ \eta_{1} < 0 $, respectively.
\begin{theorem}\label{manifold:thm1}
Let $ (M, g) $ be a closed manifold with $ n = \dim M \geqslant 3 $. Assume the scalar curvature $ S_{g} $ to be negative everywhere on $ M $. Then there exists some $ \lambda < 0 $ such that the Yamabe equation (\ref{intro:eqn2}) has a real, positive, smooth solution.
\end{theorem}
\begin{proof} Due to Theorem \ref{pre:thm5} and Remark \ref{pre:re2}, we construct a subsolution $ u_{-} \in \calC_{0}(M) \cap H^{1}(M, g) $ and a super-solution $ u_{+} \in \calC_{0}(M) \cap H^{1}(M, g) $.
\medskip

First of all, we choose a finite cover $ \lbrace U_{\alpha}, \phi_{\alpha} \rbrace $ of $ (M, g) $. Pick up a small enough Riemannian domain $ (\Omega, g) \subset (\phi_{\alpha}(U_{\alpha}), g) $. We can definitely take $ \bar{\Omega} \subset \phi_{\alpha}(U_{\alpha}) $ to be a small enough closed ball. Take any constant $ c > 0 $.

We construct the sub-solution first. By Proposition \ref{local:prop1}, there exists some $ \lambda \in (\eta_{1}, 0) $ such that $ u_{1} \in \calC^{\infty}(\Omega) $ solves the Yamabe equation (\ref{local:eqn1}) locally with $ \beta = 0 $ and $ u_{1} \equiv c > 0 $ on $ \partial \Omega $. Fix the $ \lambda $. Extend $ u_{1} $ to $ u_{-} $ by setting
\begin{equation*}
u_{-}(x) = \begin{cases} u_{1}(x), & x \in \Omega; \\ c, & x \in M \backslash \Omega, \end{cases}
\end{equation*}
Clearly $ u_{-} \in \calC_{0} $. Observe that $ u_{-} - c $ is the extension of $ u_{1} - c $ by zero outside $ \Omega $. Due to extension theorem, it follows that $ u_{1} - c \in H_{1}(\phi_{\alpha}(U_{\alpha}), g) $ since $ u_{1} - c \in H_{0}^{1}(\Omega, g) $. To see this, we can approximate by $ u_{1} - c $ by a sequence of smooth functions $ \lbrace v_{k} \rbrace $ with compact support in $ \Omega $, then extend $ v_{k} $ by zero to the $ \phi_{\alpha}(U_{\alpha}) $, then take the limit of $ \lbrace v_{k} \rbrace $ with $ H^{1}(\phi_{\alpha}(U_{\alpha}) , g) $-norm. It follows that the limit $ u_{1} - c \in H^{1}(\phi_{\alpha}(U_{\alpha}), g) $. Since $ M $ is compact, we can extend $ u_{1} - c $ by zero to the whole manifold. Thus we conclude that
\begin{equation*}
u_{-} \in \calC_{0}(M) \cap H^{1}(M, g).
\end{equation*}
In addition, $ u_{-} \geqslant 0 $ and $ u_{-} \not\equiv 0 $. We check that $ u_{-} $ is a sub-solution of Yamabe equation. Within $ \Omega $, $ u_{-} $ gives an equality of (\ref{intro:eqn2}), outside $ \Omega $, we observe that
\begin{equation}\label{manifold:eqn3}
-a\Delta_{g} u_{-} + S_{g} u_{-} - \lambda u_{-}^{p-1} = -a\Delta_{g} c + S_{g} c - \lambda c^{p-1} = c (S_{g} - \lambda c^{p-2}).
\end{equation}
As we mentioned in Proposition \ref{local:prop1}, $ C_{M_{n}} \geqslant c $ and thus
\begin{equation*}
\lvert \lambda \rvert c^{p-2} \leqslant \lvert \lambda \rvert C_{M_{n}}^{p-2} \leqslant \inf (-S_{g}).
\end{equation*}
It follows that $ S_{g} - \lambda c^{p-2} \leqslant 0, \forall x \in M $, hence the left side of (\ref{local:eqn3}) is nonpositive for $ x \in M \backslash \Omega $. In conclusion, $ u_{-} $ is a sub-solution of Yamabe equation in the weak sense.
\medskip

The super-solution is constructed as follows. For the fixed $ \lambda $, pick a constant $ C $ such that
\begin{equation*}
\inf_{x \in M} C \geqslant \max \lbrace \left( \frac{\inf_{x \in M} S_{g}}{\lambda} \right)^{\frac{1}{p - 1}}, \sup_{x \in M} u_{-} \rbrace.
\end{equation*}
Denote
\begin{equation*}
u_{+} : = C.
\end{equation*}
Immediately, $ u_{+} \in \calC^{\infty}(M) $. Since $ \lambda < 0 $, we have
\begin{align*}
& C \geqslant \left( \frac{\inf_{x \in M} S_{g}}{\lambda} \right)^{\frac{1}{p - 2}} \Rightarrow S_{g} \geqslant \lambda C^{p -2} \Rightarrow S_{g} C \geqslant \lambda C^{p-1} \\
\Rightarrow & -a\Delta_{g} C + S_{g} C \geqslant \lambda C^{p-1} \Rightarrow -a\Delta_{g} u_{+} + S_{g} u_{+} \geqslant \lambda u_{+}^{p-1}.
\end{align*}
Thus $ u_{+} $ is a super-solution of the Yamabe equation. Due to the choice of $ C $, we conclude that $ 0 \leqslant u_{-} \leqslant u_{+} $ on $ M $. It then follows from Theorem \ref{pre:thm5} that there exists some real, positive function $ u \in \calC^{\infty}(M) $ that solves (\ref{intro:eqn2}).
\end{proof}
\begin{remark}\label{manifold:re0}
Note that $ u_{1} $ above, which is constructed in Proposition \ref{local:prop1} achieves its minimum $ c $ on $ \partial \Omega $, and hence $ u_{-} = \max \lbrace u_{1}, c \rbrace $. It is easy to see that the maximum of two sub-solutions is again a sub-solution in the weak sense.
\end{remark}
\medskip

Next result deals with the case $ S_{g} > 0 $ somewhere and conformal Laplacian $ \Box_{g} $ has a negative first eigenvalue. In the next theorem, we require $ S_{g}(x) \in [-\frac{a}{4}, \frac{a}{2}], \forall x \in M $. This can be done by scaling the original metric $ g $.
\begin{theorem}\label{manifold:thm2}
Let $ (M, g) $ be a closed manifold with $ n = \dim M \geqslant 3 $. Assume the scalar curvature $ S_{g}(x) \in [-\frac{a}{4}, \frac{a}{2}], x \in M $, which is positive somewhere on $ M $. Let the first eigenvalue of conformal Laplacian $ \Box_{g} $ satisfying $ \eta_{1} < 0 $. Then there exists some $ \lambda < 0 $ such that the Yamabe equation (\ref{intro:eqn2}) has a real, positive, smooth solution.
\end{theorem}
\begin{proof} Since $ \eta_{1} < 0 $, we solve the Yamabe problem for some $ \lambda < 0 $. We construct both the sub-solution and super-solution again here. 
\medskip

We construct the super-solution first. Pick up some $ c > 0 $ which will be determined later. Since $ S_{g} > 0 $ somewhere, we take a small enough Riemannian domain $ (\Omega, g) $ in which $ S_{g} > 0 $ and $ \max_{x \in M} S_{g} $ is achieved in the interior of $ \Omega $. As above, we can take $ \bar{\Omega} \subset \phi_{\alpha}(U_{\alpha}) $ to be a small enough closed ball for some cover $ \lbrace U_{\alpha}, \phi_{\alpha} \rbrace $ of $ (M, g) $. By Proposition \ref{local:prop2}, there exists some $ \lambda < 0 $ when $ \lvert \lambda \rvert $ is small enough such that $ u_{2} \in \calC^{\infty}(\Omega) $ solves (\ref{local:eqn2}) locally in $ (\Omega,g) $ with $ u \equiv c > 0 $ on $ \partial \Omega $. Since $ S_{g} \geqslant - \frac{a}{4} $, we further choose $ \lambda $ such that
\begin{equation}\label{manifold:eqn4}
\frac{a}{4} \leqslant \lvert \lambda \rvert c^{p-2} \leqslant \frac{a}{2} \Rightarrow -\frac{a}{4} \geqslant \lambda c^{p-2} \Rightarrow S_{g} \geqslant \lambda c^{p-2} \Rightarrow S_{g} c \geqslant \lambda c^{p-1}, \forall x \in M.
\end{equation}
Note that the choice of $ \lambda $ in (\ref{manifold:eqn4}) is compliant with the choice of $ \lambda $ in (\ref{local:eqn6}) in Proposition \ref{local:prop2}, due to Remark \ref{local:re2}. Note also that we need $ \lvert \lambda \rvert $ to be small enough so that the requirement in Theorem 2.3 of \cite{XU2} is also satisfied. In addition, let's require $ \lvert \lambda \rvert \leqslant \lvert \eta_{1} \rvert $, where $ \eta_{1} $ is the first eigenvalue of $ \Box_{g} $ on $ M $ in this situation. This can be done by choose an appropriate $ c $.
Setting
\begin{equation*}
u_{+}(x) = \begin{cases} u_{2}(x), & x \in \Omega; \\ c, & x \in M \backslash \Omega, \end{cases}
\end{equation*}
Clearly $ u_{+} > 0 $ on $ M $. By the same argument in Theorem \ref{manifold:thm1}, we conclude that $ u_{+} \in \calC_{0}(M) \cap H^{1}(M, g) $ is a super-solution of Yamabe equation in the weak sense. To see this, we check that within $ \Omega $, $ u_{+} = u_{2} $ and the equality of Yamabe equation holds. When outside $ \Omega $, we have
\begin{equation*}
-a\Delta_{g} u_{+} + S_{g} u_{+} - \lambda u_{+}^{p-1} = S_{g} c - \lambda c^{p-1} \geqslant 0,
\end{equation*}
due to (\ref{manifold:eqn4}).
\medskip

For the sub-solution. We apply the result of eigenvalue problem, which says that there exists some function $ \varphi > 0 $ that solves
\begin{equation*}
-a\Delta_{g} \varphi + S_{g} \varphi = \eta_{1} \varphi \; {\rm on} \; (M, g).
\end{equation*}
Observe that we have chosen $ \lambda $ such that $ \lvert \lambda \rvert \leqslant \lvert \eta_{1} \rvert \Leftrightarrow \eta_{1} \leqslant \lambda \leqslant 0 $. Scaling $ \varphi $ by $ \delta \varphi $ with $ \delta < 1 $ small enough such that
\begin{equation*}
\sup_{x \in M} \delta \varphi \leqslant \min \lbrace \inf_{x \in M} u_{+}, 1 \rbrace.
\end{equation*}
Denote
\begin{equation*}
u_{-} : = \delta \varphi.
\end{equation*}
Observe that $ 0 \leqslant u_{-} \leqslant 1 $ implies that $ 0 \leqslant u_{-}^{p-1} \leqslant u_{-} $ since $ p - 1 = \frac{n+2}{n - 2} > 1 $. Note that $ \lambda < 0 $, we then have
\begin{align*}
& -a\Delta_{g} \varphi + S_{g} \varphi = \eta_{1} \varphi \leqslant \lambda \varphi \Rightarrow -a\Delta_{g} \delta \varphi + S_{g} \delta \varphi \leqslant \lambda \delta \varphi \\
\Rightarrow & -a\Delta_{g} u_{-} + S_{g} u_{-} \leqslant \lambda u_{-} \leqslant \lambda u_{-}^{p-1}.
\end{align*}
Hence $ u_{-} \geqslant 0 $ is a sub-solution of the Yamabe equation and $ u_{-} \not\equiv 0 $. Obviously $ u_{-} \in \calC^{\infty}(M) $. In addition, $ 0 \leqslant u_{-} \leqslant u_{+} $ on $ M $. It then follows from Theorem \ref{pre:thm5} that there exists some real, positive function $ u \in \calC^{\infty}(M) $ that solves (\ref{intro:eqn2}).
\end{proof}
\begin{remark}\label{manifold:re00}
Note that $ u_{2} $ above, which is constructed in Proposition \ref{local:prop2} achieves its maximum $ c $ on $ \partial \Omega $, and hence $ u_{+} = \min \lbrace u_{2}, c \rbrace $. It is easy to see that the minimum of two super-solutions is again a super-solution in the weak sense.
\end{remark}
\begin{remark}\label{manifold:re1}
The case when $ S_{g} \leqslant 0 $ everywhere--slightly different from the hypothesis in Theorem \ref{manifold:thm1}--can be converted to the case when $ S_{g} > 0 $ somewhere and $ \eta_{1} < 0 $. This can be done by applying Corollary \ref{manifold:cor4} below. Note the sign of first eigenvalue of $ \Box_{g} $ is a conformal invariant. 
\end{remark}
\begin{remark}\label{manifold:re2}
Historically the cases when $ \eta_{1} < 0 $ is easy to handle. On one hand, there are many other ways to construct sub-solutions and super-solutions, for instance, see \cite{KW} on closed manifolds and \cite{Aviles-McOwen} on non-compact manifolds. On the other hand, when $ S_{g} \leqslant 0 $ or $ \int_{M} S_{g} < 0 $, the functional (\ref{intro:eqn3}) satisfies $ \lambda(M) < 0 $ and hence $ \lambda(M) < \lambda(\mathbb{S}^{n}) $ holds trivially. Thus by Yamabe, Trudinger and Aubin's result, the Yamabe equation has a positive solution.
\end{remark}
\medskip

We would like to point out that the arguments in Theorem \ref{manifold:thm1} and Theorem \ref{manifold:thm2} work equally well for the perturbed Yamabe equation (\ref{manifold:eqn0}) with any $ \beta < 0 $, provided that $ \eta_{1} < 0 $.
\begin{theorem}\label{manifold:thme1}
Let $ (M, g) $ be a closed manifold with $ n = \dim M \geqslant 3 $. Assume the scalar curvature $ S_{g} < 0 $ somewhere in $ M $. Assume that the first eigenvalue of conformal Laplacian $ \Box_{g} $ satisfying $ \eta_{1} < 0 $. Then there exists some $ \lambda < 0 $ such that the Yamabe equation (\ref{manifold:eqn0}) for any $ \beta \leqslant 0 $ has a real, positive, smooth solution.
\end{theorem}
\begin{proof}
When $ \beta = 0 $, this is covered by Theorem \ref{manifold:thm1} and Theorem \ref{manifold:thm2}. When $ \beta < 0 $, then whether $ S_{g} \leqslant 0 $ everywhere or $ S_{g} > 0 $ somewhere, we can apply the same argument in Proposition \ref{local:prop1} and Proposition \ref{local:prop2} to obtain a solution of
\begin{equation*}
-a\Delta_{g} u + (S_{g} + \beta) u = \lambda u^{p-1} \; {\rm in} \; \Omega, u = c > 0 \; {\rm on} \; \partial \Omega
\end{equation*}
for any $ \beta < 0 $. Extend the $ u $ above by $ c $ to the rest of the manifold, and serves as the sub-solution of (\ref{manifold:eqn0}) when $ S_{g} \leqslant 0 $ everywhere and a super-solution when $ S_{g} > 0 $ somewhere.
When $ S_{g} + \beta \leqslant 0 $ everywhere, we mimic the proof of Theorem \ref{local:prop1} by choosing a large enough constant $ C $ as a super-solution. When $ S_{g} > 0 $ somewhere, we realize that since $ \eta_{1} < 0 $, the first eigenvalue of the operator $ -a\Delta_{g} + (S_{g} + \beta) $ is also negative since $ \beta < 0 $. It follows that we can mimic the proof of Theorem \ref{manifold:thm2} and use an appropriate scaling of the eigenfunction of $ -a\Delta_{g} + (S_{g} + \beta) $ with respect to its first eigenvalue as a sub-solution.

Finally the monotone iteration scheme implies the existence of the positive solution. The regularity is due to standard bootstrapping argument.
\end{proof}

Next we show the case when $ S_{g} < 0 $ somewhere and the conformal Laplacian $ \Box_{g} $ has a positive first eigenvalue. This means we are looking for solutions of Yamabe equation with $ \lambda > 0 $. Let $ \beta < 0 $ be any constant. Denote
\begin{equation}\label{manifold:eqnss1}
\lambda_{\beta} = \inf_{u \neq 0, u \in H^{1}(M)} \left\lbrace \frac{\int_{M} a\lvert \nabla_{g} u \rvert^{2} d\omega + \int_{M} \left( S_{g} + \beta \right) u^{2} d\omega}{ \left( \int_{M} u^{p} d\omega \right)^{\frac{2}{p}}} \right\rbrace.
\end{equation}
Recall that
\begin{equation}\label{manifold:eqnss1a}
\lambda(M) = \inf_{u \neq 0, u \in H^{1}(M)} \left\lbrace \frac{\int_{M} a\lvert \nabla_{g} u \rvert^{2} d\omega + \int_{M} S_{g} u^{2} d\omega}{ \left( \int_{M} u^{p} d\omega \right)^{\frac{2}{p}}} \right\rbrace > 0.
\end{equation}
It is crucial to assume $ S_{g} < 0 $ somewhere in order to construct a sub-solution within this region. We cannot work on the Yamabe equation directly. Instead, we solve
\begin{equation*}
-a\Delta_{g} u + \left( S_{g} + \beta \right) u = \lambda u^{p-1}
\end{equation*}
for some specific choice of $ \lambda $ with any $ \beta < 0 $. Before showing this, we need the following lemma, which indicates that a small perturbation of conformal Laplacian will not change the sign of $ \lambda_{\beta} $ in (\ref{manifold:eqnss1}) provided that $ \lambda(M) > 0 $.
\begin{lemma}\label{manifold:lemma1}
Let $ \eta_{1} > 0 $ be the first eigenvalue of $ \Box_{g} $ on close manifolds $ (M, g) $ with $ \dim M \geqslant 3 $. Let $ \beta < 0 $ be a constant. Then $ \lambda_{\beta} $ defined in (\ref{manifold:eqnss1}) is also positive provided that $ \lvert \beta \rvert $ is small enough.
\end{lemma}
\begin{proof} It is immediate that if $ \eta_{1} > 0 $ then $ \lambda(M) > 0 $. It is due to the characterization of $ \eta_{1} $
\begin{equation*}
\eta_{1} = \inf_{\varphi \neq 0} \frac{\int_{M} \left(a \lvert \nabla_{g} \varphi \rvert^{2} + S_{g} \varphi^{2} \right) d\omega}{\int_{M} \varphi^{2} d\omega}
\end{equation*}
as well as the characterization of $ \lambda(M) $ in (\ref{manifold:eqnss1a}).

Due to the characterization of $ \lambda_{\beta} $ in (\ref{manifold:eqnss1}), we conclude that for each $ \epsilon > 0 $, there exists a function $ u_{0} $ such that
\begin{equation*}
\frac{\int_{M} \left( a\lvert \nabla_{g} u_{0} \rvert^{2} + \left(S_{g} + \beta \right) u_{0}^{2} \right) d\omega}{\left( \int_{M} u_{0}^{p} d\omega\right)^{\frac{2}{p}} } \leqslant \lambda_{ \beta} + \epsilon.
\end{equation*}
It follows that
\begin{align*}
0 < \lambda(M) & \leqslant \frac{\int_{M} \left( a\lvert \nabla_{g} u_{0} \rvert^{2} + S_{g} u_{0}^{2} \right) d\omega}{\left( \int_{M} u_{0}^{p} d\omega\right)^{\frac{2}{p}}} \leqslant \frac{\int_{M} \left( a\lvert \nabla_{g} u_{0} \rvert^{2} + \left(S_{g} + \beta \right) u_{0}^{2} \right) d\omega}{\left( \int_{M} u_{0}^{p} d\omega\right)^{\frac{2}{p}}}  - \beta \frac{\int_{M} u_{0}^{2} d\omega}{\left( \int_{M} u_{0}^{p} d\omega\right)^{\frac{2}{p}}} \\
& \leqslant \lambda_{\beta} - \beta \frac{\left( \int_{M} u_{0}^{p} d\omega\right)^{\frac{2}{p}} \cdot \text{Vol}_{g}(M)^{\frac{p-2}{p}}}{\left( \int_{M} u_{0}^{p} d\omega\right)^{\frac{2}{p}}} + \epsilon = \lambda_{\beta} - \beta \text{Vol}_{g}(M)^{\frac{p-2}{p}} + \epsilon
\end{align*}
Since $ \epsilon $ can be arbitrarily small, it follows that
\begin{equation*}
\lambda_{\beta} \geqslant \lambda(M) + \beta \text{Vol}_{g}(M)^{\frac{p-2}{p}}.
\end{equation*}
Thus $ \lambda_{\beta} > 0 $ if $ \lvert \beta \rvert $ is small enough.
\end{proof}
\medskip

We now show the existence of the solution of the perturbed Yamabe equation with a special choice of constant in front of the nonlinear term $ u^{p-1} $. One of the advantages from the local analysis is the flexibility of the choice of this constant, without knowing $ \lambda(M) < \lambda(\mathbb{S}^{n}) $. We would like to mention here that if we already know $ \lambda(M) < \lambda(\mathbb{S}^{n}) $, then we can apply the analysis developed by Aubin and Trudinger to show the existence result directly, see \cite[\S4]{PL}. With $ \lambda(M) < \lambda(\mathbb{S}^{n}) $, no local analysis developed here is required anymore. 
\begin{theorem}\label{manifold:thm3}
Let $ (M, g) $ be a closed manifold with $ n = \dim M \geqslant 3 $. Let $ S_{g} $ be the scalar curvature with respect to $ g $ which is negative somewhere on $ M $. Assume the first eigenvalue of conformal Laplacian $ \Box_{g} $ satisfying $ \eta_{1} > 0 $. Let $ \beta < 0 $ be a small negative constant such that $ \lambda_{\beta} > 0 $ and $ \eta_{1} + \beta > 0 $, and let $ \kappa > 0 $ small enough so that $ \lambda_{\beta} - \kappa > 0 $. Then the following PDE
\begin{equation}\label{manifold:eqnss2}
-a\Delta_{g} u + \left(S_{g} + \lambda \right)u = \left( \lambda_{\beta} - \kappa \right) u^{p-1} \; {\rm in} \; M
\end{equation}
has a real, positive, smooth solution with $ \lambda_{s} $ given in (\ref{manifold:eqnss1}).
\end{theorem}
\begin{proof} We construct both the sub-solution and super-solution again here. 
\medskip

Let $ \beta < 0 $ be a small enough negative constant such that the first eigenvalue $ \eta_{1} $ of the conformal Laplacian admits a positive solution $ \phi \in \calC^{\infty}(M) $ satisfying
\begin{equation}\label{manifold:eqn5}
-a\Delta_{g} \varphi + \left( S_{g} + \beta \right) \varphi = \eta_{1} \varphi + \beta \varphi \; {\rm in} \; M.
\end{equation}
Note that any scaling $ \theta \varphi $ also solves (\ref{manifold:eqn5}). For the given $ \lambda_{\beta} $ and $ \kappa $, we want
\begin{equation*}
\left( \eta_{1} + \beta \right) \inf_{M} (\theta \varphi) > 2^{p-2} \left(\lambda_{\beta} - \kappa \right) \sup_{M} \left(\theta^{p-1} \varphi^{p-1} \right) \Leftrightarrow \frac{\left(\eta_{1} + \beta \right)}{2^{p-2}\left( \lambda_{\beta} - \kappa \right)} > \theta^{p-2} \frac{\sup_{M} \varphi^{p-1}}{\inf_{M} \varphi}.
\end{equation*}
For fixed $ \eta_{1}, \lambda_{\beta}, \varphi, \kappa, \beta $, this can be done by letting $ \theta $ small enough. We denote $ \phi = \theta \varphi $. It follows that
\begin{equation}\label{manifold:eqn5a}
\begin{split}
-a\Delta_{g} \phi + \left( S_{g} + \beta \right) \phi & = \left( \eta_{1} + \beta \right) \phi \; {\rm in} \; M; \\
\left(\eta_{1} + \beta \right) \inf_{M} \phi & > 2^{p-2} \left( \lambda_{\beta} - \kappa \right) \sup_{M} \phi^{p-1} \geqslant 2^{p-2} \left( \lambda_{\beta} - \kappa \right) \phi^{p-1} > \left( \lambda_{\beta} - \kappa \right) \phi^{p-1} \; {\rm in} \; M.
\end{split}
\end{equation}
Set
\begin{equation}\label{manifold:eqnss3}
\beta'  = \left( \eta_{1} + \beta \right) \sup_{M} \phi - 2^{p-2} \left( \lambda_{\beta} - \kappa \right) \inf_{M} \phi^{p-1} > \left( \eta_{1} + \beta \right)\phi - 2^{p-2} \left( \lambda_{\beta} - \kappa \right) \phi^{p-1} \; \text{pointwise}.
\end{equation}
Thus we have
\begin{equation}\label{manifold:eqn5c}
-a\Delta_{g} \phi + \left(S_{g} + \beta \right) \phi = \left( \eta_{1} + \beta \right) \phi > 2^{p-2} \left( \lambda_{\beta} - \kappa \right) \phi^{p-1} > \left( \lambda_{\beta} - \kappa \right) \phi^{p-1} \; {\rm in} \; M \; {\rm pointwise}.
\end{equation}
\medskip

For sub-solution, we apply Proposition \ref{local:prop3} within the small enough region $ (\Omega, g) $ in which $ S_{g} < 0 $. Since $ S_{g} $ is smooth on $ M $ hence it's bounded below. Thus based on hypothesis (3) in Theorem \ref{local:thm1}, we choose $ \Omega $ small enough in which the first eigenvalue $ \lambda_{1} $ of $ -\Delta_{g} u $ is larger than $ a \left( \sup_{x \in M} (-S_{g}) + \beta \right) $. By the result of (\ref{local:prop3}), we have $ u_{3} \in \calC_{0}(\bar{\Omega})  \cap H_{0}^{1}(\Omega, g) $, $ u_{3} > 0 $ in $ \Omega $ solves (\ref{local:eqn11}) for $ \lambda_{\beta} - \kappa > 0 $, i.e.
\begin{equation}\label{manifold:eqn5b}
\begin{split}
-a\Delta_{g} u_{3} + \left(S_{g} + \beta \right) u_{3} & = \left( \lambda_{\beta} - \kappa \right) u_{3}^{p-1} \; {\rm in} \; \Omega, u_{3} = 0 \; {\rm on} \; \partial \Omega; \\
\Leftrightarrow -a\Delta_{g} \left( \frac{u_{3}}{2} \right) + \left( S_{g} + \beta \right) \frac{u_{3}}{2} & = 2^{p-2} \left( \lambda_{\beta} - \kappa \right) \left( \frac{u_{3}}{2} \right)^{p-1} \; {\rm in} \; \Omega, u_{3} = 0 \; {\rm on} \; \partial \Omega \\
\end{split}
\end{equation}
Define
\begin{equation*}
u_{-}(x) = \begin{cases} u_{3}(x), & x \in \Omega \\ 0, & x \in M \backslash \Omega \end{cases}.
\end{equation*}
By the same argument in Theorem \ref{manifold:thm1}, $ u_{-} \in \calC_{0}(M) \cap H^{1}(M, g) $. We show that $ u_{-} $ is a sub-solution of (\ref{manifold:eqnss2}). In local coordinates, the outward normal derivative $ \nu = -\frac{\nabla u_{3}}{\lvert \nabla u_{3} \rvert} $ along $ \partial \Omega $, thus locally taking some $ \Omega' \supset \Omega $ and any test function $ v \geqslant 0 $ we have
\begin{align*}
& \int_{\Omega'} \left( a\nabla_{g} u_{-} \cdot \nabla_{g} v + \left( S_{g} + \beta \right) u_{-} v - \lambda u_{-}^{p-1} v \right) \dvol \\
& \qquad = \int_{\Omega} \left(-a\Delta_{g} u_{3} + \left( S_{g} + \beta \right) u_{3} -  \lambda u_{3}^{p-1} \right) v \dvol + \int_{\partial \Omega} \frac{\partial u_{3}}{\partial \nu} v dS \\
& \qquad = \int_{\partial \Omega} \left(\nabla u_{3} \cdot \nu \right) v dS = \int_{\partial \Omega} \nabla u_{3} \cdot \left( -\frac{\nabla u_{3}}{\lvert \nabla u_{3} \rvert} \right) v dS = \int_{\partial \Omega} - \lvert \nabla u_{3} \rvert v dS \leqslant 0.
\end{align*}
This argument holds on $ M $ also since $ u_{-} $ is trivial outside $ \Omega $. Thus we conclude
\begin{equation*}
-a\Delta_{g} u_{-} + \left( S_{g} + \beta \right) u_{-} - \left( \lambda_{\beta} - \kappa \right) u_{-}^{p-1} \leqslant 0 \; \text{weakly in} \; M.
\end{equation*}
\medskip

We now construct the super-solution. Pick up $ \gamma \ll 1 $ such that
\begin{equation}\label{manifold:eqns0}
0 < 20 \left( \lambda_{\beta} - \kappa \right) \gamma + 2\gamma \cdot \left( \sup_{M} \lvert S_{g} \rvert + \lvert \beta \rvert \right) \gamma < \frac{\beta'}{2}, 31 \left( \lambda_{\beta} - \kappa \right) (\phi + \gamma)^{p-2} \gamma < \frac{\beta'}{2}. 
\end{equation}
This is dimensional specific. Set
\begin{align*}
V & = \lbrace x \in \Omega: u_{3}(x) > \phi(x) \rbrace, V' = \lbrace x \in \Omega: u_{3}(x) < \phi(x) \rbrace, D_{0} = \lbrace x \in \Omega : u_{3}(x) = \phi(x) \rbrace, \\
D_{0}' & = \lbrace x \in \Omega : \text{dist}(x, D_{0}) < \delta \rbrace, D_{0}'' = \left\lbrace x \in \Omega : \text{dist}(x, D_{0}) > \frac{\delta}{2} \right\rbrace, \\
\partial D_{0}' & = \lbrace x \in \Omega : \text{dist}(x, D_{0}) = \delta \rbrace, \partial D_{0}'' = \left\lbrace x \in \Omega : \text{dist}(x, D_{0}) = \frac{\delta}{2} \right\rbrace.
\end{align*}
The constant $ \delta > 0 $ is chosen so that
\begin{equation*}
\sup_{x \in D_{0}'} \lvert u_{3}(x) - \phi(x) \rvert < \gamma.
\end{equation*}
If $ \phi \geqslant u_{3} $ pointwise, then $ \phi $ is a super-solution. If not, a good candidate of super-solution will be $ \max \lbrace u_{3}, \phi \rbrace $ in $ \Omega $ and $ \phi $ outside $ \Omega $, this is an $ H^{1} \cap \calC_{0} $-function. Let $ \nu $ be the outward normal derivative of $ \partial V $ along $ D_{0} $. If $ \frac{\partial u_{3}}{\partial \nu} = \frac{\partial \phi}{\partial \nu} $ on $ D_{0} $ then the super-solution has been constructed. However, this is in general not the case. If not, then $ \frac{\partial u_{3} - \partial \phi}{\partial \nu} \neq 0 $, which follows that $ 0 $ is a regular point of the function $ u_{3} - \phi $ and hence $ D_{0} $ is a smooth submanifold of $ \Omega $. Define
\begin{equation}\label{manifold:eqn6}
\Omega_{1} = V \cap D_{0}'', \Omega_{2} = V' \cap D_{0}'', \Omega_{3} = D_{0}'.
\end{equation}
Clearly $ \Omega_{i}, i = 1, 2, 3 $ are open sets, $ \bigcup_{i} \Omega_{i} = \Omega $ and $ \Omega_{1} \cap \Omega_{2} = \emptyset $. Note that $ D_{0}, D_{0}' $ will never intersect $ \partial \Omega $. Without loss of generality, we may assume that all $ \Omega_{i}, i = 1, 2, 3 $ are connected, since if not, then we do local analysis in all components where $ u_{3} - \phi $ changes sign and combine results together. For any $ 0 < \epsilon \ll \frac{\delta}{4} $, set
\begin{align*}
D_{1} & = \lbrace x \in V: \text{dist}(x, \partial D_{0}') < \epsilon \rbrace, D_{1}' = \lbrace x \in V': \text{dist}(x, \partial D_{0}') < \epsilon \rbrace, \\
D_{2} & = \lbrace x \in V: \text{dist}(x, \partial D_{0}'') < \epsilon \rbrace, D_{2}' = \lbrace x \in V': \text{dist}(x, \partial D_{0}'') < \epsilon \rbrace, \\
E_{1} & = \partial D_{1} \cap \Omega_{1} \cap \Omega_{3}, E_{2} = \partial D_{2} \cap \Omega_{1}, E_{1}' = \partial D_{1}' \cap \Omega_{2} \cap \Omega_{3}, E_{2}' = \partial D_{2}' \cap \Omega_{2}; \\
F_{1} & = \left(\Omega_{1} \cap \Omega_{3} \right) \backslash \left( D_{1} \cup D_{2} \right), F_{2} =  \left(\Omega_{2} \cap \Omega_{3} \right) \backslash \left( D_{1}' \cup D_{2}' \right).
\end{align*}
In conclusion, $ D_{1}, D_{2}, E_{1}, E_{2}, F_{1} $ are subsets of $ V $, in which $ u_{3} > \phi $; $ D_{1}', D_{2}', E_{1}', E_{2}', F_{2} $ are subsets of $ V' $, in which $ u_{3} < \phi $. Note that $ \partial F_{1} = E_{1} \cup E_{2} $, $ \partial F_{2} = E_{1}' \cup E_{2}' $. We now construct appropriate partition of unity subordinate to $ \Omega_{i}, i = 1, 2, 3 $. Choose a local coordinate system $ \lbrace x_{i} \rbrace $ on $ \Omega $, we consider the PDE
\begin{equation}\label{manifold:eqns1}
\begin{split}
&  -\sum_{i, j} g^{ij} (u_{3} - \phi - \gamma) \frac{\partial^{2} v_{1}}{\partial x_{i} \partial x_{j}} - 2 \sum_{i, j} g^{ij} \left( \frac{\partial u_{3}}{\partial x_{j}} - \frac{\partial \phi}{\partial x_{j}} \right) \frac{\partial v_{1}}{\partial x_{i}} \\
& \qquad - \sum_{i, j, k} (u_{3} - \phi - \gamma) g^{ij} \Gamma_{ij}^{k} \frac{\partial v_{1}}{\partial x_{k}}  = 0 \; {\rm in} \; F_{1}; \\
& v_{1} = 0 \; {\rm on} \; E_{2}; v_{1}  = 1 \; {\rm on} \; E_{1}.
\end{split}
\end{equation}
This seemingly complicated PDE (\ref{manifold:eqns1}) is a linear, second order elliptic PDE. For each fixed $ \epsilon > 0 $, this differential operator is uniformly elliptic. Since two parts of boundaries do not intersect, (\ref{manifold:eqns1}) has a unique smooth solution $ v_{1} $ due to \cite[Vol.~3]{Hor}. Since the zeroth order term of this differential operator is zero, it follows from maximum principle that both maximum and minimum of $ v_{1} $ are achieved at $ \partial F_{1} $, it follows that
\begin{equation*}
0 \leqslant v_{1} \leqslant 1 \; {\rm in} \; F_{1}.
\end{equation*}
Let $ \psi $ be the standard mollifier defined on $ B_{0}(1) $. Note that $ \psi $ can be defined as a radial function in terms of $ \lvert x \rvert $. Scaling this we get
\begin{equation*}
\psi_{\epsilon}(x) = \frac{1}{\epsilon^{n}} \varphi\left( \frac{x}{\epsilon} \right), x \in B_{0}(\epsilon).
\end{equation*}
Define
\begin{equation}\label{manifold:eqns2}
\begin{split}
v_{1}'' & : = \min \lbrace 1, \max \lbrace v_{1}, 0 \rbrace \rbrace, v_{1}'(x) : = v_{1}'' * \psi_{\epsilon}(x), \forall x \in \overline{\Omega_{1} \cap \Omega_{3}}; \\
\chi_{1}(x) & : = \begin{cases} 1, & x \in \Omega_{1} \backslash \bar{\Omega}_{3} \\ v_{1}'(x), & x \in \overline{\Omega_{1} \cap \Omega_{3}} \\ 0, & x \in \Omega \backslash \bar{\Omega}_{1} \end{cases}. 
\end{split}
\end{equation}
Since $ v_{1} $ is only nontrivial within $ F_{1} $, we conclude that $ \chi_{1}(x) = 0 $ for all $ x \in \partial D_{0}'' \cap \bar{\Omega}_{1} $ since $ E_{2} $ has an $ \epsilon $-gap away from $ \partial D_{0}'' $ due to construction. By the same argument, $ \chi_{1}(x) = 1 $ for all $ x \in \Omega_{1} \backslash \Omega_{3} $. In particular, $ \chi_{1}(x) = 1 $ on $ \partial D_{0}' \cap \Omega_{1} $. Thus $ \chi_{1} $ is continuous. Actually $ \chi_{1} $ is smooth. For instantce, we pick up a point $ x \in \partial D_{0}'' \cap \bar{\Omega}_{1} $, then we have
\begin{equation*} 
\partial^{\alpha} v_{1}'(x) = \int_{B_{x}(\epsilon)} v_{1}''(y) \partial^{\alpha} \psi_{\epsilon}(x - y) dy \equiv 0
\end{equation*}
since $ v_{1}''(y) \equiv 0 $ on $ B_{x}(\epsilon) $ due to the construction. The smoothness at $ \partial D_{0}' \cap \Omega_{1} $ is almost the same. Lastly, $ \chi_{1} \in [0, 1] $. For example, we pick up a point $ x \in \Omega_{1} \cap \Omega_{3} $, and have
\begin{align*}
1 - \chi_{1}(x) & = 1 - v_{1}'(x) = \int_{B_{0}(\epsilon)} 1 \cdot \psi_{\epsilon}(y) dy - \int_{B_{0}(\epsilon)} v_{1}''(x - y) \psi_{\epsilon}(y) dy \\
& = \int_{B_{0}(\epsilon)} (1 - v_{1}''(x - y)) \cdot \psi_{\epsilon}(y) dy \geqslant 0.
\end{align*}
The last inequality is immediate from the construction of $ v_{1}'' $. By a similar argument, we can see that $ \chi_{1}(x) \geqslant 0 $. We then consider the PDE
\begin{equation}\label{manifold:eqns3}
\begin{split}
&  -\sum_{i, j} g^{ij} \frac{\partial^{2} v_{2}}{\partial x_{i} \partial x_{j}} - \sum_{i, j, k} g^{ij} \Gamma_{ij}^{k} \frac{\partial v_{2}}{\partial x_{k}}  = 0 \; {\rm in} \; F_{2}; \\
& v_{2} = 0 \; {\rm on} \; E_{2}'; v_{2}  = 1 \; {\rm on} \; E_{1}'.
\end{split}
\end{equation}
By the same argument as above, we conclude that (\ref{manifold:eqns3}) has a unique smooth solution $ v_{2} \in [0, 1] $ in $ F_{2} $. Define
\begin{equation}\label{manifold:eqns4}
\begin{split}
v_{2}'' & : = \min \lbrace 1, \max \lbrace v_{2}, 0 \rbrace \rbrace, v_{2}'(x) : = v_{2}'' * \psi_{\epsilon}(x), \forall x \in \overline{\Omega_{2} \cap \Omega_{3}}; \\
\chi_{2}(x) & : = \begin{cases} 1, & x \in \Omega_{2} \backslash \bar{\Omega}_{3} \\ v_{1}'(x), & x \in \overline{\Omega_{2} \cap \Omega_{3}} \\ 0, & x \in \Omega \backslash \bar{\Omega}_{2} \end{cases}. 
\end{split}
\end{equation}
In the same manner, we conclude that $ \chi_{2} \in [0, 1] $ is smooth on $ \Omega $. Lastly we define
\begin{equation}\label{manifold:eqns5}
\chi_{3}(x) = 1 - \chi_{1}(x) - \chi_{2}(x), \forall x \in \Omega.
\end{equation}
Observe that the smooth function $ \chi_{3} \equiv 1 $ in $ \Omega_{3} \backslash \left(\overline{\Omega_{1} \cup \Omega_{2}} \right) $, and vanishes outside $ \Omega_{3} $. We use $ \chi_{1}, \chi_{2}, \chi_{3} $ defined in (\ref{manifold:eqns2}), (\ref{manifold:eqns4}) and (\ref{manifold:eqns5}) as a smooth partition of unity subordinate to $ \Omega_{i}, i = 1, 2, 3 $. Define
\begin{equation}\label{manifold:eqn6a}
\bar{u} = \chi_{1} u_{3} + \chi_{2} \phi + \chi_{3} \left( \phi + \gamma \right).
\end{equation}
We show that $ \bar{u} \in \calC^{\infty}(\Omega) $ is a classical super-solution of Yamabe equation in $ \Omega $. This verification is classified in five different cases. Note that
\begin{equation}\label{manifold:eqn6b}
\Omega = \bigcup_{i = 1}^{3} \Omega_{i} = \left( \Omega_{1} \backslash \bar{\Omega}_{3} \right) \cup \left( \Omega_{2} \backslash \bar{\Omega}_{3} \right) \cup \left( \Omega_{3} \backslash \left( \overline{\Omega_{1} \cup \Omega_{2}} \right) \right) \cup \left( \Omega_{1} \cap \Omega_{3} \right) \cup \left( \Omega_{2} \cup \Omega_{3} \right).
\end{equation}
In first set of (\ref{manifold:eqn6b}), we have
\begin{equation*}
\bar{u} = \chi_{1} u_{3} = u_{3} \; {\rm in} \; \Omega_{1} \backslash \bar{\Omega}_{3} \Rightarrow -a\Delta_{g} \bar{u} + \left( S_{g} + \beta \right) \bar{u} - \left( \lambda_{\beta} - \kappa \right) \left( \bar{u} \right)^{p-1} = 0 \; {\rm in} \; \Omega_{1} \backslash \bar{\Omega}_{3}.
\end{equation*}
due to (\ref{manifold:eqn5b}). Similarly, in the second set of (\ref{manifold:eqn6b}), we have
\begin{equation*}
\bar{u} = \chi_{2} \phi = \phi \; {\rm in} \; \Omega_{2} \backslash \bar{\Omega}_{3} \Rightarrow -a\Delta_{g} \bar{u} + \left( S_{g} + \beta \right) \bar{u} - \left( \lambda_{\beta} - \kappa \right) \left( \bar{u} \right)^{p-1} > 0 \; {\rm in} \; \Omega_{2} \backslash \bar{\Omega}_{3}.
\end{equation*}
due to (\ref{manifold:eqn5c}). In the third set of (\ref{manifold:eqn6b}), we have
\begin{align*}
\bar{u} & = \chi_{3} \left( \phi + \gamma \right) = \phi + \gamma \; {\rm in} \; \Omega_{3} \backslash \left( \Omega_{1} \cap \Omega_{2} \right) \\
\Rightarrow & -a\Delta_{g} \bar{u} + \left( S_{g} + \beta \right) \bar{u} - \left( \lambda_{\beta} - \kappa \right) \left( \bar{u} \right)^{p-1} = -a \Delta_{g} \phi + \left( S_{g} + \beta \right) \left( \phi + \gamma \right) - \left( \lambda_{\beta} - \kappa \right) \left( \phi + \gamma \right)^{p-1} \\
& \geqslant -\left( S_{g} + \beta \right) \phi + \left( \eta_{1} + \beta \right) \phi + \left( S_{g} + \beta \right) \left( \phi + \gamma \right) - 2^{p-2} \left( \lambda_{\beta} - \kappa \right) \phi^{p-1} - 2^{p-2} \left( \lambda_{\beta} - \kappa \right) \gamma^{p-1} \\
& \geqslant \beta' + \left( S_{g} + \beta \right) \gamma - 16 \left( \lambda_{\beta} - \kappa \right) \gamma > 0; \\
\Rightarrow & -a\Delta_{g} \bar{u} + \left( S_{g} + \beta \right) \bar{u} - \left( \lambda_{\beta} - \kappa \right) \left( \bar{u} \right)^{p-1} > 0 \; {\rm in} \; \Omega_{3} \backslash \left( \overline{\Omega_{1} \cup \Omega_{2}} \right).
\end{align*}
The last line is due to the choices of $ \beta', \gamma $ in (\ref{manifold:eqns0}) and the observation that $ 2^{p-2} = 2^{\frac{4}{n - 2}} \leqslant 16, \forall n \geqslant 3 $.

In the fourth set of (\ref{manifold:eqn6b}), we have
\begin{equation*}
\bar{u} = \chi_{1} u_{3} + \chi_{3} \left( \phi + \gamma \right) \; {\rm in} \; \Omega_{1} \cap \Omega_{3}.
\end{equation*}
Note that in $ \Omega_{1} \cap \Omega_{3} \subset \Omega $, $ S_{g} < 0 $, $ \lambda > 0 $, note that $ \phi > u_{3} - \gamma $ and $ u_{3} > \phi $ in $ \Omega_{1} \cap \Omega_{3} $, hence we have
\begin{align*}
\bar{u} & = \chi_{1} u_{3} + \chi_{3} \left( \phi + \gamma \right) < \chi_{1} u_{3} + \chi_{3} \left( u_{3} + \gamma \right) =u_{3} + \chi_{3} \gamma; \\
\Rightarrow \left( S_{g} + \beta \right) \bar{u} &  >  \left( S_{g} + \beta \right) u_{3} +  \left( S_{g} + \beta \right) \chi_{3} \gamma; \\
- \left( \lambda_{\beta} - \kappa \right) \left( \bar{u} \right)^{p - 1} & \geqslant - \left( \lambda_{\beta} - \kappa \right) \left( u_{3} + \chi_{3} \gamma \right)^{p-1} \geqslant - \left( \lambda_{\beta} - \kappa \right) u_{3}^{p-1} - \left(2^{p-1} - 1 \right) \left( \lambda_{\beta} - \kappa \right) u_{3}^{p-2} \chi_{3} \gamma \\
& \geqslant - \left( \lambda_{\beta} - \kappa \right) u_{3}^{p-1} - \left(2^{p-1} - 1 \right) \left( \lambda_{\beta} - \kappa \right) (\phi + \gamma)^{p-2} \chi_{3} \gamma \\
& \geqslant - \left( \lambda_{\beta} - \kappa \right) u_{3}^{p-1} - 31 \left( \lambda_{\beta} - \kappa \right) (\phi + \gamma)^{p-2} \chi_{3} \gamma.
\end{align*}
Here we use the inequality $ (a + b)^{p-1} \leqslant a^{p-1} + \left(2^{p-1} - 1 \right) a^{p-2} b $ when $ 0 \leqslant b \leqslant a $. Here $ a = u_{3}, b = \chi_{3} \gamma $ in $ \Omega_{1} \cap \Omega_{3} $. The constant $ 31 $ is the largest possible value of $ \left(2^{p-1} - 1 \right) $ for all $ n \geqslant 3 $.

The hard part is $ -a\Delta_{g} \bar{u} $. Define the commutator between $ -a\Delta_{g} $ and any smooth function $ f $ by
\begin{equation}\label{manifold:eqn6c}
[-a\Delta_{g}, f]u = -a\Delta_{g} (fu) - f\left(-a\Delta_{g} u\right) = -a\Delta_{g} (fu) + f\left(a\Delta_{g} u \right).
\end{equation}
Immediately, $ [-a\Delta_{g}, 1] = 0 $. In local coordinates, we see that
\begin{align*}
-a\Delta_{g} \bar{u} & = -a\Delta_{g} \left( \chi_{1} u_{3} \right) -a\Delta_{g} \left( \chi_{3} \left( \phi + \gamma \right) \right) \\
& = \chi_{1} \left( -a\Delta_{g} u_{3} \right) + [-a\Delta_{g}, \chi_{1} ] u_{3} + \chi_{3} \left( -a\Delta_{g} (\phi + \gamma) \right) + [-a\Delta_{g}, \chi_{3}] \left( \phi + \gamma \right) \\
& : = V_{1} + V_{2} + W_{1} + W_{2}.
\end{align*}
For $ V_{1} + W_{1} $, we use $ u_{3} < \phi + \gamma $ in $ \Omega_{1} \cap \Omega_{3} $ and have
\begin{align*}
V_{1} + W_{1} & = \chi_{1} \left( -a\Delta_{g} u_{3} \right) + \chi_{3} \left( -a\Delta_{g} \phi \right) \\
& = \chi_{1} \left(-\left( S_{g} + \beta \right) u_{3} + \left( \lambda_{\beta} - \kappa \right) u_{3}^{p-1} \right) + \chi_{3} \left( -\left( S_{g} + \beta \right) \phi + \left( \eta_{1} + \beta \right) \phi \right) \\
& \geqslant -\chi_{1} \left( S_{g} + \beta \right) u_{3} + \chi_{1} \left( \lambda_{\beta} - \kappa \right) u_{3}^{p-1} - \chi_{3} \left( S_{g} + \beta \right) u_{3} + \chi_{3} \left( \lambda_{\beta} - \kappa \right) u_{3}^{p-1} \\
& \qquad + \chi_{3} \left( \left( S_{g} + \beta \right) \gamma + \left( \eta_{1} + \beta \right) \phi - \left( \lambda_{\beta} - \kappa \right) u_{3}^{p-1} \right) \\
& \geqslant -\left( S_{g} + \beta \right) u_{3} + \left( \lambda_{\beta} - \kappa \right) u_{3}^{p-1} \\
& \qquad + \chi_{3} \left(\left( S_{g} + \beta \right) \gamma + \left( \eta_{1} + \beta \right) \phi - \left( \lambda_{\beta} - \kappa \right) (\phi + \gamma)^{p-1} \right) \\
& \geqslant -\left( S_{g} + \beta \right) u_{3} + \left( \lambda_{\beta} - \kappa \right) u_{3}^{p-1} \\
& \qquad + \chi_{3} \left(\left( S_{g} + \beta \right) \gamma + \left( \eta_{1} + \beta \right) \phi - 2^{p-2}\left( \lambda_{\beta} - \kappa \right) \phi^{p-1} - 2^{p-2} \left( \lambda_{\beta} - \kappa \right) \gamma^{p-1} \right) \\
& \geqslant -\left( S_{g} + \beta \right) u_{3} + \left( \lambda_{\beta} - \kappa \right) u_{3}^{p-1} \\
& \qquad + \chi_{3} \left(\left( S_{g} + \beta \right) \gamma + \left( \eta_{1} + \beta \right) \phi - 2^{p-2}\left( \lambda_{\beta} - \kappa \right) \phi^{p-1} - 16 \left( \lambda_{\beta} - \kappa \right) \gamma \right).
\end{align*}
To check $ V_{2} + W_{2} $, we recall that $ \chi_{1} + \chi_{3} = 1 $ in $ \Omega_{1} \cap \Omega_{3} $, hence
\begin{equation*}
0 = [-a\Delta_{g}, 1] = [-a\Delta_{g}, \chi_{1}] + [-a\Delta_{g}, \chi_{3}] \Rightarrow [-a\Delta_{g}, \chi_{1}] = -[-a\Delta_{g}, \chi_{3}].
\end{equation*}
It follows that
\begin{align*}
V_{2} + W_{2} & = [-a\Delta_{g}, \chi_{1} ] u_{3}  + [-a\Delta_{g}, \chi_{3}] \left( \phi + \gamma \right) = [-a\Delta_{g}, \chi_{1}] \left( u_{3} - \phi - \gamma \right) \\
& = \chi_{1} \left(a\Delta_{g} \right) \left( u_{3}- \phi - \gamma \right) -a\Delta_{g} \left(\chi_{1} \left( u_{3} - \phi - \gamma \right) \right) \\
& = -\sum_{i, j} g^{ij} (u_{3} - \phi - \gamma) \frac{\partial^{2} \chi_{1}}{\partial x_{i} \partial x_{j}} - \sum_{i, j} 2g^{ij} \left( \frac{\partial u_{3}}{\partial x_{j}} - \frac{\partial \phi}{\partial x_{j}} \right) \frac{\partial \chi_{1}}{\partial x_{i}} \\
& \qquad - \sum_{i, j, k} g^{ij} \Gamma_{ij}^{k} (u_{3} - \phi - \gamma) \frac{\partial \chi_{1}}{\partial x_{k}}.
\end{align*}
Next we show that
\begin{equation}\label{manifold:eqn6d}
V_{2} + W_{2} \geqslant -L \epsilon'
\end{equation}
in $ \Omega_{1} \cap \Omega_{3} $ for some constant $ L > 0 $ which is independent of $ \epsilon' $. Here $ \epsilon' > 0 $ can be arbitrarily small. We verify (\ref{manifold:eqn6d}) in three types of points. Set
\begin{equation*}
G_{1} : = \lbrace x \in \Omega : \text{dist}(x, E_{1}) \leqslant \epsilon \rbrace, G_{2} : = \lbrace x \in \Omega: \text{dist}(x, E_{2}) \leqslant \epsilon \rbrace.
\end{equation*}
For points $ F_{1} \backslash \left(G_{1} \cup G_{2} \right)  $, i.e. points in $ \Omega_{1} \cap \Omega_{3} $ that is at least distance $ \epsilon $ from either $ E_{1} $ or $ E_{2} $, we have
\begin{equation*}
\chi_{1}(x) = \int_{B_{0}(\epsilon)} v_{1}(x - y) \psi_{\epsilon}(y) dy.
\end{equation*}
Since $ v_{1} \in \calC^{\infty} $, it follows that
\begin{equation*}
\left\lvert \partial^{\alpha} \chi_{1}(x) - \partial^{\alpha} v_{1}(x) \right\rvert \leqslant \epsilon', \forall \lvert \alpha \rvert \leqslant 2.
\end{equation*}
Here $ \epsilon' \xrightarrow{\epsilon \rightarrow 0} 0 $. Hence in this case,
\begin{align*}
V_{2} + W_{2} & \geqslant -\sum_{i, j} g^{ij} (u_{3} - \phi - \gamma) \frac{\partial^{2} v_{1}}{\partial x_{i} \partial x_{j}} - \sum_{i, j} 2g^{ij} \left( \frac{\partial u_{3}}{\partial x_{j}} - \frac{\partial \phi}{\partial x_{j}} \right) \frac{\partial v_{1}}{\partial x_{i}} \\
& \qquad - \sum_{i, j, k} g^{ij} \Gamma_{ij}^{k} (u_{3} - \phi - \gamma) \frac{\partial v_{1}}{\partial x_{k}} - L\epsilon' \\
& = -L\epsilon'.
\end{align*}
Here $ L $ depends on $ u_{3}, \phi, g^{ij}, \Gamma_{ij}^{k} $ and is independent of $ \epsilon' $. When $ x $ in $ \Omega_{1} \cap \Omega_{3} \cap G_{1} $ or $ \Omega_{1} \cap \Omega_{3} \cap G_{2} $, we conclude that
\begin{equation*}
\partial^{\alpha} \chi_{1} = o(1), \forall \alpha, \forall x \in \Omega_{1} \cap \Omega_{3} \cap \left( G_{1} \cup G_{2} \right).
\end{equation*}
This is due to the smoothness of $ \chi_{1} $ and $ \partial^{\alpha} \chi_{1} = 0 $ on $ \partial D_{0}' \cap \Omega_{1} $ and $ \partial D_{0}'' \cap \bar{\Omega}_{1} $. To be precise, when $ x \in \Omega_{1} \cap \Omega_{3} \cap G_{1} $, $ u_{3} - \phi - \gamma \rightarrow 0 $ when $ \epsilon \rightarrow 0 $; furthermore, the first order derivative satisfies
\begin{align*}
\frac{\partial \chi_{1}}{\partial x_{j}} = \int_{B_{x}(\epsilon)} v_{1}'(y) \frac{\partial \psi_{\epsilon}(x - y)}{\partial x_{j}} dy = \int_{B_{x} \cap \Omega_{1} \cap \Omega_{3} \cap G_{1}} \frac{\partial v_{1}}{\partial y_{j}} \psi_{\epsilon}(x - y) dy.
\end{align*}
The boundary terms on $ \partial D_{1} $ are killed since on $ \partial D_{1} $ we have $ v_{1} = 1 $. The second order derivative is of the form
\begin{align*}
\frac{\partial^{2} \chi_{1}}{\partial x_{i}x_{j}} & = \int_{B_{x}(\epsilon)} v_{1}'(y) \frac{\partial^{2} \psi_{\epsilon}(x - y)}{\partial x_{i}x_{j}} dy = \int_{B_{x}(\epsilon)} v_{1}'(y) (-1)^{2} \frac{\partial^{2} \psi_{\epsilon}(x - y)}{\partial y_{i}y_{j}} dy \\
& =  \int_{B_{x}(\epsilon) \cap \Omega_{1} \cap \Omega_{3}} \frac{\partial^{2} v_{1}(y)}{\partial y_{i}y_{j}} \psi_{\epsilon}(x - y) dy - \int_{\partial \left(B_{x}(\epsilon) \cap \partial D_{1} \right)} \frac{\partial v_{1}}{\partial x_{i}} \nu_{j} \psi_{\epsilon}(x - y) dS
\end{align*}
The first term above can be controlled by the PDE (\ref{manifold:eqns1}); for the the second term above, we observe that $ u_{3} - \phi - \gamma = O(\epsilon) $ and $ \frac{\partial v_{1}}{\partial x_{i}} = O(\gamma) $ due to the construction of PDE, thus by the fact that $ \epsilon \ll \gamma $,
\begin{equation*}
-g^{ij} (u_{3} - \phi - \gamma) \int_{\partial \left(B_{x}(\epsilon) \cap \partial D_{1} \right)} \frac{\partial v_{1}}{\partial x_{i}} \nu_{j} \psi_{\epsilon}(x - y) dS \rightarrow 0
\end{equation*}
as $ \epsilon \rightarrow 0 $.
When $ x \in \Omega_{1} \cap \Omega_{3} \cap G_{2} $, the first order derivatives are controlled exactly as just discussed. Note that $ u_{3} - \phi - \gamma < 0 $ here, the second order derivative will have
\begin{align*}
\frac{\partial^{2} \chi_{1}}{\partial x_{i}x_{j}} & = \int_{B_{x}(\epsilon)} v_{1}'(y) \frac{\partial^{2} \psi_{\epsilon}(x - y)}{\partial x_{i}x_{j}} dy = \int_{B_{x}(\epsilon)} v_{1}'(y) (-1)^{2} \frac{\partial^{2} \psi_{\epsilon}(x - y)}{\partial y_{i}y_{j}} dy \\
& =  \int_{B_{x}(\epsilon) \cap \Omega_{1} \cap \Omega_{3}} \frac{\partial^{2} v_{1}(y)}{\partial y_{i}y_{j}} \psi_{\epsilon}(x - y) dy - \int_{\partial \left(B_{x}(\epsilon) \cap \partial D_{2} \right)} \frac{\partial v_{1}}{\partial x_{i}} \nu_{j} \psi_{\epsilon}(x - y) dS \\
& \geqslant  \int_{B_{x}(\epsilon) \cap \Omega_{1} \cap \Omega_{3}} \frac{\partial^{2} v_{1}(y)}{\partial y_{i}y_{j}} \psi_{\epsilon}(x - y) dy.
\end{align*}
The last inequality is due to the fact that $ v_{1} $ is descending to zero, therefore the derivative is negative when very close to $ \partial D_{2} $. It follows that the boundary part above is positive. Within the $ 2\epsilon $-strip of $ \partial D_{0}' $ and $ \partial D_{0}'' $, we have
\begin{align*}
V_{2} + W_{2} & = -\sum_{i, j} g^{ij} (u_{3} - \phi - \gamma) \frac{\partial^{2} \chi_{1}}{\partial x_{i} \partial x_{j}} - \sum_{i, j} 2g^{ij} \left( \frac{\partial u_{3}}{\partial x_{j}} - \frac{\partial \phi}{\partial x_{j}} \right) \frac{\partial \chi_{1}}{\partial x_{i}} \\
& \qquad - \sum_{i, j, k} g^{ij} \Gamma_{ij}^{k} (u_{3} - \phi - \gamma) \frac{\partial \chi_{1}}{\partial x_{k}} = o(1) \geqslant -L\epsilon'.
\end{align*}
Therefore (\ref{manifold:eqn6d}) holds. It follows that
\begin{equation}\label{manifold:eqns7}
\begin{split}
-a\Delta_{g} \bar{u} & = V_{1} + V_{2} + W_{1} + W_{2} \geqslant -\left( S_{g} + \beta \right) u_{3} + \left( \lambda_{\beta} - \kappa \right) u_{3}^{p-1} \\
& \qquad + \chi_{3} \left(\left( S_{g} + \beta \right) \gamma + \left( \eta_{1} + \beta \right) \phi - 2^{p-2}\left( \lambda_{\beta} - \kappa \right) \phi^{p-1} - 16 \left( \lambda_{\beta} - \kappa \right) \gamma \right)- L\epsilon'.
\end{split}
\end{equation}
Combining all estimates together, we have
\begin{align*}
& -a\Delta_{g} \bar{u} + \left( S_{g} + \beta \right) \bar{u} - \left( \lambda_{\beta} - \kappa \right) \bar{u}^{p-1} \geqslant -\left( S_{g} + \beta \right) u_{3} + \left( \lambda_{\beta} - \kappa \right) u_{3}^{p-1} \\
& \qquad + \chi_{3} \left( \left( S_{g} + \beta \right) \gamma + \left( \eta_{1} + \beta \right) \phi - 2^{p-2}\left( \lambda_{\beta} - \kappa \right) \phi^{p-1} - 16 \left( \lambda_{\beta} - \kappa \right) \gamma \right)- L\epsilon' \\
& \qquad \qquad + \left( S_{g} + \beta \right) u_{3} + \left( S_{g} + \beta \right) \chi_{3} \gamma - \left( \lambda_{\beta} - \kappa \right) u_{3}^{p-1} - 31 \left( \lambda_{\beta} - \kappa \right) (\phi + \gamma)^{p-2} \chi_{3} \gamma \\
& \qquad = \chi_{3} \left( \left( \eta_{1} + \beta \right) \phi - 2^{p-2}\left( \lambda_{\beta} - \kappa \right) \phi^{p-1} - 16 \left( \lambda_{\beta} - \kappa \right) \gamma + 2 \left( S_{g} + \beta \right) \gamma - 31 \left( \lambda_{\beta} - \kappa \right) (\phi + \gamma)^{p-2} \gamma \right) \\
& \qquad \qquad - L\epsilon' \\
& \geqslant \chi_{3} \left( \beta' -16 \left( \lambda_{\beta} - \kappa \right) \gamma  + 2 \left( S_{g} + \beta \right) \gamma - 31 \left( \lambda_{\beta} - \kappa \right) (\phi + \gamma)^{p-2} \gamma \right) - L\epsilon'.
\end{align*}
Due to the choice of $ \gamma $ in (\ref{manifold:eqns0}) and the fact that $ \epsilon' $ can be arbitrarily small by making $ \epsilon $ arbitrarily small, we conclude that
\begin{equation*}
-a\Delta_{g} \bar{u} + \left( S_{g} + \beta \right) \bar{u} - \left( \lambda_{\beta} - \kappa \right) \bar{u}^{p-1} \geqslant 0 \; {\rm in} \; \Omega_{1} \cap \Omega_{3}.
\end{equation*}
 In the fifth set of (\ref{manifold:eqn6b}), we have
\begin{equation*}
\bar{u} = \chi_{2} \phi + \chi_{3} \left( \phi + \gamma \right).
\end{equation*}
Repeat the same argument above in replacing $ u_{3} $ by $ \phi $, we conclude that $ \bar{u} $ is a super-solution in $ \Omega_{2} \cap \Omega_{3} $. The argument is essentially the same thus we just sketch here. In $ \Omega_{2} \cap \Omega_{3} $,
\begin{align*}
\bar{u} & = \phi + \chi_{3} \gamma \leqslant \phi + \gamma; \\
\left( S_{g} + \beta \right) \bar{u} & \geqslant \left( S_{g} + \beta \right) \phi + \left( S_{g} + \beta \right) \gamma; \\
-\left( \lambda_{\beta} - \kappa \right) \left(\bar{u} \right)^{p-1} & \geqslant - 2^{p-2} \left( \lambda_{\beta} - \kappa \right) \phi^{p-1} - 16 \left( \lambda_{\beta} - \kappa \right) \gamma.
\end{align*}
The Laplacian part can be estimated as
\begin{align*}
-a\Delta_{g} \bar{u} & = -a\Delta_{g} \left( \chi_{2} \phi \right) -a\Delta_{g} \left( \chi_{3} \left( \phi + \gamma \right) \right) \\
& = \chi_{2} \left( -a\Delta_{g} \phi \right) + [-a\Delta_{g}, \chi_{2} ] \phi + \chi_{3} \left( -a\Delta_{g} \left(\phi + \gamma \right) \right) + [-a\Delta_{g}, \chi_{3}] \left( \phi + \gamma \right) \\
& \geqslant -a\Delta_{g} \phi + [-a\Delta_{g}, \chi_{2}] \left( -\gamma \right) = -a\Delta_{g} \phi - a\Delta_{g} \left( -\gamma \chi_{2} \right).
\end{align*}
We observe that $ - a\Delta_{g} \left( -\gamma \chi_{2} \right) \geqslant -L \epsilon' $ with arbitrarily small $ \epsilon' $ by the same reasoning above, due to the construction in (\ref{manifold:eqns3}) and (\ref{manifold:eqns4}). The rest are the same as above, hence we conclude that
\begin{equation*}
-a\Delta_{g} \bar{u} + \left( S_{g} + \beta \right) \bar{u} - \left( \lambda_{\beta} - \kappa \right) \left( \bar{u} \right)^{p-1} \geqslant 0 \; {\rm in} \; \Omega_{2} \cap \Omega_{3}.
\end{equation*}
Hence $ \bar{u} $ is a super-solution of Yamabe equation (\ref{intro:eqn2}) in $ \Omega $ pointwise. By the setup of $ \Omega_{i}, i = 1, 2, 3 $, it is immediate to check that
\begin{equation*}
\bar{u} \geqslant u_{3} \; {\rm in} \; \Omega.
\end{equation*}
Since $ \bar{u} = \phi $ near $ \partial \Omega $, we define
\begin{equation}\label{manifold:eqn6d}
u_{+} : = \begin{cases} \bar{u}, & \; {\rm in} \; \Omega; \\ \phi, & \; {\rm in} \; M \backslash \Omega. \end{cases}
\end{equation}
It follows that $ u_{+} \in \calC^{\infty}(M) $ is a super-solution of (\ref{manifold:eqn5}) and we conclude that $ 0 \leqslant u_{-} \leqslant u_{+} $. Lastly, by Theorem \ref{pre:thm5} and Remark \ref{pre:re2}, we conclude that there exists a solution $ u \in \calC^{\infty}(M) $ that solves the perturbed Yamabe equation (\ref{manifold:eqnss2}) with the fixed $ \lambda_{s} > 0 $ given in (\ref{manifold:eqnss1}). By maximum principle and the fact that $ u \geqslant u_{-} \not\equiv 0 $, we conclude that $ u > 0 $ on $ M $.
\end{proof}
\begin{remark}\label{manifold:re001} Recall that the local PDE (\ref{local:eqn11}) holds for all $ \lambda > 0 $. This is crucial since it gives us the flexibility to choose the constant associate with $ u^{p-1} $ term, hence we can introduce the extra constant $ \kappa > 0 $ here.
\end{remark}
\medskip

We now apply Theorem \ref{manifold:thm3} to show the existence of the solution of Yamabe equation with $ \eta_{1} > 0 $. Note that from perturbed Yamabe equation (\ref{manifold:eqn5}) to the Yamabe equation, we need to pass the limit by letting $ \beta \rightarrow 0^{-} $. Yamabe tried to let $ u^{s - 1} \rightarrow u^{p-1} $ when $ s \rightarrow p $, and the uniform boundedness of the sequence is not that easy. Due to the local construction of our perturbed Yamabe equations in (\ref{manifold:eqnss2}), we obtain the uniformly positive upper and lower bounds of $ \lVert u_{\beta} \rVert_{\calL^{p}(M, g)} $ when $ \beta \in [\beta_{0}, 0] $, with some fixed $ \beta_{0} < 0 $. With the introduction of the extra positive constant $ \kappa > 0 $, a similar argument due to Trudinger and Aubin then applies to show the uniform boundedness of $ \lVert u_{\beta} \rVert_{\calL^{r}(M, g)} $ for all $ \beta \in [\beta_{0}, 0] $ for some $ r > p $. In contrast to classical method, we obtain this without using $ \lambda(M) < \lambda(\mathbb{S}^{n}) $.
\begin{theorem}\label{manifold:thms}
Let $ (M, g) $ be a closed manifold with $ n = \dim M \geqslant 3 $. Let $ \kappa > 0 $ be some small enough positive constant. Assume the scalar curvature $ S_{g} < 0 $ somewhere on $ M $ and the first eigenvalue $ \eta_{1} > 0 $. Then there exists some $ \lambda > 0 $ such that the Yamabe equation (\ref{intro:eqn2}) has a real, positive, smooth solution.
\end{theorem}
\begin{proof} By Theorem \ref{manifold:thm3}, we have a sequence of real, positive, smooth solutions $ \lbrace u_{\beta} \rbrace $ when $ \beta < 0 $ and $ \lvert \beta \rvert $ is small enough, i.e.
\begin{equation}\label{manifold:eqnss4}
-a\Delta_{g} u_{\beta} + \left(S_{g} + \beta\right) u_{\beta} = \left( \lambda_{\beta} - \kappa \right) u_{\beta}^{p-1} \; {\rm in} \; M.
\end{equation}
We show first that $ \lbrace \lambda_{\beta} \rbrace $ is bounded above, and is increasing and continuous when $ \beta \rightarrow 0^{-} $. It then follows that we can choose $ \kappa $ uniformly so that
\begin{equation*}
\lambda_{\beta} - \kappa > 0, \forall \beta \in [\beta_{0}, 0].
\end{equation*}
We may assume $ \int_{M} d\omega = 1 $ for this continuity verification, since otherwise only an extra term with respect to $ \text{Vol}_{g} $ will appear. Recall that
\begin{equation*}
\lambda_{\beta} = \inf_{u \neq 0, u \in H^{1}(M)} \left\lbrace \frac{\int_{M} a\lvert \nabla_{g} u \rvert^{2} d\omega + \int_{M} \left( S_{g} + \beta \right) u^{2} d\omega}{\left( \int_{M} u^{p} d\omega \right)^{\frac{2}{p}}} \right\rbrace.
\end{equation*}
It is immediate that if $ \beta_{1} < \beta_{2} < 0 $ then $ \lambda_{\beta_{1}} \leqslant \lambda_{\beta_{2}} $. For continuity we assume $ 0 < \beta_{2} - \beta_{1} < \gamma $. For each $ \epsilon > 0 $, there exists a function $ u_{0} $ such that
\begin{equation*}
\frac{\int_{M} a\lvert \nabla_{g} u_{0} \rvert^{2} d\omega + \int_{M} \left( S_{g} + \beta_{1} \right) u_{0}^{2} d\omega}{\left( \int_{M} u_{0}^{p} d\omega \right)^{\frac{2}{p}}} < \lambda_{\beta_{1}} + \epsilon.
\end{equation*}
It follows that
\begin{align*}
\lambda_{\beta_{2}} & \leqslant \frac{\int_{M} a\lvert \nabla_{g} u_{0} \rvert^{2} d\omega + \int_{M} \left( S_{g} + \beta_{2} \right) u_{0}^{2} d\omega}{\left( \int_{M} u_{0}^{p} d\omega \right)^{\frac{2}{p}}} \\
& \leqslant \frac{\int_{M} a\lvert \nabla_{g} u_{0} \rvert^{2} d\omega + \int_{M} \left( S_{g} + \beta_{1} \right) u_{0}^{2} d\omega}{\left( \int_{M} u_{0}^{p} d\omega \right)^{\frac{2}{p}}} + \frac{\left( \beta_{2} - \beta_{1} \right) \int_{M} u_{0}^{2} d\omega}{\left( \int_{M} u_{0}^{p} d\omega \right)^{\frac{2}{p}}} \\
& \leqslant \lambda_{\beta_{1}} + \epsilon + \beta_{2} - \beta_{1} < \lambda_{\beta_{1}} + \epsilon + \beta_{2} - \beta_{1}.
\end{align*}
Since $ \epsilon $ is arbitrarily small, we conclude that
\begin{equation*}
0 < \beta_{2} - \beta_{1} < \gamma \Rightarrow \lvert \lambda_{\beta_{2}} - \lambda_{\beta_{1}} \rvert \leqslant 2\gamma.
\end{equation*}
By the argument in Appendix \ref{APP}, we conclude that $ \lambda_{\beta} < aT $, where $ T $ is the best Sobolev constant for the embedding $ \calL^{p}(\R^{n}) \hookrightarrow H^{1}(\R^{n}) $. Fix some $ \beta_{0} < 0 $ with $ \eta_{1, \beta_{0}} > 0 $, we have
\begin{equation}\label{manifold:eqnss5}
\lambda_{\beta_{0}} \leqslant \lambda_{\beta} \leqslant aT, \forall \beta \in [\beta_{0}, 0], \lim_{\beta \rightarrow 0^{-}} \lambda_{\beta} - \kappa : = \lambda.
\end{equation}
Next we show $ \lVert u_{\beta} \rVert_{H^{1}(M)} \leqslant C $ uniformly when $ \beta \in [\beta_{0}, 0] $. Pairing the PDE (\ref{manifold:eqnss4}) with $ u_{\beta} $, it follows that
\begin{equation}\label{manifold:eqns6a}
\begin{split}
a\lVert \nabla_{g} u_{\beta} \rVert_{\calL^{2}(M, g)}^{2} & = \left(\lambda_{\beta} - \kappa \right) \lVert u_{\beta} \rVert_{\calL^{p}(M, g)}^{p} - \int_{M} \left( S_{g} + \beta \right) u_{\beta}^{2} d\omega \\
& \leqslant  \left(\lambda_{\beta} - \kappa \right) \lVert u_{\beta} \rVert_{\calL^{p}(M, g)}^{p} + \left( \sup_{M} \lvert S_{g} \rvert + \lvert \beta_{0} \rvert \right) \lVert u_{\beta} \rVert_{\calL^{2}(M, g)}^{2}.
\end{split}
\end{equation}
Thus it suffices to show
\begin{equation*}
\lVert u_{\beta} \rVert_{\calL^{p}(M, g)} \leqslant C', \forall \beta \in [\beta_{0}, 0].
\end{equation*}
Denote the corresponding local solutions of (\ref{local:eqn11}) by $ \lbrace \tilde{u}_{\beta} \rbrace $, i.e.
\begin{equation*}
-a\Delta_{g} \tilde{u}_{\beta} + \left(S_{g} + \beta\right) \tilde{u}_{\beta} = \left( \lambda_{\beta} - \kappa \right) \tilde{u}_{\beta}^{p-1} \; {\rm in} \; \Omega, \tilde{u}_{\beta} = 0 \; {\rm on} \; \partial \Omega
\end{equation*}
with fixed domain $ \Omega $. Again we have
\begin{equation}\label{manifold:eqns6}
\begin{split}
a\lVert \nabla_{g} \tilde{u}_{\beta} \rVert_{\calL^{2}(\Omega, g)}^{2} & =  \left( \lambda_{\beta} - \kappa \right) \lVert \tilde{u}_{\beta} \rVert_{\calL^{p}(\Omega, g)}^{p} - \int_{\Omega} \left( S_{g} + \beta \right) u_{\beta}^{2} \dvol \\
\Rightarrow \left( \lambda_{\beta } - \kappa \right) \lVert \tilde{u}_{\beta} \rVert_{\calL^{p}(\Omega, g)}^{p} & \leqslant a\lVert \nabla_{g} \tilde{u}_{\beta} \rVert_{\calL^{2}(\Omega, g)}^{2} + \left( \sup_{\Omega} \lvert S_{g} \rvert + \lvert \beta \rvert \right) \lVert u_{\beta} \rVert_{\calL^{2}(\Omega, g)}^{2}.
\end{split}
\end{equation}
Recall the fuctional $ J(u) $ in (\ref{local:eqn8}) and the quantity $ K_{0} $ in (\ref{local:eqn9}), due to Theorem 1.1 of \cite{WANG}, each solution $ \tilde{u}_{\beta} $ satisfies
\begin{equation}\label{manifold:eqns7}
J(\tilde{u}_{\beta}) \leqslant K_{0} \Rightarrow \frac{a}{2} \lVert \nabla_{g} \tilde{u}_{\beta} \rVert_{\calL^{2}(\Omega, g)}^{2} - \frac{\left( \lambda_{\beta} - \kappa \right)}{p} \lVert \tilde{u}_{\beta} \rVert_{\calL^{p}(\Omega, g)}^{p} - \frac{1}{2} \int_{\Omega} \left(S_{g} + \beta \right) \tilde{u}_{\beta}^{2} \dvol \leqslant K_{0}.
\end{equation}
Note that since $ \lambda_{\beta} $ is bounded above and below when $ \beta \in [\beta_{0}, 0] $, the quantity $ K_{0} $ is uniformly bounded above, and we denote it by $ K_{0} $ again. Apply the estimate (\ref{manifold:eqns6}) into (\ref{manifold:eqns7}), we have
\begin{align*}
\frac{a}{2} \lVert \nabla_{g} \tilde{u}_{\beta} \rVert_{\calL^{2}(\Omega, g)}^{2} & \leqslant K_{0} + \frac{1}{p} \left(a\lVert \nabla_{g} \tilde{u}_{\beta} \rVert_{\calL^{2}(\Omega, g)}^{2} + \left( \sup_{\Omega} \lvert S_{g} \rvert + \lvert \beta \rvert \right) \lVert u_{\beta} \rVert_{\calL^{2}(\Omega, g)}^{2} \right) \\
& \qquad + \frac{1}{2} \left( \sup_{\Omega} \lvert S_{g} \rvert + \lvert \beta \rvert \right) \lVert u_{\beta} \rVert_{\calL^{2}(\Omega, g)}^{2} \\
& \leqslant K_{0} + \frac{a(n - 2)}{2n} \lVert \nabla_{g} \tilde{u}_{\beta} \rVert_{\calL^{2}(\Omega, g)}^{2} \\
& \qquad + \left( \frac{n - 2}{2n} + \frac{1}{2} \right) \left( \sup_{\Omega} \lvert S_{g} \rvert + \lvert \beta \rvert \right) \cdot \lambda_{1}^{-1} \lVert \nabla_{g} \tilde{u}_{\beta} \rVert_{\calL^{2}(\Omega, g)}^{2}; \\
\Rightarrow & \left(\frac{a}{n} - \left( \frac{n - 2}{2n} + \frac{1}{2} \right) \left( \sup_{\Omega} \lvert S_{g} \rvert + \lvert \beta \rvert \right) \cdot \lambda_{1}^{-1} \right) \lVert \nabla_{g} \tilde{u}_{\beta} \rVert_{\calL^{2}(\Omega, g)}^{2} \leqslant K_{0}.
\end{align*}
We have already chosen $ \Omega $ small enough so that (\ref{local:eqns1}) holds for $ \beta_{0} $, thus for all $ \beta \in [\beta_{0}, 0] $, (\ref{local:eqns1}) holds. It follows from above that there exists a constant $ C_{0}' $ such that
\begin{equation*}
\lVert \nabla_{g} \tilde{u}_{\beta} \rVert_{\calL^{2}(\Omega, g)}^{2} \leqslant C_{0}', \forall \beta \in [\beta_{0}, 0].
\end{equation*}
Apply (\ref{manifold:eqns6}) with the other way around, we conclude that
\begin{align*}
\left( \lambda_{\beta} - \kappa \right) \lVert \tilde{u}_{\beta} \rVert_{\calL^{p}(\Omega, g)}^{p} & \leqslant a\lVert \nabla_{g} \tilde{u}_{\beta} \rVert_{\calL^{2}(\Omega, g)}^{2} + \left( \sup_{\Omega} \lvert S_{g} \rvert + \lvert \beta \rvert \right)  \lVert \tilde{u}_{\beta} \rVert_{\calL^{2}(\Omega, g)}^{2} \\
& \leqslant \left( a + \left( \sup_{\Omega} \lvert S_{g} \rvert + \lvert \beta \rvert \right)  \lambda_{1}^{-1} \right) \lVert \nabla_{g} u_{\beta} \rVert_{\calL^{2}(\Omega, g)}^{2}.
\end{align*}
We conclude that
\begin{equation}\label{manifold:eqns8}
\lVert u_{\beta} \rVert_{\calL^{p}(\Omega, g)}^{p} \leqslant C_{1}', \forall \beta \in [\beta_{0}, 0].
\end{equation}
Recall the construction of super-solution of each $ u_{\beta} $ in Theorem \ref{manifold:thm3}, we have
\begin{equation*}
0 < u_{\beta} \leqslant u_{+, \beta} = \begin{cases} \bar{u}_{\beta}, & {\rm in} \; \Omega \\ \phi, & {\rm in} \; M \backslash \Omega \end{cases}.
\end{equation*}
where $ \bar{u}_{\beta} $ is of the form
\begin{equation*}
\bar{u}_{\beta} = \chi_{1} \tilde{u}_{\beta} + \chi_{2} \phi + \chi_{3} (\phi + \gamma).
\end{equation*}
It follows from the uniform upper bound of (\ref{manifold:eqns8}) that
\begin{equation*}
\lVert u_{\beta} \rVert_{\calL^{p}(M, g)}^{p} \leqslant \lVert u_{+, \beta} \rVert_{\calL^{p}(M, g)}^{p} \leqslant E_{0} \left( \lVert \tilde{u}_{\beta} \rVert_{\calL^{p}(\Omega, g)}^{p} + \lVert \phi \rVert_{\calL^{p}(M, g)}^{p} \right).
\end{equation*}
Recall in determining $ \phi $ we require
\begin{equation*}
\frac{\left( \eta_{1} + \beta \right)}{2^{p-2}\left( \lambda_{\beta} - \kappa \right)} > \delta^{p-2} \frac{\sup_{M} \varphi^{p-1}}{\inf_{M} \varphi}.
\end{equation*}
Since when $ \beta \in [\beta_{0}, 0] $, we have $ \lambda_{\beta} \in [\lambda_{\beta_{0}}, aT] $, we can choose a fixed scaling $ \delta $ for all $ \beta $, thus $ \lVert \phi \rVert_{\calL^{p}(M, g)}^{p} $ is uniformly bounded. It follows that there exists a constant $ C' $ such that
\begin{equation*}
\lVert u_{\beta} \rVert_{\calL^{p}(M, g)}^{p} \leqslant C', \forall \beta \in [\beta_{0}, 0].
\end{equation*}
Due to H\"older's inequality, we conclude that
\begin{equation*}
\lVert u_{\beta} \rVert_{\calL^{2}(M, g)} \leqslant \lVert u_{\beta} \rVert_{\calL^{p}(M, g)}^{\frac{2}{p}} \cdot \text{Vol}_{g}(M)^{\frac{p-2}{p}}.
\end{equation*}
Thus by (\ref{manifold:eqns6a}) we conclude that
\begin{equation}\label{manifold:eqns9}
\lVert u_{\beta} \rVert_{H^{1}(M, g)} \leqslant C, \forall \beta \in [\beta_{0}, 0].
\end{equation}
According to the definition of $ \lambda_{\beta} $, we apply the solution of (\ref{manifold:eqnss4}) with the form (\ref{manifold:eqns6a}), it follows that for all $ \beta \in [\beta_{0}, 0] $,
\begin{align*}
& \lambda_{\beta} \leqslant \frac{\int_{M} a\lvert \nabla_{g} u_{\beta} \rvert^{2} d\omega + \int_{M} \left( S_{g} + \beta \right) u_{\beta}^{2} d\omega}{\left( \int_{M} u_{\beta}^{p} d\omega \right)^{\frac{2}{p}}} = \left( \lambda_{\beta} - \kappa \right) \cdot \frac{\lVert u_{\beta} \rVert_{\calL^{p}(M, g)}^{p}}{\left( \int_{M} u_{\beta}^{p} d\omega \right)^{\frac{2}{p}}} \\
\Rightarrow & \frac{\lambda_{\beta}}{\lambda_{\beta} - \kappa} \leqslant \lVert u_{\beta} \rVert_{\calL^{p}(M, g)}^{p - \frac{2}{p}} \Rightarrow \lVert u_{\beta} \rVert_{\calL^{p}(M, g)} \geqslant 1.
\end{align*}
The last cancellation is due to the fact that $ \lambda_{\beta} \geqslant \lambda_{\beta} - \kappa \geqslant \lambda_{\beta_{0}} - \kappa > 0, \forall \beta \in [\beta_{0}, 0] $. It follows that
\begin{equation}\label{manifold:eqns10}
\lVert u_{\beta} \rVert_{H^{1}(M, g)} \leqslant C, 1 \leqslant \lVert u_{\beta} \rVert_{\calL^{p}(M, g)} \leqslant C', \forall \beta \in [\beta_{0}, 0].
\end{equation}
Due to uniform boundedness of $ \calL^{p} $-norms of $ u_{\beta} $, we can scale the metric $ g \mapsto \gamma g $, $ \gamma < 1 $ uniformly for all $ \beta \in [\beta_{0}, 0] $ such that
\begin{equation}\label{manifold:eqns11}
0 < C'' \leqslant \lVert u_{\beta} \rVert_{\calL^{p}(M, g)} \leqslant 1, \forall \beta \in [\beta_{0}, 0].
\end{equation}
We still denote the new metric by $ g $. This scaling does not affect the local estimates in Appendix \ref{APP} since $ \beta < 0 $, and hence does not affect the local solvability of the PDE in Proposition \ref{local:prop3}. This scaling does not affect the fact that $ \lambda_{\beta_{0}} \leqslant \lambda_{\beta} \leqslant aT $ in the sense that both $ \lambda_{\beta} $ and $ T $ are scaled by the factor $ \lambda^{\frac{2}{n}} $, since $ \beta < 0 $, this does not affect the choice of $ \kappa $, i.e. the ratio
\begin{equation*}
\frac{\lambda_{\beta} - \kappa}{T}
\end{equation*}
has a uniform upper bound after we scaling the metric $ g $ by $ \lambda < 1 $. We now apply the analysis, essentially due to Trudinger and Aubin, to show that there exists some $ r > p $, such that
\begin{equation}\label{manifold:eqns12}
\lVert u_{\beta} \rVert_{\calL^{r}(M, g)} \leqslant \mathcal{K}, \forall \beta \in [\beta_{0}, 0].
\end{equation}
Let $ \delta > 0 $ be some constant. Pairing (\ref{manifold:eqnss4}) with $ u_{\beta}^{1 + 2\delta} $ and denote $ w_{\beta} = u_{\beta}^{1 + \delta} $, it follows that
\begin{align*}
& \int_{M} a \nabla_{g} u_{\beta} \cdot \nabla_{g} \left(u_{\beta}^{1 + 2\delta} \right) d\omega + \int_{M} \left(S_{g} + \beta \right) u_{\beta}^{2 + 2\delta} d\omega = \left(\lambda_{\beta} - \kappa \right) \int_{M} u_{\beta}^{p + 2\delta} d\omega; \\
\Rightarrow & \frac{1 + 2\delta}{(1 + \delta)^{2}} \int_{M} a \lvert \nabla_{g} w_{\beta} \rvert^{2} d\omega = \left(\lambda_{\beta} - \kappa \right) \int_{M} w_{\beta}^{2} u_{\beta}^{p-2} d\omega - \int_{M} \left(S_{g} + \beta \right) w_{\beta}^{2} d\omega.
\end{align*}
According to the sharp Sobolev inequality on closed manifolds \cite[Thm.~2.3, Thm.~3.3]{PL}, it follows that for any $ \epsilon > 0 $,
\begin{align*}
\lVert w_{\beta} \rVert_{\calL^{p}(M, g)}^{2} & \leqslant (1 + \epsilon) \frac{1}{T} \lVert \nabla_{g} w_{\beta} \rVert_{\calL^{2}(M, g)}^{2} + C_{\epsilon}' \lVert w_{\beta} \rVert_{\calL^{2}(M, g)}^{2} \\
& = (1 + \epsilon) \frac{1}{aT} \cdot \frac{(1 + \delta)^{2}}{1 + 2\delta} \left( \frac{1 + 2\delta}{(1 + \delta)^{2}} \int_{M} a \lvert \nabla_{g} w_{\beta} \rvert^{2} d\omega \right) + C_{\epsilon}' \lVert w_{\beta} \rVert_{\calL^{2}(M, g)}^{2} \\
& = (1 + \epsilon) \frac{\lambda_{\beta} - \kappa}{aT} \cdot \frac{(1 + \delta)^{2}}{1 + 2\delta} \int_{M} w_{\beta}^{2} u_{\beta}^{p - 2} d\omega + C_{\epsilon}  \lVert w_{\beta} \rVert_{\calL^{2}(M, g)}^{2} \\
& \leqslant (1 + \epsilon) \frac{\lambda_{\beta} - \kappa}{aT} \cdot \frac{(1 + \delta)^{2}}{1 + 2\delta} \lVert w_{\beta} \rVert_{\calL^{p}(M, g)}^{2} \lVert u_{\beta} \rVert_{\calL^{p}(M, g)}^{p} + C_{\epsilon}  \lVert w_{\beta} \rVert_{\calL^{2}(M, g)}^{2} \\
& \leqslant (1 + \epsilon) \frac{\lambda_{\beta} - \kappa}{aT} \cdot \frac{(1 + \delta)^{2}}{1 + 2\delta} \lVert w_{\beta} \rVert_{\calL^{p}(M, g)}^{2} + C_{\epsilon}  \lVert w_{\beta} \rVert_{\calL^{2}(M, g)}^{2}.
\end{align*}
Note that $ C_{\epsilon} $ is uniform over all $ \beta $. The last two lines above are due to H\"older's inequality and (\ref{manifold:eqns11}). Since $ \lambda_{\beta} \leqslant aT $, $ \lambda_{\beta} - \kappa < aT $; thus we can choose $ \epsilon, \delta $ small enough such that
\begin{equation*}
(1 + \epsilon) \frac{\lambda_{\beta} - \kappa}{aT} \cdot \frac{(1 + \delta)^{2}}{1 + 2\delta} < 1.
\end{equation*}
It follows that
\begin{equation*}
\lVert w_{\beta} \rVert_{\calL^{p}(M, g)}^{2} \leqslant \mathcal{K}_{1} \lVert w_{\beta} \rVert_{\calL^{2}(M, g)}^{2} \leqslant \mathcal{K}_{2} \lVert u_{\beta} \rVert_{\calL^{p}(M, g)}^{1 + \delta} \leqslant \mathcal{K}_{3}, \forall \beta \in [\beta_{0}, 0].
\end{equation*}
Note that $ \lVert w_{\beta} \rVert_{\calL^{p}(M, g)} = \lVert u_{\beta} \rVert_{\calL^{p(1 + \delta)}(M, g)}^{1 + \delta} $, it follows that (\ref{manifold:eqns12}) holds. Since $ r > p $, by bootstrapping method we conclude that
\begin{equation}\label{manifold:eqns13}
\lVert u_{\beta} \rVert_{\calC^{2, \alpha}(M)} \leqslant \mathcal{K}', \forall \beta \in [\beta_{0}, 0].
\end{equation}
Due to Arzela-Ascoli, we pass the limit by letting $ \beta \rightarrow 0^{-} $ on both sides of (\ref{manifold:eqnss4}), it follows that
\begin{equation*}
-a\Delta_{g} u + S_{g} u = \lambda u^{p-1} \; {\rm on} \; M.
\end{equation*}
Another bootstrapping method shows that $ u \in \calC^{\infty}(M) $ and $ u \geqslant 0 $ since $ u $ is the limit of a positive sequence. Lastly we show that $ u > 0 $. Due to uniform boundedness in (\ref{manifold:eqns13}) and lower bound of $ \calL^{p}$-norms in (\ref{manifold:eqns11}), we conclude that
\begin{equation*}
\lVert u \rVert_{\calL^{p}(M, g)} \geqslant C'' > 0.
\end{equation*}
It follows from maximum principle that $ u > 0 $ on $ M $.
\end{proof}
\begin{remark}\label{manifold:re002}
With the introduction of $ \kappa > 0 $ in (\ref{manifold:eqnss4}), we can apply this method for all cases, even when the manifold is conformal to the standard sphere $ \mathbb{S}^{n} $, since we still have $ \lambda_{\beta} - \kappa < \lambda(\mathbb{S}^{n}) = aT $.
\end{remark}
\medskip

Lastly we show the case when $ S_{g} \geqslant 0 $ everywhere. Due to \cite{KW}, this amounts to say $ \eta_{1} > 0 $. Therefore we consider the solution of Yamabe equation when $ \lambda > 0 $. We show this by showing that there exists some $ \tilde{g} = e^{2f} g = u^{p-2} g $ for some positive $ u \in \calC^{\infty}(M) $ such that $ S_{\tilde{g}} < 0 $ somewhere. Obviously, $ S_{\tilde{g}} $ cannot be too negative too often, actually due to Theorem \ref{pre:thm4}(i) and taking $ \varphi = 1 $ in (\ref{manifold:eqn2}), a necessary condition is $ \int_{M} S_{\tilde{g}} > 0 $. Consequently we can apply Theorem \ref{manifold:thm3} to obtain some metric with positive constant scalar curvature.
\begin{theorem}\label{manifold:thm4} Let $ (M, g) $ be a closed manifold with $ n = \dim M \geqslant 3 $. Assume the scalar curvature $ S_{g} \geqslant 0 $ everywhere on $ M $. Then there exists some smooth function $ H \in \calC^{\infty}(M) $ which is negative somewhere such that $ H $ is the scalar curvature with respect to some conformal change of $ g $.
\end{theorem}
\begin{proof} We show this by constructing the smooth function $ H $ with desired property for which the following PDE
\begin{equation}\label{manifold:eqnaa}
-a\Delta_{g} u + S_{g} u = Hu^{p-1} \; {\rm in} \; (M, g)
\end{equation}
admits a smooth, real, positive solution $ u $. Without loss of generality, we can assume $ \sup_{x \in M} (S_{g}) \leqslant 1 $ for some $ g $ after scaling. 
\medskip 

To construct $ H $, we first construct a smooth function $ F \in \calC^{\infty}(M) $ such that: (i) $ F $ is very negative at some point in $ M $; (ii) $ \int_{M} F d\omega = 0 $; and (iii) $ \lVert F \rVert_{H^{s - 2}(M, g)} $ is very small simultaneously. Here $ s = \frac{n}{2} + 1 $ if $ n $ is even and $ s = \frac{n + 1}{2} $ if $ n $ is odd. For simplicity, we discuss the even dimensional case, the odd dimensional case is exactly the same except the order $ s $.

Pick up a normal ball $B_{r} $ in $ M $ with radius $ r < 1 $ centered at some point $ q \in M $. Choose a function 
\begin{equation}\label{manifold:eqn7a}
f(q) = - C< -1, f \leqslant 0, \lvert f \rvert \leqslant C, f \in \calC^{\infty}(M), f \equiv 0 \; {\rm on} M \backslash B_{r}.
\end{equation}
Such an $ f $ can be chosen as any smooth bump function or cut-off function in partition of unity multiplying by $ - 1 $. For $ f $, we have
\begin{equation}\label{manifold:eqn7}
\lvert \partial^{\alpha} f \rvert \leqslant \frac{C'}{r^{\lvert \alpha \rvert}}, \forall \lvert \alpha \rvert \leqslant \frac{n}{2} - 1 = s - 2.
\end{equation}
By (\ref{manifold:eqn7}), we have
\begin{equation}\label{manifold:eqn8}
\left\lvert \int_{M} f d\omega \right\rvert = \left\lvert \int_{B_{r}} f \dvol \right\rvert \leqslant C \text{Vol}_{g}(B_{r}) \leqslant C \cdot \Gamma r^{n}.
\end{equation}
Here $ C $ is the upper bound of $ \lvert f \rvert $. The hypothesis of the statement says that $ S_{g} \geqslant 0 $ everywehere, hence $ \Gamma $ depends only on the dimension of $ M $ for every geodesic ball centered at any point of $ M $, due to Bishop-Gromov comparison inequality. Specifically, the volume of a geodesic ball of radius $ r \ll 1 $ centered at $ p $ is given by $ Vol_{g}(B_{r}) = \left( 1 - \frac{S_{g}(0)}{6(n + 2)} r^{2} + O(r^{4}) \right) Vol_{e}(B_{r}) $. Thus $ \Gamma $ is fixed when $ (M, g) $ is fixed.
Now we define
\begin{equation}\label{manifold:eqn9}
F : = f + \epsilon
\end{equation}
where $ 0 < \epsilon < C $ is determined later. First we show that for any $ C $, we can choose $ \epsilon $ arbitrarily small so that $ \int_{M} F d\omega = 0 $. Set
\begin{equation*}
\epsilon : = \frac{1}{\text{Vol}_{g}(M)} \int_{B_{r}} -f \dvol \Rightarrow \int_{M} F d\omega = \int_{M} (f + \epsilon) d\omega = \int_{B_{r}} f \dvol + \epsilon \int_{M} d\omega = 0.
\end{equation*}
As we have shown in (\ref{manifold:eqn8}),
\begin{equation*}
\epsilon \leqslant \frac{1}{\text{Vol}_{g}(M)} \left\lvert \int_{B_{r}} f \dvol \right\rvert \leqslant \frac{1}{\text{Vol}_{g}(M)} C \Gamma r^{n} = O(r^{n}).
\end{equation*}
Thus if we choose $ r $ small enough then $ \epsilon $ is small enough.

Secondly we show that for any $ C, C' $, we can make $ \lVert F \rVert_{H^{s - 2}(M, g)} $ arbitrarily small. By Bishop-Gromov inequality and (\ref{manifold:eqn7}), we observe that
\begin{align*}
\lVert F \rVert_{H^{s - 2}(M, g)}^{2} & = \int_{M} \lvert f + \epsilon \rvert^{2} d\omega + \sum_{1 \leqslant \lvert \alpha \rvert \leqslant s - 2} \int_{B_{r}} \lvert \partial^{\alpha} f \rvert^{2} \dvol \\
& \leqslant C_{0} r^{n} + \sum_{1 \leqslant \lvert \alpha \rvert \leqslant s - 2} \int_{B_{r}} \frac{(C')^{2}}{r^{n - 2}} \dvol \leqslant C_{0} r^{n} + \sum_{1 \leqslant \lvert \alpha \rvert \leqslant s - 2} \frac{(C')^{2}}{r^{n - 2}} \Gamma r^{n} \leqslant (C')^{2} C_{0}'r^{2}
\end{align*}
for some $ C_{0}' $ depends only on the the order $ s - 2 = \frac{n}{2} - 1 $, the lower bound of Ricci curvature of $ M $, and the diameter of $ (M, g) $. Therefore, for any $ C, C' $, we can choose $ r $ small enough so that $ \lVert F \rVert_{H^{s - 2}(M, g)}^{2} $ is arbitrarily small.

It follows that this smooth function $ F $ we constructed satisfies all requirements mentioned above. Now we construct $ u \in \calC^{\infty}(M) $. Let $ u' $ be the solution of
\begin{equation}\label{manifold:eqn9}
-a\Delta_{g} u' = F \; {\rm in} \; (M, g).
\end{equation}
Due to Lax-Milgram \cite[Ch.~6]{Lax} and hypothesis (i) of $ F $, there exists such an $ u' $ that solves (\ref{manifold:eqn9}). Since $ F \in \calC^{\infty}(M) $, standard elliptic regularity implies that $ u' \in \calC^{\infty}(M) $. Recall the Sobolev embedding in Theorem \ref{pre:thm2}(iii) and $ H^{s} $-type elliptic regularity in Theorem \ref{pre:thm1}(i), we have fixed constants $ K' $ and $ C^{**} $ such that
\begin{align*}
\lvert u' \rvert \leqslant \lVert u' \rVert_{\calC^{0, \beta}(M)} \leqslant K' \lVert u' \rVert_{H^{s}(M, g)}, \frac{1}{2} - \frac{s}{n} = \frac{1}{2} - \frac{1}{2} - \frac{1}{n} < 0; \\
\lVert u' \rVert_{H^{s}(M, g)} \leqslant C^{*} \left( \lVert F \rVert_{H^{s - 2}(M, g)} + \lVert u' \rVert_{\calL^{2}(M, g)} \right) \leqslant C^{**} \lVert F \rVert_{H^{s - 2}(M, g)}.
\end{align*}
Note that the last inequality comes from the PDE (\ref{manifold:eqn9}) by pairing $ u ' $ on both sides, and apply Poincar\'e inequality. However, we should have $ \int_{M} u' d\omega = 0 $ to apply Poincar\'e inequality, this can be done by taking $ u' \mapsto u' + \epsilon $ if necessary. Thus $ K' $ depends only on $ n, \alpha, s, M, g $ and $ C^{**} $ depends only on $ s, -a\Delta_{g}, \lambda_{1}, M, g $, all of which are fixed once $ M, g, s $ are fixed.

Note that $ f(q) = - C $ for some $ C > 1 $ at the point $ q \in M $. First choose $ r $ small enough so that $ \epsilon < \frac{1}{2} $, it follows that
\begin{equation}\label{manifold:eqn10a}
F(q) = f(q) + \epsilon \leqslant -\frac{C}{2}.
\end{equation}
Now we further shrink $ r $ such that
\begin{equation*}
C^{**}K' \lVert F \rVert_{H^{s - 2}(M, g)} \leqslant  C^{**} K' \cdot C' \left( C_{0}' \right)^{\frac{1}{2}} r \leqslant \frac{C}{8}.
\end{equation*}
Thus by Sobolev embedding and elliptic regularity, it follows that
\begin{equation}\label{manifold:eqn10}
\lvert u' \rvert \leqslant C^{**} K' \lVert F \rVert_{H^{s - 2}(M, g)} \leqslant \frac{C}{8} \Rightarrow u' \in [-\frac{C}{8}, \frac{C}{8}], \forall x \in M.
\end{equation}
We now define $ u $ to be
\begin{equation}\label{manifold:eqn11}
u : = u' + \frac{C}{4}.
\end{equation}
By (\ref{manifold:eqn10}), it is immediate to see that $ u > 0 $ on $ M $ and $ u \in \calC^{\infty}(M) $ due to definition and smoothness of $ u' $. In particular $ u \in [\frac{C}{8}, \frac{3C}{8}] $. Furthermore, $ u $ solves (\ref{manifold:eqn9}). Finally, we define our $ H $ to be
\begin{equation}\label{manifold:eqn12}
H : = u^{1 - p} \left( -a\Delta_{g} u + S_{g} u \right), \forall x \in M.
\end{equation}
Clearly $ H \in \calC^{\infty}(M) $. By $ -a\Delta_{g} u = F $, (\ref{manifold:eqn10a}), (\ref{manifold:eqn10}) and (\ref{manifold:eqn11}), we observe that
\begin{equation*}
H(q) = u(q)^{1- p}\left( -a\Delta_{g} u(q) + S_{g}(q) u(q) \right) = u(q)^{1- p}\left( F(q) + S_{g}(q) u(q) \right) \leqslant u(q)^{1 - p} \left(-\frac{C}{2} + \frac{3C}{8} \right) < 0.
\end{equation*}
Based on our constructions of $ u $ and $ H $, (\ref{manifold:eqn12}) is equivalent to say that the following PDE
\begin{equation*}
-a\Delta_{g} u + S_{g} u = Hu^{p-1} \; {\rm in} \; (M, g)
\end{equation*}
has a real, positive, smooth solution. It follows that $ H $ is the scalar curvature of the metric $ u^{p-2} g $ on $ M $.
\end{proof}
\begin{remark}\label{manifold:renew}
We see from above that
\begin{equation*}
H(q) \leqslant u(q)^{1 - p} \left(-\frac{C}{2} + \frac{3C}{8} \right) < 0.
\end{equation*}
Furthermore, we observe from above that
\begin{equation*}
\lvert u' \rvert \leqslant C^{**} K' \lVert F \rVert_{H^{s - 2}(M, g)} \leqslant  C^{**} K \cdot C' \left( C_{0}' \right)^{\frac{1}{2}} r.
\end{equation*}
It follows that we can make $ \lvert u' \rvert < \epsilon' $ for any $ \epsilon' $ small enough. Thus taking
\begin{equation*}
u = u' + 2 \epsilon' \Rightarrow \lvert u \rvert \leqslant 3\epsilon'.
\end{equation*}
With fixed $ C $, we can make $ H(q) $ arbitrarily large since $ u^{1 - p} = u^{-\frac{n+2}{n-2}} $ would be arbitrarily large with arbitrarily small $ \epsilon' $.
\end{remark}
\medskip

An immediate analogy of Theorem \ref{manifold:thm4} is that when $ S_{g} \leqslant 0 $ everywhere and hence $ \eta_{1} < 0 $, there exists some $ H' \in \calC^{\infty}(M) $, positive somewhere, that is the prescribed scalar curvature of some $ \tilde{g} $ under conformal change.
\begin{corollary}\label{manifold:cor4} Let $ (M, g) $ be a closed manifold with $ n = \dim M \geqslant 3 $. Let $ S_{g} \leqslant 0 $ everywhere on $ M $ be the scalar curvature with respect to $ g $. Then there exists some smooth function $ H' \in \calC^{\infty}(M) $ which is positive somewhere such that $ H' $ is the scalar curvature with respect to some conformal change of $ g $.
\end{corollary}
\begin{proof}
Taking $ f \geqslant 0 $ with properties as in Theorem \ref{manifold:thm4} and do the same argument.
\end{proof}

With the help of Theorem \ref{manifold:thm4}, we can solve the Yamabe equation for the case $ S_{g} \geqslant 0 $ everywhere. Note that this amounts to say $ \eta_{1} > 0 $ due to (\ref{manifold:eqn2}).
\begin{theorem}\label{manifold:thm5} Let $ (M, g) $ be a closed manifold with $ n = \dim M \geqslant 3 $. Assume the scalar curvature $ S_{g} \geqslant 0 $ everywhere on $ M $. Then there exists some $ \lambda > 0 $ such that the Yamabe equation (\ref{intro:eqn2}) has a real, positive, smooth solution.
\end{theorem}
\begin{proof} After scaling we can, without loss of generality, assume $ \sup_{M} S_{g} \leqslant 1 $. We start with this metric. By Theorem 4.4, we conclude that there exists some smooth function $ H < 0 $ somewhere such that it is the scalar curvature with respect to the metric $ \tilde{g}_{1} = u^{p-2} g $ for some smooth $ u > 0 $ on $ M $. By Theorem \ref{pre:thm4}(i) the first eigenvalue for the conformal Laplacian $ -a\Delta_{\tilde{g}_{1}} + H $ is positive.

Consequently by Theorem \ref{manifold:thm3}, there exists some $ \lambda > 0 $ such that (\ref{intro:eqn2}) has some real, positive solution $ v \in \calC^{\infty}(M) $, i.e. $ \lambda $ is the scalar curvature of the metric $ \tilde{g} = v^{p-2} \tilde{g}_{1} $. Combining both conformal changes together, we see that
\begin{equation*}
\tilde{g} = v^{p-2} \tilde{g}_{1} = v^{p-2} u^{p-2} g = (uv)^{p-2} g.
\end{equation*}
Clearly $ uv > 0 $ and smooth on $ M $. We now conclude that $ \tilde{g} $ is associated with the positive constant scalar curvature $ \lambda $ under the conformal change $ \tilde{g} = (uv)^{p-2} g $.
\end{proof}

\section*{Acknowledgement}
The author would like to thank his advisor Prof. Steven Rosenberg for his support and mentorship.

\appendix

\section{Proof of The Claim}\label{APP}
In this section, we show that the claim (\ref{local:eqn18}) in Proposition \ref{local:prop3} holds. The key is a very careful analysis in terms of local expression of $ \sqrt{\det(g)} $ and $ g^{ij} $. This estimate is sharp in the sense that if we drop any negative term in numerator or any positive term in denominator with some specific order then we cannot detect the desired upper bound, see Remark \ref{local:re1}. Note that the quantity $ Q_{\epsilon, \Omega} $ is not a Yamabe quotient because of the perturbed term $ \beta $. Without loss of generality, we choose $ \Omega = B_{0}(r) $ be a ball of radius $ r $ and centered at $ 0 $. We fix the radius $ r $ due to the discussion in Proposition \ref{local:prop3}. We interchangeably use $ \Omega $ and $ B_{0}(r) $ when there will be no confusion.

When $ n = 3 $, We set $ \Omega = B_{0}(1) $ and thus we choose the cut-off function as
\begin{equation*}
\varphi_{\Omega}(x) = \cos \left(\frac{\pi \lvert x \rvert}{2} \right).
\end{equation*}

When $ n \geqslant 4 $, we choose $ \varphi_{\Omega}(x) = \varphi_{\Omega} \left( \lvert x \rvert \right) $ to be any radial bump function supported in $ \Omega $ such that $  \varphi_{\Omega}(x) \equiv 1 $ on $ B_{0} \left( \frac{r}{2} \right) $. Therefore, the test function we choose in $ \Omega $ is
\begin{equation}\label{APP:eqnt0}
u_{\epsilon, \Omega}(x) = \frac{\varphi_{\Omega}(x)}{\left(\epsilon + \lvert x \rvert^{2}\right)^{\frac{n - 2}{2}}}, n \geqslant 3.
\end{equation}
We will simply denote $ u_{\epsilon, \Omega}  = u $ when there would be no confusion. 
\medskip

When $ n \geqslant 4 $, denote
\begin{align*}
K_{1} & = (n - 2)^{2} \int_{\R^{n}} \frac{\lvert y \rvert^{2}}{( 1 + \lvert y \rvert^{2} )^{n}} dy; \\
K_{2} & =\left( \int_{\R^{n}} \frac{1}{( 1 + \lvert y \rvert^{2} )^{n}} dy \right)^{\frac{2}{p}}; \\
K_{3} & = \int_{\R^{n}} \frac{1}{(1 + \lvert y \rvert^{2})^{n - 2}} dy.
\end{align*}
We will keep using $ K_{1}, K_{2} $ when $ n = 3 $, but $ K_{3} $ is not integrable when $ n = 3 $. The best Sobolev constant $ T $ satisfies
\begin{equation}\label{APP:eqnt4}
T = \frac{K_{1}}{K_{2}}, n \geqslant 3.
\end{equation}
We use the notation $ T $ for all cases $ n \geqslant 3 $.
\medskip

\begin{theorem}\label{APP:thm1} Let $ \Omega = B_{0}(r) $ for small enough $ r $ with $ \dim \Omega \geqslant 3 $. Let $ \beta < 0 $ be any negative constant and $ T $ be the best Sobolev constant in (\ref{APP:eqnt4}). The quantity $ Q_{\epsilon, \Omega} $ satisfies
\begin{equation}\label{APP:eqn1}
Q_{\epsilon, \Omega} : =\frac{\lVert \nabla_{g} u_{\epsilon, \Omega} \rVert_{\calL^{2}(\Omega, g)}^{2} + \frac{1}{a} \int_{\Omega} \left(S_{g} + \beta \right) u_{\epsilon, \Omega}^{2} \sqrt{\det(g)} dx}{\lVert u_{\epsilon, \Omega} \rVert_{\calL^{p}(\Omega, g)}^{2}} < T.
\end{equation}
with the test functions given in (\ref{APP:eqnt0}).
\end{theorem}
The theorem will be proved in three cases: (i) $ n \geqslant 5 $; (ii) $ n = 4 $; (iii) $ n = 3 $.
\begin{proof} Let $ \Omega $ be a small enough geodesic ball centered at $ 0 $ with radius $ r $, as discussed above. It is well known that \cite[\S5]{PL}, \cite[\S3]{ESC} on normal coordinates $ \lbrace x_{i} \rbrace $ within a small geodesic ball of radius $ r $, we have:
\begin{equation}\label{APP:eqn1a}
\begin{split}
g^{ij} & = \delta^{ij} - \frac{1}{3} \sum_{r, s} R_{ijrs} x_{r}x_{s} + O\left(\lvert x \rvert^{3}\right); \\
\sqrt{\det(g)} & = 1 - \frac{1}{6} \sum_{i, j} R_{ij} x_{i}x_{j} + O\left( \lvert x \rvert^{3} \right).
\end{split}
\end{equation}
Here $ R_{ij} $ are the coefficients of Ricci curvature tensors. In local expression of $ \sqrt{\det(g)} $, the term $ \sum_{i, j} R_{ij} x_{i}x_{j} = \text{Ric}(x, x) $. Since the Ricci curvature tensor is a symmetric bilinear form, there exists an orthonormal principal basis $ \lbrace v_{1}, \dotso, v_{n} \rbrace $ on tangent space of each point in the geodesic ball such that its corresponding eigenvalues $ \lbrace \lambda_{i} \rbrace $ satisfies
\begin{equation*}
\sum_{i} \lambda_{i} = S_{g}(0).
\end{equation*}
Let $ \lbrace y_{1}, \dotso, y_{n} \rbrace $ be the corresponding normal coordinates with respect to the principal basis, the Ricci curvature tensor is diagonalized at the center $ 0 $, thus in a very small ball we have
\begin{equation}\label{APP:eqn1b}
\begin{split}
& R_{ij}(y) = \delta_{ij} R_{ij}(0) + o(1) = \delta_{ij} R_{ij}(0) + \sum_{\lvert \alpha \rvert = 1} \partial^{\alpha} R_{ij}(0) y^{\alpha} + O(\lvert y \rvert^{2}) \\
\Rightarrow & S_{g}(y) = R_{ij} g^{ij} = S_{g}(0) + \sum_{i, \lvert \alpha \rvert = 1} \partial^{\alpha} R_{ii}(0) y^{\alpha} + O(\lvert y \rvert^{2}); \\
& R_{ij}(y) = \delta_{ij} R_{ij}(0) + \sum_{\lvert \alpha \rvert = 1} \partial^{\alpha} R_{ij}(0) y^{\alpha} + O(\lvert y \rvert^{2}) \\
\Rightarrow & \sqrt{\det(g)} = 1 - \frac{1}{6} \sum_{i} \left(R_{ii}(0) + \sum_{\lvert \alpha \rvert = 1} \partial^{\alpha} R_{ii}(0) y^{\alpha} \right) y_{i}^{2} - \frac{1}{6} \sum_{i \neq j, \lvert \alpha \rvert = 1} \partial^{\alpha} R_{ij}(0) y^{\alpha} y_{i} y_{j} + o(\lvert y \rvert^{3}).
\end{split}
\end{equation}
The smallness $ o(1) $ above only depends on the smallness of $ r $ and is no larger than $ O(r) $ due to the Taylor expansion of the smooth functions $ R_{ij}(y) $ near $ y = 0 $:
\begin{equation*}
R_{ij}(y) = R_{ij}(0) + \sum_{\lvert \alpha \rvert = 1} \partial^{\alpha} R_{ij}(0) y^{\alpha} + O(\lvert y \rvert^{2}).
\end{equation*}
Since $ \varphi_{\Omega}(y) = \varphi_{\Omega}(\lvert y \rvert) $ is radial; note also that in normal coordinates $ \lbrace y_{i} \rbrace $, we also have
\begin{equation*}
s^{2} = \lvert y \rvert^{2} \Rightarrow g^{ss} \equiv 1 \Rightarrow \lvert \nabla_{g} u_{\epsilon, \Omega} \rvert^{2} = \lvert \partial_{s} u \rvert^{2}.
\end{equation*}
We apply (\ref{APP:eqn1b}) and have
\begin{equation}\label{APP:eqn2}
\begin{split}
\lVert \nabla_{g} u \rVert_{\mathcal{L}^{2}(\Omega, g)}^{2} & = \int_{\Omega} \lvert \nabla_{g} u \rvert^{2} \sqrt{\det(g)} dy \\
& \leqslant \int_{\Omega} \left( 1 - \frac{1}{6} \sum_{i} \left(R_{ii}(0) + \sum_{\lvert \alpha \rvert = 1} \partial^{\alpha} R_{ii}(0) y^{\alpha} \right) y_{i}^{2} \right) \lvert \partial_{s} u \rvert^{2} dy \\
& \qquad - \int_{\Omega} \left( \frac{1}{6} \sum_{i \neq j, \lvert \alpha \rvert = 1} \partial^{\alpha} R_{ij}(0) y^{\alpha} y_{i} y_{j} + o(\lvert y \rvert^{3}) \right) \lvert \partial_{s} u \rvert^{2} dy \\
& \leqslant \int_{\Omega} \lvert \partial_{s} u \rvert^{2} dy -\frac{1}{6n}  \int_{\Omega} S_{g}(0) \lvert y \rvert^{2} \lvert \partial_{s} u \rvert^{2} dy + \tilde{C}_{1} \int_{\Omega} \lvert y \rvert^{3} \lvert \partial_{s} u \rvert^{2} dy \\
& : = \lVert \partial_{s} u \rVert_{\calL^{2}(\Omega)}^{2} + A_{1} + A_{2}; \\
\frac{1}{a} \int_{\Omega} (S_{g} + \beta) u^{2} \sqrt{\det(g)} dx & \leqslant \frac{1}{a} \int_{\Omega} \left( S_{g}(0) + \beta \right) \lvert u \rvert^{2} dy + \frac{1}{a} \int_{\Omega} \sum_{i, \lvert \alpha \rvert = 1} \partial^{\alpha} R_{ii}(0) y^{\alpha} \lvert u \rvert^{2} dy \\
& \qquad + \tilde{C}_{2} \int_{\Omega} \lvert y \rvert^{2} \lvert u \rvert^{2} dy + \tilde{C}_{3} \int_{\Omega} \lvert y \rvert^{3} \lvert  u \rvert^{2} dy \\
& : =  \frac{1}{a} \int_{\Omega} \left( S_{g}(0) + \beta \right) \lvert u \rvert^{2} dy + B_{1} + B_{2} +B_{3}; \\
\lVert u \rVert_{\calL^{p}(\Omega, g)}^{p} & = \int_{\Omega} \lvert u \rvert^{p} \sqrt{\det(g)} dx = \int_{\Omega} \lvert u \rvert^{p} dx - \frac{1}{6n} \int_{\Omega} S_{g}(0) \lvert y \rvert^{2} \lvert u \rvert^{p} dy \\
& \qquad + \tilde{C}_{4} \int_{\Omega} \lvert y \rvert^{3} \lvert  u \rvert^{p} dy : = \lVert u \rVert_{\calL^{p}(\Omega)}^{p} + C_{1} + C_{2}.
\end{split}
\end{equation}
Right sides of inequalities in (\ref{APP:eqn2}) are all Euclidean norms and derivatives. We show that when $ \epsilon $ small enough, the test function in (\ref{APP:eqnt0}) satisfies
\begin{align*}
&\frac{\lVert \nabla_{g} u_{\epsilon, \Omega} \rVert_{\calL^{2}(\Omega, g)}^{2} + \frac{1}{a} \int_{\Omega} \left(S_{g} + \beta \right) u_{\epsilon, \Omega}^{2} \sqrt{\det(g)} dx}{\lVert u_{\epsilon, \Omega} \rVert_{\calL^{p}(\Omega, g)}^{2}} \\
& \qquad \leqslant \frac{\lVert \partial_{s} u_{\epsilon, \Omega} \rVert_{\calL^{2}(\Omega)}^{2} + \int_{\Omega} S_{g}(0) u_{\epsilon, \Omega}^{2} dx + A_{1} + A_{2}  + B_{1} + B_{2} + B_{3}}{\left( \lVert u_{\epsilon, \Omega} \rVert_{\mathcal{L}^{p}(\Omega)}^{p}  + C_{1}  + C_{2} \right)^{\frac{2}{p}}} < T.
\end{align*}
We show the inequality (\ref{APP:eqn1}) in three cases, classified by dimensions. Note that by (\ref{APP:eqnt0}), $ u = u_{\epsilon, \Omega} $ is radial, thus we have
\begin{equation*}
\partial_{s} u_{\epsilon, \Omega} = \partial_{s} \left( \frac{\varphi_{\Omega}(s)}{\left( \epsilon + \lvert s \rvert^{2} \right)^{\frac{n-2}{2}}} \right) = \frac{\partial_{s} \varphi_{\Omega}(\lvert y \rvert)}{\left( \epsilon + \lvert y \rvert^{2} \right)^{\frac{n-2}{2}}} - (n - 2) \frac{\varphi_{\Omega}(\lvert y \rvert)  y }{\left( \epsilon + \lvert y \rvert^{2} \right)^{\frac{n}{2}}}.
\end{equation*}
Thus the terms $ \lVert \partial_{s} u \rVert_{\calL^{2}(\Omega)}^{2} $, $ A_{1} $, $ A_{2} $ and $ A_{3} $ can be estimated as
\begin{align*}
\lVert \partial_{s} u \rVert_{\calL^{2}(\Omega)}^{2} & = \int_{\Omega} \frac{\left\lvert \partial_{s} \varphi_{\Omega} \right\rvert^{2}}{(\epsilon + \lvert y \rvert^{2})^{n - 2}} dy - 2(n - 2) \int_{\Omega} \frac{\varphi_{\Omega} \left( \partial_{s} \varphi_{\Omega} \cdot y \right)}{(\epsilon + \lvert y \rvert^{2})^{n - 1}} dy + (n - 2)^{2} \int_{\Omega} \frac{\varphi_{\Omega}^{2} \lvert y \rvert^{2}}{(\epsilon + \lvert y \rvert^{2})^{n}} dy \\
& = (n - 2)^{2} \int_{\R^{n}} \frac{\lvert y \rvert^{2}}{(\epsilon + \lvert y \rvert^{2})^{n}} dy + \int_{\Omega} \frac{\left\lvert \partial_{s} \varphi_{\Omega} \right\rvert^{2}}{(\epsilon + \lvert y \rvert^{2})^{n - 2}} dy - 2(n - 2) \int_{\Omega} \frac{\varphi_{\Omega} \left( \partial_{s} \varphi_{\Omega} \cdot y \right)}{(\epsilon + \lvert y \rvert^{2})^{n - 1}} dy \\
& \qquad + (n - 2)^{2}  \int_{\Omega} \frac{\left(\varphi_{\Omega}^{2} - 1\right) \lvert y \rvert^{2}}{(\epsilon + \lvert y \rvert^{2})^{n}} dy - (n - 2)^{2} \int_{\R^{n} \backslash \Omega} \frac{\lvert y \rvert^{2}}{(\epsilon + \lvert y \rvert^{2})^{n}} dy \\
& : = (n - 2)^{2} \int_{\R^{n}} \frac{\lvert y \rvert^{2}}{(\epsilon + \lvert y \rvert^{2})^{n}} dy + A_{0, 1} + A_{0, 2} + A_{0, 3} + A_{0, 4}; \\
A_{1} & = -\frac{S_{g}(0)}{6n} \int_{\Omega} \frac{\lvert y \rvert^{2} \left\lvert \partial_{s} \varphi_{\Omega} \right\rvert^{2}}{(\epsilon + \lvert y \rvert^{2})^{n - 2}} dy + \frac{S_{g}(0)(n - 2)}{3n} \int_{\Omega} \frac{\varphi_{\Omega} \lvert y \rvert^{2} \left( \partial_{s} \varphi_{\Omega} \cdot y \right)}{(\epsilon + \lvert y \rvert^{2})^{n - 1}} dy \\
& \qquad -\frac{S_{g}(0)(n - 2)^{2}}{6n} \int_{\Omega} \frac{\varphi_{\Omega}^{2} \lvert y \rvert^{4}}{(\epsilon + \lvert y \rvert^{2})^{n}} dy \\
& : = A_{1, 1} + A_{1, 2} + A_{1, 3}; \\
A_{2} & = \tilde{C}_{1} \int_{\Omega} \lvert y \rvert^{3} \lvert \partial_{s} u_{\epsilon, \Omega} \rvert^{2} dy = \tilde{C}_{1} \int_{\Omega} \frac{\lvert y \rvert^{3} \left\lvert \partial_{s} \varphi_{\Omega} \right\rvert^{2}}{(\epsilon + \lvert y \rvert^{2})^{n - 2}} dy -2 \tilde{C}_{1} (n - 2) \int_{\Omega} \frac{\varphi_{\Omega} \lvert y \rvert^{3} \left( \partial_{s} \varphi_{\Omega} \cdot y \right)}{(\epsilon + \lvert y \rvert^{2})^{n - 1}} dy \\
& \qquad + \tilde{C}_{1} (n - 2)^{2} \int_{\Omega} \frac{\varphi_{\Omega}^{2} \lvert y \rvert^{5}}{(\epsilon + \lvert y \rvert^{2})^{n}} dy \\
& : = A_{2, 1} + A_{2, 2} + A_{2, 3}.
\end{align*}
The other two terms above can be estimated as
\begin{align*}
\frac{1}{a} \int_{\Omega} \left( S_{g}(0) + \beta \right) u_{\epsilon, \Omega}^{2} dx & = \frac{1}{a} \int_{\Omega} \left( S_{g}(0) + \beta \right) \frac{\varphi_{\Omega}^{2}}{(\epsilon + \lvert y \rvert^{2})^{n - 2}} dy =  \frac{1}{a} \int_{\R^{n}} \left( S_{g}(0) + \beta \right) \frac{1}{(\epsilon + \lvert y \rvert^{2})^{n - 2}} dy \\
& \qquad -  \frac{1}{a}\int_{\R^{n} \backslash \Omega} \left( S_{g}(0) + \beta \right) \frac{1}{(\epsilon + \lvert y \rvert^{2})^{n - 2}} dy \\
& \qquad \qquad +  \frac{1}{a} \left( S_{g}(0) + \beta \right) \int_{\Omega} \frac{\varphi_{\Omega}^{2} - 1}{(\epsilon + \lvert y \rvert^{2})^{n - 2}} dy \\
& : =  \frac{1}{a} \int_{\R^{n}} \left( S_{g}(0) + \beta \right) \frac{1}{(\epsilon + \lvert y \rvert^{2})^{n - 2}} dy + B_{0, 1} + B_{0, 2}; \\
 \lVert u_{\epsilon, \Omega} \rVert_{\mathcal{L}^{p}(\Omega)}^{p} & = \int_{\Omega} \frac{\varphi_{\Omega}^{p}}{(\epsilon + \lvert y \rvert^{2})^{n}} dy = \int_{\R^{n}} \frac{1}{{(\epsilon + \lvert y \rvert^{2})^{n}} dy} - \int_{\R^{n} \backslash \Omega} \frac{1}{(\epsilon + \lvert y \rvert^{2})^{n}} dy \\
 & \qquad  + \int_{\Omega} \frac{ \varphi_{\Omega}^{p} - 1}{(\epsilon + \lvert y \rvert^{2})^{n}} dy \\
 & : = \int_{\R^{n}} \frac{1}{{(\epsilon + \lvert y \rvert^{2})^{n}} dy} + C_{0, 1} + C_{0, 2}.
\end{align*}
\medskip

\noindent {\bf Case I: $ n \geqslant 5 $.} Fix $ r $ to be small enough, independent of the choice of $ \epsilon $. Claim that
\begin{equation}\label{APP:eqn4}
\begin{split}
& \lVert \partial_{s} u_{\epsilon, \Omega} \rVert_{\calL^{2}(\Omega)}^{2} +  \frac{1}{a} \int_{\Omega} S_{g}(0) u_{\epsilon, \Omega}^{2} dx + A_{1} + A_{2} + B_{1} + B_{2} + B_{3} \\
& \qquad < (n - 2)^{2} \int_{\R^{n}} \frac{\lvert y \rvert^{2}}{(\epsilon + \lvert y \rvert^{2})^{n}} dy + \frac{(n -2)\beta}{4(n - 1)} \epsilon^{-\frac{4 - n}{2}} \int_{\R^{n}} \frac{1}{(1 + \lvert y \rvert^{2})^{n - 2}} dy + O\left(\epsilon^{\frac{5 - n}{2}} \right) \\
& \qquad \qquad - \epsilon^{\frac{4 - n}{2}} \frac{S_{g}(0) (n - 2)}{6} \frac{n-2}{n} \int_{B_{P}\left(\frac{r}{2} \epsilon^{-\frac{1}{2}} \right)} \frac{\lvert y \rvert^{2}}{( 1 + \lvert y \rvert^{2})^{n}} dy \\
& =  (n - 2)^{2} \int_{\R^{n}} \frac{\lvert y \rvert^{2}}{(\epsilon + \lvert y \rvert^{2})^{n}} dy + \frac{(n - 2)\beta}{4(n - 1)} \epsilon^{-\frac{4 - n}{2}} \int_{\R^{n}} \frac{1}{(1 + \lvert y \rvert^{2})^{n - 2}} dy + O\left(\epsilon^{\frac{5 - n}{2}} \right) + \Sigma_{2}; \\
& \lVert u_{\epsilon, \Omega} \rVert_{\mathcal{L}^{p}(\Omega)}^{p} + C_{1} + C_{2}  = \int_{\R^{n}} \frac{1}{(\epsilon + \lvert y \rvert^{2})^{n}} dy + O\left(\epsilon^{\frac{3 - n}{2}} \right) +  \Sigma_{1}.
\end{split}
\end{equation}
The term $ \Sigma_{1} $ in the last equation will be given later. Note $ \beta < 0 $ is some fixed constant, independent of $ \epsilon $.

Let's estimate the rest $ B_{i}, C_{i}, A_{i, j}, B_{i,j}, C_{i,j} $ except the crucial term $ A_{1, 3} $. Let $ \Omega' = B_{0}\left( r \epsilon^{-\frac{1}{2}} \right) $. Since $ \varphi_{\Omega} \equiv 1 $ on $ B_{0}\left(\frac{r}{2} \right) $, $ \lvert \partial_{s} \varphi_{\Omega} \rvert \leqslant \frac{L}{r} $ for some fixed constant $ L $, we observe that
\begin{align*}
A_{0, 1} & = O(1), A_{0, 2} = O(1), A_{0,3} \leqslant 0, A_{0, 4} \leqslant 0; \\
A_{1, 1} & =  O(1), \lvert A_{1, 2} \rvert =  O(1), A_{2, 1} =  O(1), \lvert A_{2,2} \rvert =  O(1); \\  
A_{2, 3} & = \tilde{C}_{1} (n - 2)^{2} \int_{\Omega} \frac{\varphi_{\Omega}^{2} \lvert y \rvert^{5}}{(\epsilon + \lvert y \rvert^{2})^{n}} dy = \epsilon^{\frac{5 - n}{2}} \tilde{C}_{1} (n - 2)^{2} \int_{\Omega'} \frac{\varphi_{\Omega}^{2} \lvert y \rvert^{5}}{(1 + \lvert y \rvert^{2})^{n}} dy \leqslant L_{1} \epsilon^{\frac{5 - n}{2}} \\
\Rightarrow & A_{2, 3} = O\left(\epsilon^{\frac{5 - n}{2}} \right); \\
\lvert B_{1} \rvert & \leqslant \left\lvert \frac{1}{a} \int_{\Omega} \sum_{i, \lvert \alpha \rvert = 1} \partial^{\alpha} R_{ii}(0) y^{\alpha} \lvert u \rvert^{2} dy \right\rvert \leqslant L_{2} \int_{\Omega} \frac{\lvert y \rvert}{(\epsilon + \lvert y \rvert^{2})^{n - 2}} dy \\
& \leqslant L_{2} \epsilon^{\frac{5 - n}{2}} \int_{\Omega'} \frac{\lvert y \rvert}{(1 + \lvert y \rvert^{2})^{n - 2}} dy \Rightarrow B_{1} = O\left(\epsilon^{\frac{5 - n}{2}} \right); \\
B_{2} & = \tilde{C}_{2} \int_{\Omega} \frac{\varphi_{\Omega}^{2} \lvert y \rvert^{2}}{(\epsilon + \lvert y \rvert^{2})^{ n - 2}} dy = \epsilon^{\frac{6 - n}{2}} \tilde{C}_{2} \int_{\Omega'} \frac{\varphi_{\Omega}^{2} \lvert y \rvert^{2}}{(1 + \lvert y \rvert^{2})^{ n - 2}} dy \leqslant L_{3} \epsilon^{\frac{6 - n}{2}} \\
\Rightarrow & B_{2} = O\left(\epsilon^{\frac{6 - n}{2}} \right); B_{3} \leqslant O\left(\epsilon^{\frac{7 - n}{2}}\right); \\
B_{0, 1}  & = O(1), B_{0, 2} = O(1); C_{0, 1} = O(1), C_{0, 2} = O(1); \\
C_{1} + C_{2} & = - \frac{1}{6n} \int_{\Omega} S_{g}(0) \lvert y \rvert^{2} \lvert u \rvert^{p} dy + \tilde{C}_{4} \int_{\Omega} \lvert y \rvert^{3} \lvert  u \rvert^{p} dy \\
& = -\frac{S_{g}(0)}{6n} \int_{\Omega} \frac{\lvert y \rvert^{2} \varphi_{\Omega}^{p}}{(\epsilon + \lvert y \rvert^{2})^{n}} dy + O \left( \epsilon^{\frac{3 - n}{2}} \right) \\
& = -\epsilon^{\frac{2 - n}{2}} \frac{S_{g}(0)}{6n} \int_{B_{0}\left(\left( \frac{r}{2} \right) \epsilon^{-\frac{1}{2}} \right)} \frac{\lvert y \rvert^{2}}{(1 + \lvert y \rvert^{2})^{n}} dy + O\left(1 \right) + O \left( \epsilon^{\frac{3 - n}{2}} \right)  \\
& : = \Sigma_{1} + O \left( \epsilon^{\frac{3 - n}{2}} \right).
\end{align*}
Here $ L_{i}, i = 1, \dotso, 3 $ are constants independent of $ r, \epsilon $.

For the crucial terms $ A_{1, 3} $, we show that
\begin{equation}\label{APP:eqn5}
\begin{split}
& A_{1, 3} + \frac{1}{a} \int_{\R^{n}} \left( S_{g}(0) + \beta \right) \frac{1}{(\epsilon + \lvert y \rvert^{2})^{n - 2}} dy \\
& \qquad < \beta \epsilon^{\frac{4 - n}{2}} \int_{\R^{n}}\frac{1}{( 1 + \lvert y \rvert^{2})^{n - 2}} dy - \epsilon^{\frac{4 - n}{2}} \frac{S_{g}(0) (n - 2)}{6} \cdot \frac{n- 2}{n} \int_{B_{0}\left(\frac{r}{2} \epsilon^{-\frac{1}{2}} \right)} \frac{\lvert y \rvert^{2}}{( 1 + \lvert y \rvert^{2})^{n}} dy \\
& \qquad \qquad + O(1) \\
& \qquad : =  \frac{(n - 2)\beta}{4(n - 1)} \epsilon^{\frac{4 - n}{2}} \int_{\R^{n}}\frac{1}{( 1 + \lvert y \rvert^{2})^{n - 2}} dy + \Sigma_{2} + O(1).
\end{split}
\end{equation}
The term $ A_{1, 3} $ can be bounded as
\begin{align*}
A_{1, 3} & = -\frac{S_{g}(0) (n - 2)^{2}}{6n} \int_{\Omega} \frac{\varphi_{\Omega}^{2} \lvert y \rvert^{4} }{(\epsilon + \lvert y \rvert^{2})^{n}} dy = - \epsilon^{\frac{4 - n}{2}} \frac{S_{g}(0) (n - 2)^{2}}{6n} \int_{\Omega'} \frac{\varphi_{\Omega}^{2} \lvert y \rvert^{4} }{(1 + \lvert y \rvert^{2})^{n}} dy \\
& = - \epsilon^{\frac{4 - n}{2}} \frac{S_{g}(0) (n - 2)^{2}}{6n} \int_{\Omega'} \frac{\varphi_{\Omega}^{2} \lvert y \rvert^{2} \left( \lvert y \rvert^{2} - \frac{n}{n - 2} \right) }{(1 + \lvert y \rvert^{2})^{n}} dy - \epsilon^{\frac{4 - n}{2}} \frac{S_{g}(0) (n - 2)^{2}}{6n} \int_{\Omega'} \frac{\varphi_{\Omega}^{2} \lvert y \rvert^{2} \cdot \frac{n}{n - 2} }{(1 + \lvert y \rvert^{2})^{n}} dy \\
& \leqslant - \epsilon^{\frac{4 - n}{2}} \frac{S_{g}(0)}{12n} \int_{\R^{n}} \lvert y \rvert^{2} \cdot \frac{2(n - 2)^{2} \left( \lvert y \rvert^{2} - \frac{n}{n - 2} \right) }{(1 + \lvert y \rvert^{2})^{n}} dy \\
& \qquad - \epsilon^{\frac{4 - n}{2}} \frac{S_{g}(0) (n - 2)^{2}}{6n} \int_{\Omega'} \frac{\varphi_{\Omega}^{2} \lvert y \rvert^{2} \cdot \frac{n}{n - 2} }{(1 + \lvert y \rvert^{2})^{n}} dy  : = \Gamma_{1} + \Gamma_{2}.
\end{align*}
The above inequality holds since the radius of $ \Omega' $ is much larger than $ 1 $, which follows that $  \lvert y \rvert^{2} - \frac{n}{n - 2} > 0 $ for $ y \in \Omega'^{c} $. It is direct to check that for Euclidean Laplacian $ -\Delta $, the function $ \frac{1}{(1 + \lvert y \rvert^{2})^{n - 2}} $ satisfies
\begin{equation*}
\Delta \left(  \frac{1}{(1 + \lvert y \rvert^{2})^{n - 2}} \right) = \frac{2(n - 2)^{2} \left( \lvert y \rvert^{2} - \frac{n}{n - 2} \right) }{(1 + \lvert y \rvert^{2})^{n}}.
\end{equation*}
Using this, we estimate $ \Gamma_{1} $ as
\begin{align*}
\Gamma_{1} & = - \epsilon^{\frac{4 - n}{2}} \frac{S_{g}(0)}{12n} \int_{\R^{n}} \lvert y \rvert^{2} \cdot \frac{2(n - 2)^{2} \left( \lvert y \rvert^{2} - \frac{n}{n - 2} \right) }{(1 + \lvert y \rvert^{2})^{n}} dy \\
& = - \epsilon^{\frac{4 - n}{2}} \frac{S_{g}(0)}{12n} \int_{\R^{n}} \lvert y \rvert^{2} \cdot \Delta \left( \frac{1}{(1 + \lvert y \rvert^{2})^{n - 2}} \right) dy = - \epsilon^{\frac{4 - n}{2}} \frac{S_{g}(0)}{12n} \int_{\R^{n}} \left( \Delta \lvert y \rvert^{2} \right) \cdot  \frac{1}{(1 + \lvert y \rvert^{2})^{n - 2}} dy \\
& = - \epsilon^{\frac{4 - n}{2}} \frac{S_{g}(0)}{6} \int_{\R^{n}} \frac{1}{(1 + \lvert y \rvert^{2})^{n - 2}} dy.
\end{align*}
For $ \Gamma_{2} $, we see that
\begin{align*}
\Gamma_{2} & = - \epsilon^{\frac{4 - n}{2}} \frac{S_{g}(0) (n - 2)^{2}}{6n} \int_{\Omega'} \frac{\varphi_{\Omega}^{2} \lvert y \rvert^{2} \cdot \frac{n}{n - 2} }{(1 + \lvert y \rvert^{2})^{n}} dy \\
& = - \epsilon^{\frac{4 - n}{2}} \frac{S_{g}(0) (n - 2)}{6} \int_{B_{0}\left(\frac{r}{2} \epsilon^{-\frac{1}{2}} \right)} \frac{\lvert y \rvert^{2}}{( 1 + \lvert y \rvert^{2})^{n}} dy - \epsilon^{\frac{4 - n}{2}} \frac{S_{g}(0) (n - 2)}{6} \int_{\Omega' \backslash B_{0}\left(\frac{r}{2} \epsilon^{-\frac{1}{2}} \right)} \frac{\varphi_{\Omega}^{2} \lvert y \rvert^{2}}{( 1 + \lvert y \rvert^{2})^{n}} dy \\
& = - \epsilon^{\frac{4 - n}{2}} \frac{S_{g}(0) (n - 2)}{6} \cdot \frac{n-2}{n} \int_{B_{0}\left(\frac{r}{2} \epsilon^{-\frac{1}{2}} \right)} \frac{\lvert y \rvert^{2}}{( 1 + \lvert y \rvert^{2})^{n}} dy \\
& \qquad - \epsilon^{\frac{4 - n}{2}} \frac{S_{g}(0) (n - 2)}{6} \cdot \frac{2}{n} \int_{B_{0}\left(\frac{r}{2} \epsilon^{-\frac{1}{2}} \right)} \frac{\lvert y \rvert^{2}}{( 1 + \lvert y \rvert^{2})^{n}} dy+ O\left(1 \right) \\
& < - \epsilon^{\frac{4 - n}{2}} \frac{S_{g}(0) (n - 2)}{6} \cdot \frac{n-2}{n} \int_{B_{0}\left( \frac{r}{2} \epsilon^{-\frac{1}{2}} \right)} \frac{\lvert y \rvert^{2}}{( 1 + \lvert y \rvert^{2})^{n}} dy \\
& \qquad - \epsilon^{\frac{4 - n}{2}} \frac{S_{g}(0)}{6} \cdot \frac{2(n - 2)}{n} \int_{\R^{n}} \frac{\lvert y \rvert^{2}}{( 1 + \lvert y \rvert^{2})^{n}} dy + O\left(1 \right) \\
& : = \Sigma_{2} + \Gamma_{2, 1} + O\left(1 \right).
\end{align*}
Consider the integration $  \int_{\R^{n}} \frac{\lvert y \rvert^{2}}{( 1 + \lvert y \rvert^{2})^{n}} dy $ in $ \Gamma_{2, 1} $. Recall that $ \omega_{n} $ is the area of the unit $ n - 1 $-sphere. We check with change of variables $ s = \tan (\theta) $ in some middle step that
\begin{align*}
\int_{\R^{n}} \frac{\lvert y \rvert^{2}}{( 1 + \lvert y \rvert^{2})^{n}} dy & = \omega_{n} \int_{0}^{\infty} \frac{s^{n + 1}}{( 1 + s^{2})^{n}} ds = \omega_{n} \int_{0}^{\frac{\pi}{2}} \frac{\tan^{n + 1}(\theta) \sec^{2}(\theta)}{( 1 + \tan^{2}(\theta))^{n}} d\theta \\
& = \omega_{n} \int_{0}^{\frac{\pi}{2}} \sin^{2\left(\frac{n}{2} + 1 \right) - 1} \cos^{2\left(\frac{n}{2} - 1 \right) - 1} d\theta = \omega_{n} B\left(\frac{n}{2} + 1, \frac{n}{2} - 1 \right) \\
& = \omega_{n} \frac{\Gamma\left( \frac{n}{2} + 1 \right) \Gamma \left( \frac{n}{2} - 1 \right)}{\Gamma(n)}.
\end{align*}
Here $ B(a, b) $ is the Beta function with parameters $ a, b $ and $ \Gamma(a) $ is the Gamma function with parameter $ a $. The term below satisfies
\begin{equation*}
\frac{1}{a} \int_{\R^{n}} S_{g}(0) \frac{1}{(\epsilon + \lvert y \rvert^{2})^{n - 2}} dy = \epsilon^{\frac{4 - n}{2}} \frac{1}{a} \cdot S_{g}(0) \int_{\R^{n}} \frac{1}{(1+ \lvert y \rvert^{2})^{n - 2}} dy.
\end{equation*}
Check the integration $ \int_{\R^{n}} \frac{\lvert y \rvert^{2}}{( 1 + \lvert y \rvert^{2})^{n}} dy $ above, we have
\begin{equation*}
\int_{\R^{n}} \frac{\lvert y \rvert^{2}}{( 1 + \lvert y \rvert^{2})^{n}} dy = \omega_{n} B\left( \frac{n}{2}, \frac{n}{2} - 2 \right) =  \omega_{n} \frac{\Gamma\left( \frac{n}{2} \right) \Gamma \left( \frac{n}{2} - 2 \right)}{\Gamma(n - 2)}.
\end{equation*}
By the standard relation $ \Gamma(z + 1) = z \Gamma(z) $ we conclude that
\begin{equation}\label{APP:eqns5}
\int_{\R^{n}} \frac{\lvert y \rvert^{2}}{( 1 + \lvert y \rvert^{2})^{n}} dy = \frac{n(n - 4)}{4(n - 1)(n - 2)} \int_{\R^{n}} \frac{1}{( 1 + \lvert y \rvert^{2})^{n - 2}} dy.
\end{equation}
Note that $ \frac{1}{a} = \frac{n - 2}{4(n - 1)} $. Applying (\ref{APP:eqns5}) as well as the estimates of $ \Gamma_{1} $ and $ \Gamma_{2} $ above, we conclude that
\begin{align*}
& A_{1, 3} + \frac{1}{a} \left( S_{g}(0) + \beta \right) \int_{\R^{n}} \frac{1}{(\epsilon + \lvert y \rvert^{2})^{n - 2}} dy \leqslant \Gamma_{1} + \Gamma_{2} + \epsilon^{\frac{4 - n}{2}} \frac{1}{a} \cdot \left( S_{g}(0) + \beta \right) \int_{\R^{n}} \frac{1}{(1+ \lvert y \rvert^{2})^{n - 2}} dy \\
& \qquad = - \epsilon^{\frac{4 - n}{2}} \frac{S_{g}(0)}{6} \int_{\R^{n}} \frac{1}{(1 + \lvert y \rvert^{2})^{n - 2}} dy + \Sigma_{2} \\
& \qquad \qquad + \frac{n(n - 4)}{4(n - 1)(n - 2)} \cdot \left( - \epsilon^{\frac{4 - n}{2}} \frac{S_{g}(0)}{6} \cdot \frac{2(n - 2)}{n} \right) \int_{\R^{n}} \frac{1}{(1+ \lvert y \rvert^{2})^{n - 2}} dy \\
& \qquad \qquad \qquad + \epsilon^{\frac{4 - n}{2}} \frac{n - 2}{4(n - 1)} \cdot \left( S_{g}(0) + \beta \right) \int_{\R^{n}} \frac{1}{(1+ \lvert y \rvert^{2})^{n - 2}} dy + O(1) \\
& = \frac{(n - 2)\beta}{4(n - 1)} \epsilon^{\frac{4 - n}{2}} \int_{\R^{n}} \frac{1}{(1+ \lvert y \rvert^{2})^{n - 2}} dy + \Sigma_{2} + O(1).
\end{align*}
It follows that (\ref{APP:eqn5}) holds. Therefore we conclude that the estimates in (\ref{APP:eqn4}) hold by combining all estimates above together. 
Recall $ K_{1}, K_{2} $ defined above, we apply estimates in (\ref{APP:eqn4}), we get
\begin{align*}
Q_{\epsilon, \Omega} & \leqslant \frac{(n - 2)^{2} \int_{\R^{n}} \frac{\lvert y \rvert^{2}}{(\epsilon + \lvert y \rvert^{2})^{n}} dy + \frac{(n - 2)\beta}{4(n - 1)} \epsilon^{\frac{4 - n}{2}} \int_{\R^{n}} \frac{1}{(1+ \lvert y \rvert^{2})^{n - 2}} dy + O\left(\epsilon^{\frac{5 - n}{2}} \right) + \Sigma_{2}}{\left( \int_{\R^{n}} \frac{1}{(\epsilon + \lvert y \rvert^{2})^{n}} dy + O \left( \epsilon^{\frac{3 - n}{2}} \right) + \Sigma_{1} \right)^{\frac{2}{p}} } \\
& = \frac{(n - 2)^{2} \epsilon^{\frac{2 - n}{2}} \int_{\R^{n}} \frac{\lvert y \rvert^{2}}{(1 + \lvert y \rvert^{2})^{n}} dy + \frac{(n - 2)\beta}{4(n - 1)} \epsilon^{\frac{4 - n}{2}} \int_{\R^{n}} \frac{1}{(1+ \lvert y \rvert^{2})^{n - 2}} dy + O\left(\epsilon^{\frac{5 - n}{2}} \right) + \Sigma_{2} }{\left( \epsilon^{\frac{-n}{2}} \int_{\R^{n}} \frac{1}{(1 + \lvert y \rvert^{2})^{n}} dy + O \left( \epsilon^{\frac{3 - n}{2}} \right) + \Sigma_{1} \right)^{\frac{2}{p}}} \\
& = \frac{K_{1} + \epsilon \frac{(n - 2)\beta}{4(n - 1)} K_{3} + O\left(\epsilon^{\frac{3}{2}}\right) + \epsilon^{\frac{n - 2}{2}} \Sigma_{2}}{K_{2} + O \left( \epsilon^{\frac{3}{2}} \right) + \frac{n - 2}{n} \epsilon^{\frac{n}{2}} \Sigma_{1}K_{2}^{-\frac{2}{n - 2}}}.
\end{align*}
The last term above is due to the Taylor expansion of the function $ f(x) = x^{\frac{2}{p}} $. To control the last terms in numerator and denominator above, respectively, we first recall that the ratio $ \frac{K_{1}}{K_{2}} $ is the square of the reciprocal of the best Sobolev constant of
\begin{equation*}
\lVert u \rVert_{\calL_{p}(\R^{n})} \leqslant E \lVert Du \rVert_{\calL^{2}(\R^{n})} \Rightarrow \inf_{E} E^{-2} = T = \frac{K_{1}}{K_{2}}.
\end{equation*}
Aubin and Talenti \cite{Aubin} showed that
\begin{equation*}
T^{-\frac{1}{2}} = \pi^{-\frac{1}{2}} n^{-\frac{1}{2}} \left( \frac{1}{n - 2} \right)^{\frac{1}{2}} \left( \frac{\Gamma(n)}{\Gamma\left(\frac{n}{2} \right)} \right)^{\frac{1}{n}}.
\end{equation*}

Recall that $ \beta < 0 $, independent of $ \epsilon $, thus $ Q_{\epsilon, \Omega} $ can be estimated as
\begin{align*}
Q_{\epsilon, \Omega} & \leqslant \frac{K_{1} + \epsilon \frac{(n - 2)\beta}{4(n - 1)} K_{3} + O\left(\epsilon^{\frac{3}{2}}\right) + \epsilon^{\frac{n - 2}{2}} \Sigma_{2}}{K_{2} + O \left( \epsilon^{\frac{3}{2}} \right) + \frac{n - 2}{n} \epsilon^{\frac{n}{2}} \Sigma_{1}K_{2}^{-\frac{2}{n - 2}}} \\
& \leqslant  \frac{K_{1} + \epsilon^{\frac{n - 2}{2}} \Sigma_{2}}{K_{2} + \frac{n - 2}{n} \epsilon^{\frac{n}{2}} \Sigma_{1}K_{2}^{-\frac{2}{n - 2}}} + O\left(\epsilon^{\frac{3}{2}} \right) - \epsilon \cdot \frac{(n - 2)\lvert \beta \rvert}{4(n - 1)} \frac{K_{3}}{K_{2}}.
\end{align*}
The last inequality is due to the fact that both $ \epsilon^{\frac{n}{2}} \Sigma_{1}K_{2}^{-\frac{2}{n -2 }} $ and $ \epsilon^{\frac{n - 2}{2}} \Sigma_{2} $ are bounded above and below by constant multiples of $ \epsilon $. Lastly we must show that
\begin{equation}\label{APP:eqn6a}
\frac{K_{1} + \epsilon^{\frac{n - 2}{2}} \Sigma_{2}}{K_{2} + \frac{n - 2}{n} \epsilon^{\frac{n}{2}} \Sigma_{1}K_{2}^{-\frac{2}{n - 2}}} \leqslant \frac{K_{1}}{K_{2}}.
\end{equation}
By the best Sobolev constant, we observe that
\begin{equation*}
\frac{K_{1}}{K_{2}} = T \Rightarrow K_{1} = K_{2} T.
\end{equation*}
Thus (\ref{APP:eqn6a}) is equivalent to
\begin{align*}
& \frac{K_{1} + \epsilon^{\frac{n - 2}{2}} \Sigma_{2}}{K_{2} + \frac{n - 2}{n} \epsilon^{\frac{n}{2}} \Sigma_{1}K_{2}^{-\frac{2}{n}}} \leqslant \frac{K_{1}}{K_{2}} \Leftrightarrow \frac{K_{1} + \epsilon^{\frac{n - 2}{2}} \Sigma_{2}}{K_{2} + \frac{n - 2}{n} \epsilon^{\frac{n}{2}} \Sigma_{1}K_{2}^{-\frac{2}{n - 2}}} - \frac{K_{1}}{K_{2}} \leqslant 0 \\
& \frac{K_{2}^{2} T + \epsilon^{\frac{n - 2}{2}} \Sigma_{2} K_{2} - K_{2}^{2} T - \frac{n - 2}{n} K_{2} T \epsilon^{\frac{n}{2}} \Sigma_{1}K_{2}^{-\frac{2}{n - 2}}}{\left( K_{2} + \frac{n - 2}{n} \epsilon^{\frac{n}{2}} \Sigma_{1}K_{2}^{-\frac{2}{n - 2}} \right)K_{2} } \leqslant 0 \\
\Leftrightarrow & \epsilon^{\frac{n - 2}{2}} \Sigma_{2} K_{2}  - \frac{n - 2}{n} K_{2} T \epsilon^{\frac{n}{2}} \Sigma_{1}K_{2}^{-\frac{2}{n - 2}} \leqslant 0 \\
\Leftrightarrow & \epsilon^{\frac{n - 2}{2}} \left( - \epsilon^{\frac{4 - n}{2}} \frac{S_{g}(0) (n - 2)}{6} \frac{n - 2}{n} \int_{B_{P}\left(\frac{r}{2} \epsilon^{-\frac{1}{2}} \right)} \frac{\lvert y \rvert^{2}}{( 1 + \lvert y \rvert^{2})^{n}} dy \right) K_{2} \\
& \qquad  -  \frac{n - 2}{n} K_{2} T \epsilon^{\frac{n}{2}} \left( -\epsilon^{\frac{2 - n}{2}} \frac{S_{g}(0)}{6n} \int_{B_{P}\left(\frac{r}{2} \epsilon^{-\frac{1}{2}} \right)} \frac{\lvert y \rvert^{2}}{(1 + \lvert y \rvert^{2})^{n}} dy \right) K_{2}^{-\frac{2}{n - 2}} \leqslant 0 \\
\Leftrightarrow & (n - 2) - \frac{ T K_{2}^{-\frac{2}{n - 2}}}{n} \leqslant 0.
\end{align*}
Recall that the best Sobolev constant $ T $ is of the form
\begin{equation*}
T = \pi n (n - 2) \left( \frac{\Gamma \left( \frac{n}{2} \right)}{\Gamma(n)} \right)^{\frac{2}{n}}
\end{equation*}
For the term $ K_{2}^{-\frac{2}{n - 2}} $, we have
\begin{equation*}
K_{2}^{-\frac{2}{n - 2}} = \left( \int_{\R^{n}} \frac{1}{(1 + \lvert y \rvert^{2} )^{n}} dy \right)^{-\frac{2}{n}}  = \left( \omega_{n} \int_{0}^{\infty} \frac{r^{n - 1}}{(1 + r^{2})^{n}} dy \right)^{-\frac{2}{n}} : = \left( \omega_{n} D_{1} \right)^{-\frac{2}{n}} 
\end{equation*}
Here $ \omega_{n} $ is the surface area of the $ (n - 1) $-sphere. It is well known that
\begin{equation*}
\omega_{n} = n \cdot \frac{\pi^{\frac{n}{2}}}{\Gamma\left( \frac{n}{2} + 1 \right)} = n \cdot \frac{2\pi^{\frac{n}{2}}}{n\Gamma\left(\frac{n}{2} \right)} = \frac{2\pi^{\frac{n}{2}}}{\Gamma\left(\frac{n}{2} \right)} \Rightarrow \omega_{n}^{-\frac{2}{n}} = \left( \frac{\Gamma \left( \frac{n}{2} \right)}{2\pi^{\frac{n}{2}}} \right)^{\frac{2}{n}} = \left( \Gamma \left(\frac{n}{2} \right) \right)^{\frac{2}{n}} 2^{-\frac{2}{n}} \pi^{-1}.
\end{equation*}
For $ D_{1} $, we use change of variables $ r = \tan \theta $, $ \theta \in [0, \frac{\pi}{2}] $, we have
\begin{align*}
D_{1} & = \int_{0}^{\infty} \frac{r^{n - 1}}{(1 + r^{2})^{n}} dy = \int_{0}^{\frac{\pi}{2}} \frac{\tan^{n - 1}(\theta)}{( 1 + \tan^{2}(\theta))^{n}} \sec^{2}(\theta) d\theta =  \int_{0}^{\frac{\pi}{2}} \frac{\sin^{n - 1}(\theta) \sec^{n + 1}(\theta)}{\sec^{2n}(\theta)} d\theta \\
& = \int_{0}^{\frac{\pi}{2}} \sin^{n - 1}(\theta) \cos^{n - 1}\theta d\theta = 2^{1 - n}  \int_{0}^{\frac{\pi}{2}} \sin^{n - 1}(2\theta) d\theta = 2^{- n} \int_{0}^{\pi} \sin^{n-1}(x) dx.
\end{align*}
We use $ 2\theta = x $ in the last equality. It is well-known that the last integration above is symmetric with respect to $ x = \frac{\pi}{2} $ and is connected to the Gamma function. To be precise,
\begin{align*}
D_{1} & = 2^{- n} \int_{0}^{\pi} \sin^{n-1}(x) dx = 2^{1 - n} \int_{0}^{\frac{\pi}{2}} \sin^{n-1}(x) dx = 2^{1 - n} \frac{\sqrt{\pi} \Gamma \left( \frac{n -1 + 1}{2} \right)}{2 \Gamma \left(\frac{n - 1}{2} + 1 \right)} \\
\Rightarrow D_{1}^{-\frac{2}{n}} & = \left( 2^{n} \pi^{-\frac{1}{2}} \frac{\Gamma \left( \frac{n}{2} + \frac{1}{2} \right)}{ \Gamma \left(\frac{n}{2} \right)} \right)^{\frac{2}{n}} = 2^{2} \pi^{-\frac{1}{n}} \Gamma \left( \frac{n}{2} + \frac{1}{2} \right)^{\frac{2}{n}} \Gamma \left(\frac{n}{2} \right)^{-\frac{2}{n}}.
\end{align*}
Packing terms together, we have
\begin{align*}
\frac{ T K_{2}^{-\frac{2}{n - 2}}}{n} & =  n^{-1} \pi n (n - 2) \left( \frac{\Gamma \left( \frac{n}{2} \right)}{\Gamma(n)} \right)^{\frac{2}{n}} \cdot \left( \Gamma \left(\frac{n}{2} \right) \right)^{\frac{2}{n}} 2^{-\frac{2}{n}} \pi^{-1} \\
& \qquad \cdot 2^{2} \pi^{-\frac{1}{n}} \left( \Gamma \left( \frac{n}{2} + \frac{1}{2} \right) \right)^{\frac{2}{n}} \left( \Gamma \left(\frac{n}{2} \right) \right)^{-\frac{2}{n}} \\
& =  (n - 2) 2^{2 - \frac{2}{n}} \pi^{-\frac{1}{n}} \left( \Gamma \left( \frac{n}{2} + \frac{1}{2} \right)  \Gamma \left(\frac{n}{2} \right)  \right)^{\frac{2}{n}} \left( \Gamma(n) \right)^{-\frac{2}{n}}
\end{align*}
The Legendre duplication formula for Gamma functions says
\begin{equation*}
\Gamma \left(z+ \frac{1}{2} \right) \Gamma \left( z \right) = 2^{1 - 2z} \sqrt{\pi} \Gamma(2z).
\end{equation*}
Take $ z = \frac{n}{2} $ above, we have
\begin{align*}
\frac{ T K_{2}^{-\frac{2}{n - 2}}}{n} & =  (n - 2) 2^{2 - \frac{2}{n}} \pi^{-\frac{1}{n}} \left( \Gamma \left( \frac{n}{2} + \frac{1}{2} \right)  \Gamma \left(\frac{n}{2} \right)  \right)^{\frac{2}{n}} \left( \Gamma(n) \right)^{-\frac{2}{n}} \\
& =  (n - 2) 2^{2 - \frac{2}{n}} \pi^{-\frac{1}{n}} \left( 2^{1 - 2 \frac{n}{2}} \sqrt{\pi} \Gamma(n) \right)^{\frac{2}{n}} \left( \Gamma(n) \right)^{-\frac{2}{n}} \\
& =  (n - 2).
\end{align*}
Hence we conclude that
\begin{equation*}
(n - 2) - \frac{ T K_{2}^{-\frac{2}{n - 2}}}{n} = (n - 2) -  (n - 2) = 0 \Rightarrow \frac{K_{1} + \epsilon^{\frac{n - 2}{2}} \Sigma_{2}}{K_{2} + 
\frac{n - 2}{n}\epsilon^{\frac{n}{2}} \Sigma_{1}K_{2}^{-\frac{2}{n - 2}}} \leqslant \frac{K_{1}}{K_{2}}.
\end{equation*}
Therefore (\ref{APP:eqn6a}) holds. It follows that when $ \epsilon $ is small enough,
\begin{equation*}
Q_{\epsilon, \Omega} \leqslant \frac{K_{1}}{K_{2}} + O\left(\epsilon^{\frac{3}{2}} \right) - \epsilon \cdot \frac{(n - 2)\lvert \beta \rvert}{4(n - 1)} \frac{K_{3}}{K_{2}} < \frac{K_{1}}{K_{2}} = T.
\end{equation*}
Thus the inequality in (\ref{APP:eqn1}) holds when $ n \geqslant 5 $.
\medskip

\noindent {\bf Case II: $ n = 4 $.} In this case, we keep notations to be the same as $ n \geqslant 5 $ cases. Again we fix $ r $. Most estimates of $ A_{i,j}, B_{i}, B_{i,j}, C_{i}, C_{i, j} $ with $ n = 4 $ are the same as above:
\begin{align*}
A_{0, 1} & = O(1), A_{0, 2} = O(1), A_{0,3} \leqslant 0, A_{0, 4} \leqslant 0; \\
A_{1, 1} & =  O(1), \lvert A_{1, 2} \rvert =  O(1), A_{2, 1} =  O(1), \lvert A_{2,2} \rvert =  O(1); A_{2, 3} = O\left(\epsilon^{\frac{1}{2}} \right); \\
\lvert B_{1} \rvert & = O\left(\epsilon^{\frac{1}{2}} \right); B_{2} = O\left(\epsilon \right); B_{3} = O\left(\epsilon^{\frac{3}{2}}\right); \\
C_{0, 1} & = O(1), C_{0, 2} = O(1); C_{1} \geqslant 0, C_{2} = O\left(\epsilon^{-\frac{1}{2}} \right).
\end{align*}
Here $ C_{1} > 0 $ due to the fact that $ S_{g}(0) < 0 $. We do not need $ B_{0, 1} $ and $ B_{0, 2} $ since the key term
\begin{equation*}
\frac{1}{a} \cdot \left(S_{g}(0) + \beta \right) \int_{\Omega} \frac{\varphi_{\Omega}^{2}}{(\epsilon + \lvert y \rvert)^{2})^{2}} dy
\end{equation*}
is not integrable. Note that the term $ A_{1, 3} $ with $ n = 4 $ can be estimated as
\begin{align*}
A_{1, 3} & = -\frac{S_{g}(0) (n - 2)^{2}}{6n} \int_{\Omega} \frac{\varphi_{\Omega}^{2} \lvert y \rvert^{4} }{(\epsilon + \lvert y \rvert^{2})^{n}} dy = -\frac{S_{g}(0)}{6} \int_{\Omega} \frac{\varphi_{\Omega}^{2} \lvert y \rvert^{4} }{(\epsilon + \lvert y \rvert^{2})^{4}} dy \\
& \leqslant -\frac{S_{g}(0)}{6} \int_{\Omega} \frac{\varphi_{\Omega}^{2} }{(\epsilon + \lvert y \rvert^{2})^{2}} dy
\end{align*}
since the ratio $ \frac{\lvert y \rvert^{2}}{\epsilon + \lvert y \rvert^{2}} < 1 $ as always. Thus with $ a = \frac{n - 2}{4(n - 1)} = \frac{1}{6} $ when $ n = 4 $, we can see
\begin{equation*}
A_{1, 3} + \frac{1}{a} \cdot \left(S_{g}(0) + \beta \right) \int_{\Omega} \frac{\varphi_{\Omega}^{2}}{(\epsilon + \lvert y \rvert)^{2})^{2}} dy \leqslant \frac{\beta}{6}  \int_{\Omega} \frac{\varphi_{\Omega}^{2}}{(\epsilon + \lvert y \rvert)^{2})^{2}} dy.
\end{equation*}
Applying exactly the same calculation as in \cite[Lemma.~1.1]{Niren3}, we conclude that
\begin{equation*}
\int_{\Omega} \frac{\varphi_{\Omega}^{2}}{(\epsilon + \lvert y \rvert)^{2})^{2}} dy \geqslant L_{5} \lvert \log \epsilon \rvert
\end{equation*}
with some positive constant $ L_{5} $ independent of $ \epsilon $. It follows that
\begin{equation}\label{APP:eqn7}
A_{1, 3} + \frac{1}{a} \cdot \left(S_{g}(0) + \beta \right) \int_{\Omega} \frac{\varphi_{\Omega}^{2}}{(\epsilon + \lvert y \rvert)^{2})^{2}} dy \leqslant -L_{5} \lvert \beta \rvert \lvert \log \epsilon \rvert.
\end{equation}
Applying (\ref{APP:eqn7}) and all other estimates above, we have
\begin{align*}
Q_{\epsilon, \Omega} & \leqslant \frac{4 \int_{\R^{n}} \frac{\lvert y \rvert^{2}}{(\epsilon + \lvert y \rvert^{2})^{4}} dy -L_{5} \lvert \beta \rvert \lvert \log \epsilon \rvert  + O(1) + O\left(\epsilon^{\frac{1}{2}} \right)}{\left(\int_{\R^{n}} \frac{1}{(\epsilon + \lvert y \rvert^{2})^{4}} dy + O(1) + O\left(\epsilon^{-\frac{1}{2}} \right) \right)^{\frac{1}{2}}} \\
& \leqslant \frac{4\epsilon^{-1} \int_{\R^{n}} \frac{\lvert y \rvert^{2}}{(1 + \lvert y \rvert^{2})^{2}} dy -L_{5} \lvert \beta \rvert \lvert \log \epsilon \rvert + O\left(1 \right)}{\left( \epsilon^{-2} \int_{\R^{n}} \frac{1}{(1 + \lvert y \rvert^{2})^{n}} dy + O\left(\epsilon^{-\frac{1}{2}} \right) \right)^{\frac{1}{2}}} \\
& \leqslant \frac{K_{1} -L_{5} \lvert \beta \rvert \epsilon \lvert \log \epsilon \rvert + O\left(\epsilon^{\frac{3}{2}} \right)}{K_{2} + O\left(\epsilon^{\frac{3}{2}} \right)} \\
& = \frac{K_{1}}{K_{2}} + O\left(\epsilon \right) - \frac{L_{5} \lvert \beta \rvert}{K_{2}} \epsilon \lvert \log \epsilon \rvert.
\end{align*}
It follows that when $ \epsilon $ small enough, we have
\begin{equation*}
Q_{\epsilon, \Omega} < T.
\end{equation*}
Thus (\ref{APP:eqn1}) holds for $ n = 4 $.
\medskip

\noindent {\bf Case III : $ n = 3 $.} This case is inspired by the argument in Brezis and Nirenberg \cite[Lemma~1.3]{Niren3}, but not the same. We have to do some conformal change so that
\begin{equation}\label{APP:eqn8}
\frac{S_{g}(0)}{8} + \frac{\pi^{2}}{4} = 0 \; {\rm on} \; B_{0}(1).
\end{equation}
This can be done due to Theorem \ref{manifold:thm4} and Remark \ref{manifold:renew}. Thus
\begin{equation*}
\frac{S_{g}(0) + \beta}{8} + \frac{\pi^{2}}{4} < 0 \; {\rm on} \; B_{0}(1)
\end{equation*}
since $ \beta < 0 $. Note that the first eigenvalue $ \lambda_{1} $ of $ -\Delta_{g} $ is almost the same as the first eigenvalue of Euclidean Laplacian when the geodesic ball is small enough. Since $ S_{g} \lambda_{1}^{-1} $ is scaling invariant, we can scale the metric, i.e. scale the radius of the ball without changing the relation in (\ref{APP:eqn8}) as well as the relation in (\ref{local:eqns1}) with a small enough $ \beta $. Thus it suffices to set $ \Omega = B_{0}(1) $. We would like to point out that the choice of $ S_{g}(0) $ in (\ref{APP:eqn8}) is compatible with the choice of $ S_{g}(0)$ in (\ref{local:eqns1}). We estimate $ A_{1} $ first. Recall that
\begin{align*}
A_{1} & = -\frac{S_{g}(0)}{18} \int_{\Omega} \frac{\lvert y \rvert^{2} \left\lvert \partial_{s} \varphi_{\Omega} \right\rvert^{2}}{\epsilon + \lvert y \rvert^{2}} dy + \frac{S_{g}(0)}{9} \int_{\Omega} \frac{\varphi_{\Omega} \lvert y \rvert^{2} \left( \partial_{s} \varphi_{\Omega} \cdot y \right)}{(\epsilon + \lvert y \rvert^{2})^{2}} dy -\frac{S_{g}(0)}{18} \int_{\Omega} \frac{\varphi_{\Omega}^{2} \lvert y \rvert^{4}}{(\epsilon + \lvert y \rvert^{2})^{3}} dy \\
& : = A_{1, 1} + A_{1, 2} + A_{1, 3}.
\end{align*}
For $ A_{1, 1} $, we have
\begin{align*}
A_{1,1} & = -\frac{S_{g}(0)}{18} \int_{\Omega} \frac{\lvert y \rvert^{2} \left\lvert \partial_{s} \varphi_{\Omega} \right\rvert^{2}}{\epsilon + \lvert y \rvert^{2}} dy =\frac{\pi^{2}}{9} \omega_{3} \cdot \frac{\pi^{2}}{4} \int_{0}^{1} \frac{\sin^{2} \left(\frac{\pi s}{2} \right) s^{4}}{\epsilon + s^{2}} ds \leqslant \frac{\pi^{2}}{9} \omega_{3} \cdot \frac{\pi^{2}}{4} \int_{0}^{1}  \sin^{2} \left(\frac{\pi s}{2} \right) s^{2} ds \\
& = \omega_{3} \cdot \frac{\pi^{4}}{36} \left( \frac{1}{6} + \frac{1}{\pi^{2}} \right).
\end{align*}
For $ A_{1, 2} $, we have
\begin{align*}
A_{1, 2} & =  \frac{S_{g}(0)}{9} \int_{\Omega} \frac{\varphi_{\Omega} \lvert y \rvert^{2} \left( \partial_{s} \varphi_{\Omega} \cdot y \right)}{(\epsilon + \lvert y \rvert^{2})^{2}} dy = \frac{2\pi^{2}}{9} \omega_{3} \frac{\pi}{4} \int_{0}^{1} \frac{\sin (\pi s) s^{5}}{(\epsilon + s^{2})^{2}} ds \leqslant \frac{2\pi^{2}}{9} \omega_{3} \frac{\pi}{4} \int_{0}^{1} \sin (\pi s) s ds \\
& = \omega_{3} \frac{\pi^{3}}{18} \cdot \frac{1}{\pi} =  \omega_{3} \frac{\pi^{2}}{18}.
\end{align*}
For $ A_{1,3} $, we have
\begin{align*}
A_{1, 3} & = -\frac{S_{g}(0)}{18} \int_{\Omega} \frac{\varphi_{\Omega}^{2} \lvert y \rvert^{4}}{(\epsilon + \lvert y \rvert^{2})^{3}} dy = \frac{\pi^{2}}{9} \omega_{3} \int_{0}^{1} \frac{\cos^{2} \left(\frac{\pi s}{2} \right) s^{6}}{(\epsilon + s^{2})^{3}} ds \leqslant  \frac{\pi^{2}}{9} \omega_{3} \int_{0}^{1} \cos^{2} \left(\frac{\pi s}{2} \right) ds \\
& = \omega_{3} \frac{\pi^{2}}{18}.
\end{align*}
Here $ L_{6} > 0 $ is some constant independent of $ \epsilon, r $. For $ A_{2}, C_{1}, C_{2}, C_{0, 1}, C_{0, 2} $, a simple estimate shows that
\begin{equation*}
A_{2} = O\left(\epsilon^{\frac{1}{2}}\right), C_{1} > 0, C_{2} = O(1), C_{0, 1} = O(1), C_{0, 2} = O(1).
\end{equation*}
Based on Brezis and Nirenberg \cite{Niren3}, the term $ \lVert \partial_{s} u \rVert_{\calL^{2}(\Omega)}^{2} $ can be written as
\begin{equation*}
\lVert \partial_{s} u \rVert_{\calL^{2}(\Omega)}^{2} = \omega_{3} \frac{\pi^{2}}{4} \int_{0}^{1} \frac{ \sin^{2} \left( \frac{\pi s}{2} \right) s^{2}}{\epsilon + s^{2}} ds + 3\omega_{3} \epsilon \int_{0}^{1} \frac{ \cos^{2} \left( \frac{\pi s}{2} \right) s^{2}}{(\epsilon + s^{2})^{3}} ds : = E_{1} + E_{2}.
\end{equation*}
Due to Brezis and Nirenberg again, we have
\begin{equation*}
E_{1} = \omega_{3} \frac{\pi^{2}}{4} \int_{0}^{1} \sin^{2} \left( \frac{\pi s}{2} \right) ds + O(\epsilon).
\end{equation*}
For the term $ E_{2} $, we estimate as
\begin{align*}
E_{2} & = 3\omega_{3} \epsilon \int_{0}^{1} \frac{ \cos^{2} \left( \frac{\pi s}{2} \right) s^{2}}{(\epsilon + s^{2})^{3}} ds \\
& = 3\omega_{3} \epsilon \int_{0}^{\infty} \frac{s^{2}}{(\epsilon + s^{2})^{3}} ds - 3\omega_{3} \epsilon \int_{1}^{\infty} \frac{s^{2}}{(\epsilon + s^{2})^{3}} ds - 3 \omega_{3} \epsilon \int_{0}^{1} \frac{ \sin^{2} \left( \frac{\pi s}{2} \right) s^{2}}{(\epsilon + s^{2})^{3}} ds \\
& : = 3\omega_{3} \epsilon \int_{0}^{\infty} \frac{s^{2}}{(\epsilon + s^{2})^{3}} ds + E_{2, 1} + E_{2, 2}.
\end{align*}
Using polar coordinates $ s = \epsilon^{\frac{1}{2}} \tan(\theta) $, the term $ E_{2, 1} $ can be computed as
\begin{align*}
E_{2, 1} & = - 3\omega_{3} \epsilon \int_{1}^{\infty} \frac{s^{2}}{(\epsilon + s^{2})^{3}} ds = -3 \omega_{3} \epsilon^{-\frac{1}{2}} \int_{\arctan \left(\epsilon^{-\frac{1}{2}} \right)}^{\frac{\pi}{2}} \frac{\tan^{2}(\theta) \sec^{2}(\theta)}{\sec^{6}(\theta)} d\theta \\
& = -3 \omega_{3} \epsilon^{-\frac{1}{2}} \int_{\arctan \left(\epsilon^{-\frac{1}{2}} \right)}^{\frac{\pi}{2}} \frac{1}{4} d\theta + o(1) \\
& = -\frac{3}{4} \omega_{3} \epsilon^{-\frac{1}{2}} \left(\frac{\pi}{2} - \arctan \left(\epsilon^{-\frac{1}{2}} \right) \right) + o(1) = -\frac{3}{4} \omega_{3} + o(1).
\end{align*}
The last inequality is due to the Taylor expansion of $ \arctan(x) $ at $ x = +\infty $, which is
\begin{equation*}
\arctan(x) \bigg|_{x = +\infty} = \frac{\pi}{2} - \frac{1}{x} + \frac{1}{3x^{3}} + o\left(\frac{1}{x^{3}} \right).
\end{equation*}
Again by Brezis and Nirenberg, we know that
\begin{equation}\label{APP:eqn9}
\frac{3\pi}{16} = \int_{0}^{\infty} \frac{s^{4}}{(1 + s^{2})^{3}} ds = 3\int_{0}^{\infty} \frac{s^{2}}{(1 + s^{2})^{3}} ds.
\end{equation}
Thus with (\ref{APP:eqn9}), the Taylor expansion of $ \sin(x) $ at origin, we have
\begin{align*}
E_{2, 2} & = - 3 \omega_{3} \epsilon \int_{0}^{1} \frac{ \sin^{2} \left( \frac{\pi s}{2} \right) s^{2}}{(\epsilon + s^{2})^{3}} ds = -3\omega_{3} \epsilon \frac{\pi^{2}}{4} \int_{0}^{1} \frac{s^{4}}{(\epsilon + s^{2})^{3}} ds + O(\epsilon) \\
& \geqslant - \frac{9\pi^{3}}{64} \omega_{3} + O(\epsilon).
\end{align*}
By a direct calculation, it follows that
\begin{equation*}
A_{1} + E_{2, 1} + E_{2, 2} \leqslant O(\epsilon).
\end{equation*}
By (\ref{APP:eqn9}), we can see that
\begin{equation*}
3\omega_{3} \epsilon \int_{0}^{\infty} \frac{s^{2}}{(\epsilon + s^{2})^{3}} ds = 3\omega_{3} \epsilon^{-\frac{1}{2}} \int_{0}^{\infty} \frac{s^{2}}{(1 + s^{2})^{3}} ds = \omega_{3} \epsilon^{-\frac{1}{2}} \int_{0}^{\infty} \frac{s^{4}}{(1 + s^{2})^{3}} ds = \epsilon^{-\frac{1}{2}} K_{1}.
\end{equation*}
Exactly following Brezis and Nirenberg \cite[Lemma~1.3]{Niren3}, we have
\begin{align*}
\lVert u \rVert_{\calL^{6}(\Omega)}^{2} & = \epsilon^{-\frac{1}{2}} K_{2} + O\left(\epsilon^{\frac{1}{2}} \right); \\
\frac{S_{g}(0) + \beta }{8} \int_{\Omega} \frac{\varphi_{\Omega}^{2}(y)}{\epsilon + \lvert y \rvert^{2}} dy & = \frac{S_{g}(0) + \beta }{8} \omega_{3} \int_{0}^{1} \cos^{2} \left(\frac{\pi s}{2} \right) ds + O\left(\epsilon^{\frac{1}{2}} \right). 
\end{align*}
Lastly, note that
\begin{equation*}
\int_{0}^{1} \cos^{2} \left(\frac{\pi s}{2} \right) ds = \int_{0}^{1} \sin^{2} \left(\frac{\pi s}{2} \right) ds.
\end{equation*}
Thus the quantity $ Q_{\epsilon, \Omega} $ can be estimated as
\begin{align*}
Q_{\epsilon, \Omega} & \leqslant \frac{\epsilon^{-\frac{1}{2}} K_{1} + \omega_{3}\int_{0}^{1} \cos^{2} \left(\frac{\pi s}{2} \right) ds \left(\frac{\pi^{2}}{4} + \frac{S_{g}(0) + \beta }{8} \right) +  O\left(\epsilon^{\frac{1}{2}} \right) }{\epsilon^{-\frac{1}{2}} K_{2}  + O\left(\epsilon^{\frac{1}{2}} \right)} \\
& \leqslant \frac{K_{1}}{K_{2}} + \frac{\epsilon^{\frac{1}{2}}}{K_{2}} \omega_{3}\int_{0}^{1} \cos^{2} \left(\frac{\pi s}{2} \right) ds \left(\frac{\pi^{2}}{4} + \frac{S_{g}(0) + \beta}{8} \right) + O(\epsilon).
\end{align*}
By (\ref{APP:eqn8}), we know that $ \frac{\pi^{2}}{4} + \frac{S_{g}(0) + \beta}{8} < 0 $ hence
\begin{equation*}
Q_{\epsilon, \Omega} < T
\end{equation*}
When $ \epsilon $ is small enough. All cases are now proven.
\end{proof}

\bibliographystyle{plain}
\bibliography{Yamabes}

\end{document}